\documentclass[reqno]{amsart}
\usepackage[scale=0.75, centering, headheight=14pt]{geometry}
\pdfoutput=1
\usepackage[latin1]{inputenc}
\usepackage[T1]{fontenc}
\usepackage{lmodern}
\usepackage[english]{babel}
\usepackage{amsmath}

\usepackage{tikz}
\usetikzlibrary{positioning}

\usepackage{amsmath,amssymb,amsfonts,amsthm}
\usepackage{mathtools,accents}
\usepackage{mathrsfs}
\usepackage{xfrac}
\usepackage{array}
\usepackage{bbm}

\usepackage{aliascnt}

\usepackage{booktabs} 
\usepackage{array}    
\usepackage{verbatim} 
\usepackage{subcaption} 

\usepackage{mathrsfs, dsfont}
\usepackage{amssymb}
\usepackage{amsthm}
\usepackage{amsmath,amsfonts,amssymb,esint}
\usepackage{graphics,color}
\usepackage{graphicx}
\usepackage{enumerate}
\usepackage{mathtools,centernot}
\usepackage[normalem]{ulem}

\usepackage{microtype}
\usepackage{paralist} 
\usepackage{cases}
\usepackage[initials, alphabetic]{amsrefs}
\allowdisplaybreaks

\usepackage{braket}
\usepackage{bm}

\usepackage[citecolor=blue,colorlinks]{hyperref}
\addto\extrasenglish{}

\usepackage{enumerate}
\usepackage{xcolor}
\usepackage{aliascnt}

\makeatletter
\def\newaliasedtheorem#1[#2]#3{
  \newaliascnt{#1@alt}{#2}
  \newtheorem{#1}[#1@alt]{#3}
  \expandafter\newcommand\csname #1@altname\endcsname{#3}
}
\makeatother

\usepackage{indentfirst}

\usepackage{esint}
\usepackage{mathrsfs}
\usepackage{stmaryrd}

\usepackage{pgfplots}
\usepgfplotslibrary{fillbetween}
\usetikzlibrary{fadings}
\tikzfading[name=myfading, bottom color=transparent!50, top color=transparent!50]

\makeatletter
\newsavebox{\measure@tikzpicture}

\newcommand{\setword}[2]{%
  \phantomsection
  #1\def\@currentlabel{\unexpanded{#1}}\label{#2}%
}

\renewcommand\labelenumi{(\roman{enumi})}
\renewcommand\theenumi\labelenumi

\newtheorem{theorem}{\bf Theorem}[section]
\newtheorem{remark}[theorem]{\bf Remark}
\newtheorem{definition}[theorem]{\bf Definition}
\newtheorem{lemma}[theorem]{\bf Lemma}
\newtheorem{proposition}[theorem]{\bf Proposition}

\newtheorem*{theorem*}{Theorem}

\newcommand{\e}{{\rm e}}
\newcommand{\eps}{\varepsilon}
\newcommand{\R}{\mathbb R}
\newcommand{\N}{\mathbb N}

\newcommand{\Z}{\mathbb Z}

\newcommand{\T}{\mathbb T}
\newcommand{\Expect}{\mathbb E}
\newcommand{\Prob}{\mathbb P}

\newcommand{\Leb}{\mathcal{L}}

\newcommand{\divergence}{\operatorname{div}}

\newcommand{\Haus}{\mathcal{H}}

\DeclareMathOperator{\supp}{supp}

\DeclareMathOperator{\dist}{dist}

\DeclareMathOperator{\initial}{in}

\DeclareMathOperator{\diver}{div}

\DeclareMathOperator{\exit}{exit}
\DeclareMathOperator{\mi}{mid}

\allowdisplaybreaks
\numberwithin{equation}{section}

\title[Anomalous dissipation via spontaneous stochasticity]{Anomalous dissipation via spontaneous stochasticity with a two-dimensional autonomous velocity field}

\author[Carl J. P. Johansson]{Carl Johan Peter Johansson}
\address{Institute of Mathematics, EPFL, Station 8, 1015 Lausanne, Switzerland}
\email{carl.johansson@epfl.ch}

\author{Massimo Sorella}
\address{Institute of Mathematics, EPFL, Station 8, 1015 Lausanne, Switzerland}
\email{massimo.sorella@epfl.ch}
\address{Department of mathematics, Imperial, London, SW7 2AZ, UK}
\email{m.sorella@imperial.ac.uk}

\thanks{\copyright Duke University Press}

\begin{document}
\maketitle
\begin{abstract}
We study anomalous dissipation in the context of passive scalars and we   construct a two-dimensional autonomous divergence-free velocity field in $C^\alpha$ (with $\alpha \in (0,1)$ arbitrary but fixed) which exhibits anomalous dissipation. Our proof employs the fluctuation-dissipation formula, which links spontaneous stochasticity with anomalous dissipation. Therefore, we address the issue of anomalous dissipation by showing that the variance of stochastic trajectories, in the zero noise limit, remains positive.
Based on this result, we answer \cite[Question 2.2 and Question 2.3]{BDL22} regarding anomalous dissipation for the forced three-dimensional Navier--Stokes equations.
\end{abstract}

\setcounter{tocdepth}{2}
\tableofcontents

\section{Introduction}
We study the evolution of a passive scalar advected by a two-dimensional divergence-free velocity field. More precisely, in the two-dimensional torus  $\T^2 \cong \R^2/ \Z^2$, given  a velocity field $u : [0,1] \times \T^2 \to \R^2$, we consider the Cauchy problem of the advection-diffusion equation 
\begin{align} \label{adv-diff}
\begin{cases}
\partial_t \theta_\kappa + u \cdot \nabla \theta_\kappa = \kappa \Delta \theta_\kappa; \tag{ADV-DIFF}
\\
\theta_{\kappa} (0, \cdot ) = \theta_{\initial} (\cdot ) \in L^\infty(\T^d);
\end{cases}
\end{align}
where the unknown is the scalar  $\theta_{\kappa} : [0,1] \times \T^2 \to \R$, $\kappa \geq 0$ is the diffusivity parameter and $\theta_{\initial } : \T^2 \to \R $ is a given bounded initial datum. If $\kappa =0$ this equation reduces to the well-known advection equation (also called transport equation). Supposing that $u \in L^2((0,1) \times \T^2; \R^2)$ is divergence-free, there exists a unique solution $\theta_\kappa \in L^\infty ((0,1) \times \T^2) \cap L^2((0,1); H^1(\T^2))$, and this solution satisfies the energy equality 
\begin{equation}\tag{E}
\frac{1}{2} \int_{\T^2} |\theta_\kappa (t, x)|^2 dx + \kappa \int_0^t \int_{\T^2} | \nabla \theta_\kappa (s, x)|^2 dx ds = \frac{1}{2} \int_{\T^2} |\theta_{\initial}(x)|^2 dx \,. 
\end{equation}
We say that the collection of solutions
$\{ \theta_{\kappa} \}_{\kappa \in (0,1)}$ exhibits anomalous dissipation if
\begin{align} \label{AD}
\limsup_{\kappa \to 0} \kappa \int_0^1 \int_{\T^2} | \nabla \theta_\kappa (s,x)|^2 dx ds  >0 \,, \tag{AD}
\end{align} 
which directly implies the existence of dissipative
 vanishing diffusivity solutions to the advection equation.
  The main contribution of this paper is the following: 
\begin{theorem}\label{thm-main}
Let $\alpha \in [0,1)$, then there exists an autonomous and divergence-free velocity field $u = \nabla^{\perp} H \in C^\alpha (\T^2; \R^2)$  and an initial datum $\theta_{\initial} \in C^\infty (\T^2)$ such that  the sequence of unique solutions $\{ \theta_\kappa \}_{\kappa >0 }$  to \eqref{adv-diff} exhibits  anomalous dissipation.
Furthermore, 
$$t \to  \kappa \int_{\T^2} |\nabla \theta_{\kappa} (t, x) |^2 dx $$
is uniformly bounded in $\kappa >0$.
\end{theorem}
{ The second part of the theorem follows from Proposition \ref{prop:absolute}, which holds for general autonomous velocity fields. In particular, the energy map \( t \mapsto \| \theta(t, \cdot) \|_{L^2}^2 \) is Lipschitz continuous for any bounded solution \( \theta \) of the advection equation with any autonomous velocity field, obtained via vanishing-diffusivity.}
\begin{remark}[No anomalous dissipation for all initial data]
 For two-dimensional autonomous velocity fields $u = \nabla^{\perp} H \in L^{\infty}(\T^2)$, anomalous dissipation cannot occur for all initial data, namely there exists an initial datum $\theta_{\initial} \in W^{1, \infty}(\T^2)$ such that for the corresponding solutions to the advection-diffusion equation $\theta_{\kappa}$, we have
 \begin{equation}\label{eq:NoAnomalousDissipation}
  \lim_{\kappa \to 0} \kappa \int_{0}^{T} \int_{\T^2} |\nabla \theta_{\kappa}|^2 \, dx dt = 0 \quad \text{for any fixed $T > 0$}.
 \end{equation}
Consider $\theta_{\initial} = H$,  { which is a steady state of the advection equation}, then
 \[
  \partial_t (\theta_{\kappa} - H) + u \cdot \nabla (\theta_{\kappa} - H) = \kappa \Delta \theta_{\kappa}.
 \]
 Multiplying by $(\theta_{\kappa} - H)$ and integrating in space-time (mollifying and passing to the limit) we find
 \begin{align*}
  \| (\theta_{\kappa} - H)(T, \cdot)  \|_{L^2(\T^2)}^2 & + 2 \kappa \int_0^T \int_{\T^2} |\nabla \theta_{\kappa}|^2 \, dx dt 
  \\
  &= 2 \kappa \int_0^T \int_{\T^2} \nabla \theta_{\kappa} \cdot \nabla H \, dx dt \\
  &\leq \kappa \int_0^T \int_{\T^2} |\nabla \theta_{\kappa}|^2 \, dx dt + \kappa \int_0^T \int_{\T^2} |\nabla H|^2 \, dx dt
 \end{align*}
 from which we conclude \eqref{eq:NoAnomalousDissipation}. 
{ A natural question is whether one can construct an example in which anomalous dissipation occurs for every initial datum that is not a steady state.}
\end{remark}
\begin{remark} \label{remark:uniqueness}
An autonomous, divergence-free velocity field $u = \nabla^\perp H \in C^0(\T^2)$ satisfies the \emph{weak Sard property} if 
\[
\mathcal{L}^1 (H (S)) =0, \quad \text{where} \quad S = \{x \in \T^2 : \nabla H =0 \}. 
\]
This is a necessary and sufficient condition for uniqueness of solutions at the level of the advection equation, as proved by Alberti, Bianchini and Crippa in \cite[Theorem 4.7]{ABC14}. 
{ The autonomous velocity field in Theorem \ref{thm-main} cannot satisfy the weak Sard property, otherwise this would imply uniqueness of solutions to the advection equation, contradicting a result of Rowan \cite[Theorem 1.3]{R23}.}
\end{remark}
\begin{remark}[Forced steady state of two-dimensional Euler equations]
{ Given \(\alpha \in [0,1)\), the velocity field \(u\) constructed in Theorem~\ref{thm-main} is a steady solution to the two-dimensional Euler equations with a body force \(F \in C^{\alpha}(\mathbb{T}^2)\) and a pressure \(p \in C^{1,\alpha}(\mathbb{T}^2)\), obtained via the Leray projector
$$u \cdot \nabla u + \nabla p = F \,. $$}
\end{remark}

The problem of anomalous dissipation with autonomous velocity fields was posed by Elgindi and Liss in \cite{EL23} as a problem ``of great mathematical, and possibly physical, interest''. Indeed, the advection-diffusion equation in two dimensions with autonomous velocity fields enjoys more rigidity and constraints with respect to the case of two-dimensional time-dependent velocity fields or autonomous velocity fields in higher dimensions. We now give a swift outline of some aspects of this rigidity.
\begin{itemize}
 \item Uniqueness of solutions to the advection equation with two-dimensional autonomous velocity fields is equivalent to having the velocity field satisfying the weak Sard property, see \cite[Theorem 4.7]{ABC14} by Alberti, Bianchini and Crippa. For time-dependent or autonomous velocity fields in higher dimensions stronger regularity assumptions are required, see for instance \cite{DPL89,A04,BC13}. 
\item  For two-dimensional autonomous  Hamiltonian velocity fields $u = \nabla^\perp H \in C^0$ we have that the sublevel sets are invariant for all the  trajectories (which can be ``non-unique''), since the chain rule  holds $\frac{d}{dt} H (\gamma (t)) = (\nabla H \cdot \nabla^\perp H)(\gamma(t)) =0$. However, this property is not always true in lower regularity regimes, see for instance \cite{SG21} and this is why the $C^\alpha$ regularity is important in Theorem \ref{thm-main}. In particular, we observe that  this constraint on sublevel sets prevents mixing (and even ergodicity) in the sense of ergodic theory, since for instance the open set $\{ H < c \} $ is invariant under any  flow map (that may be non-unique). 
 \item For two-dimensional time-dependent velocity fields { it has been proved in \cite{bedrendiss,ELM23} that the optimal enhanced diffusion rate is $|\ln \kappa|$.
 In contrast, in the two-dimensional autonomous case, the optimal enhanced dissipation rate  is $\kappa^{-1/3}$, as shown by Bru\`e, Coti Zelati and Marconi \cite{BCZM22}.}
 \item At the level of the advection equation, the mixing scale with time-dependent Lipschitz velocity fields can be exponential-in-time, see the first example by Alberti, Crippa and Mazzucato  \cite{ACM-JAMS}, whereas in the two-dimensional autonomous case it is at most linear-in-time \cite{BM21} as proved by Bonicatto and Marconi.
 \item Regarding anomalous dissipation, a consequence of the result in \cite{ABC14} is that anomalous dissipation does not occur if the velocity field is continuous, divergence-free, autonomous and nowhere vanishing, see \cite[Corollary 1.4]{R23}. This fact is known to be false in the case of time-dependent velocity fields and autonomous velocity fields in higher dimensions \cite{Gautam, JS23}.
\end{itemize}

To tackle the two-dimensional autonomous case, the novelty of the present paper is twofold.
Firstly, we introduce a new mathematical approach (criterion) relying on ideas from spontaneous stochasticity, a concept that was introduced in \cite{BGK98,CK03} and, more recently, further developed by Drivas and Eyink in \cite{DE17,DE172,DE173}. In \cite{DE17,DE172,DE173}, they introduce the fluctuation-dissipation formula \eqref{eq:FL-DISS} which provides a link between anomalous dissipation and spontaneous stochasticity.
Secondly, we construct a two-dimensional velocity field which exhibits anomalous dissipation. 
Specifically, we construct a velocity field to which the already mentioned criterion can be applied. 
This velocity field cannot satisfy the weak Sard property as noticed in Remark \ref{remark:uniqueness}. 
Examples of velocity fields which do not satisfy the weak Sard property are known \cite{hamilt2} but it is not clear whether any of those exhibit anomalous dissipation. 
Among the available examples of velocity fields which do not satisfy the weak Sard property in the literature, the one that inspired us the most is the non-divergence free velocity field given in \cite[Section 5]{hamilt2}.

The anomalous dissipation phenomenon presented in our example is  new and distinct from known techniques, which are detailed at the end of this introduction. We consider an autonomous velocity field with the following property: the forward flow of a smoothed approximation of this field tends to push a certain portion of the torus into a fat Cantor set with positive Lebesgue measure.

The stochastic flow is initially governed by the transport term, up to a certain small scale where diffusion begins to dominate, eventually spreading the flow across the entire fat Cantor set. This ``dispersive'' phenomenon of forward stochastic trajectories is responsible for the anomalous dissipation. To mathematically prove this, it is more effective to study the backward stochastic flow.
In the context of the backward stochastic flow, this  phenomenon appears as a highly unstable behavior in response to small perturbations (as a small noise), reminiscent of the concept of spontaneous stochasticity, which we discuss in Section \ref{sec:spontaneous}. This perspective is more advantageous for proving anomalous dissipation. Specifically, we will show that backward stochastic trajectories originating from the fat Cantor set can exhibit significantly different behaviors depending on the realization of the Brownian motion. According to the fluctuation-dissipation formula \eqref{eq:spontaneous-strong}, this variability implies anomalous dissipation, as demonstrated by the criterion provided in Proposition \ref{prop:criterion}. Finally, we note that the dissipation caused by this phenomenon is continuous over time, due to the autonomous nature of the velocity field, as detailed in Proposition \ref{prop:absolute}.
\\
\\
A central motivation for the study of anomalous dissipation comes from fluid dynamics. 
To be precise, we aim at understanding how the cascade of frequencies arises due to a transport term. 
In the incompressible Euler and Navier--Stokes equations, the nonlinear term takes the shape of a transport term and hence understanding anomalous dissipation at the linear level, i.e \eqref{adv-diff} could allow us to understand how the cascade of frequencies arises due to the nonlinear term in the incompressible Euler and Navier--Stokes equations. Anomalous dissipation in the Navier--Stokes setting is known also as 0-th law of turbulence, see for instance \cite{K41,K41c,kolmo,frisch}.
 In recent years, this has been made mathematically rigorous is some situations. In \cite{BDL22} Bru\`e and De Lellis propose to study the three-dimensional Navier--Stokes equations with body forces that may depend on the viscosity parameter but the forces enjoy a uniform regularity such as $\sup_{\nu >0} \| F_\nu \|_{L^{1 + \varepsilon}_t  C^\varepsilon_x} < \infty $ for some $\varepsilon >0$. By imposing this, there are no trivial examples of anomalous dissipation given by the Stokes equations. Following this, Bru\`e, Colombo, Crippa, De Lellis and the second author of this paper prove anomalous dissipation under the sharp Onsager regularity of the solutions $L^3_t C^{1/3-}_x$ in \cite{BCCDLS22}. 
 { Hofmanov\'a, Pappalettera, R. Zhu and X. Zhu using observations from \cite{BCZPSW19}   prove in \cite{HPZZ241} that such solutions  enjoy Kolmogorov $4/5$-th law. It has been noted later by Novack \cite{N23} that Karman--Howarth--Monin relation together with anomalous dissipation imply Kolmogorov's 4/5th law. The author also improves the length scale given in the previous result \cite{HPZZ241}.}
 Furthermore, the same authors in \cite{HPZZ24} give examples of anomalous dissipation for Navier--Stokes equations with stochastic forces.
  Finally, Cheskidov in \cite{C24} provides other interesting non--smooth phenomena in this context, differentiating the  dissipation anomaly phenomenon from anomalous dissipation. 
  These proofs as well as our proof of Theorem \ref{thm:NS} are based on the study of the $(2 + \frac{1}{2})$-dimensional Euler and Navier--Stokes equations already used in \cite{JYo20,JY20}.
 To clarify our second main result, we consider the three-dimensional forced Navier--Stokes equations with a given   force $F_\nu : [0,1] \times  \T^3 \to \R^3 $
\begin{align} \label{NS}
\begin{cases}
\partial_t v_\nu + v_\nu \cdot \nabla v_\nu + \nabla P_\nu = \nu \Delta v_\nu + F_{\nu}  \,,
\\
\diver (v_\nu) =0 \,,
\\
v_\nu (0, \cdot ) = v_{\initial, \nu } (\cdot) \,,
\end{cases}
\end{align}
where the unknowns are the velocity $v_\nu : [0,1] \times \T^3 \to \R^3 $ and the pressure $P_\nu : [0,1] \times \T^3 \to \R$.\footnote{If $\nu =0$ these are the 3D forced Euler equations.}
Assuming $v_{\nu} = (u_{\nu}^{(1)}, u_{\nu}^{(2)}, \theta_{\nu})$ depends only on the two first spatial components, the third component of the equations reduces to an advection-diffusion equation with velocity $(u_{\nu}^{(1)}, u_{\nu}^{(2)})$. With this assumption, \eqref{NS} are known as $(2 + \frac{1}{2})$-dimensional Navier--Stokes equations.
Theorem \ref{thm-main}, together with some key regularity properties of our two-dimensional autonomous velocity field (see Lemma \ref{lemma:NS-Calpha}) yield { anomalous dissipation for the forced three dimensional Navier--Stokes equations with a continuous in time kinetic energy and  autonomous forces.}  In particular, this result answers two open questions by Bru\`e and De Lellis posed in \cite[Question 2.2 and Question 2.3]{BDL22} for the three-dimensional forced Navier-Stokes equations. 
{ These questions were previously addressed in \cite{JS23} for the four-dimensional forced Navier-Stokes equations. 
The main improvement of our result comes from Theorem \ref{thm-main}, which constructs an autonomous velocity field in two dimensions exhibiting anomalous dissipation and that was previously unknown.}
In Section~\ref{sec:NS} we prove the following statement.
\begin{theorem} \label{thm:NS}
For any $\alpha \in (0,1)$, there exist a sequence of viscosity parameters $\{ \nu_q \}_{q \in \N}$,  a sequence of time independent smooth forces  $\{ F_{\nu_q }\}_{q \in \N} \subset C^\infty ( \T^3 )$ and a sequence of smooth initial data $\{ v_{\initial, q} \}_{q \in \N} \subset C^\infty ( \T^3 ) $ such that for any $q \geq 1$ there exists a unique solution $v_{\nu_q} : [0,1] \times \T^3 \to \R^3$ to \eqref{NS} and it enjoys the following properties: 
\begin{enumerate}
\item $ \| F_{\nu_q} -  F_0 \|_{C^\alpha} + \| v_{\initial, \nu_q} - v_{\initial} \|_{C^\alpha} \to 0$ as $\nu_q \to 0$ for some force $F_0 \in C^\alpha$ and initial datum $v_{\initial } \in C^\alpha$.
\item  Anomalous dissipation holds, i.e. $$ \limsup_{\nu_q \to 0} \nu_q \int_0^1 \int_{\T^3} | \nabla v_{\nu_q} (t, x)|^2 dx dt >0 \,,$$
\item  There exists $v_0 \in L^\infty$ such that  up to not relabelled subsequences
$$v_{\nu_q} \overset{*-L^\infty}{\rightharpoonup} v_0  \in L^\infty $$
and $v_0$ is a solution to the $3D$ forced Euler equations with  $v_{\initial} $ and force $F_0$ and finally $e(t) = \frac{1}{2} \int_{\T^3 } |v_0 (t, x)|^2 dx \in W^{1, \infty} (0,1)$.
\end{enumerate}
\end{theorem}
\bigskip
We now provide a summary of existing works and techniques in the study of anomalous dissipation. 
For the sake of clarity and brevity, we restrict our attention only to advection-diffusion equations.
Subsequently, we go through the techniques while pointing out the challenges that these techniques present in the two-dimensional autonomous case. This motivates and justifies the introduction of the new approach presented in this work. 
The first and pioneering result of anomalous dissipation was given in \cite{Gautam} by Drivas, Elgindi, Iyer and Jeong, where the authors use a mixing velocity field in $L^1_t C^\alpha_x $, with $\alpha \in (0,1)$ fixed but arbitrary, for which bounded solutions, with initial data close to an eigenfunction of the Laplacian, exhibit anomalous dissipation.
This study attracted the attention of many subsequent mathematical investigations. In \cite{CCS22}, Colombo, Crippa and the second author of the present paper construct a new mixing velocity field to prove anomalous dissipation in any supercritical Yaglom's regime. More precisely, for any fixed 
\begin{align} \label{yaglom}
 \alpha + 2 \beta < 1 \tag{YAG}
\end{align}
the authors construct a velocity field $u \in L^\infty((0,1); C^\alpha (\T^2))$ and an initial datum for which the corresponding solutions $\{ \theta_{\kappa} \}_{\kappa \in (0,1)}$ exhibit anomalous dissipation and enjoy the regularity 
$$\sup_{\kappa \in (0,1)} \| \theta_\kappa \|_{L^2 ((0,1); C^\beta (\T^2))} < \infty.$$
{ In the case of the construction given in Theorem \ref{thm-main}, it is unclear how to prove regularity of the solutions to \eqref{adv-diff} uniformly in $\kappa>0$ and anomalous dissipation in the full supercritical Yaglom's regime with a two dimensional autonomous divergence-free velocity field remains open. }
In \cite{AV23}, Armstrong and Vicol construct a velocity field in $L^\infty ((0,1); C^\alpha (\T^2))$ with $\alpha < 1/3$ and prove via a  striking  technique known as quantitative homogenization  that corresponding solutions exhibit anomalous dissipation for all non-constant initial data in $H^1(\T^2)$. 
Furthermore, dissipation in this scenario occurs continuously in time, as predicted by the theory of scalar turbulence. In constructing this velocity field, at a fixed point in time, all frequencies are activated, and the singular set of this velocity field spans the full space-time dimension in $(0,1) \times \mathbb{T}^2$. Inspired by this construction, in \cite{BBS23} Burczak, Sz\'ekelyhidi and Wu  combined convex integration and quantitative homogenization theory to construct a dense set of solutions to the three-dimensional Euler equations in $C^0$ for which the corresponding solutions to the advection-diffusion equation exhibit anomalous dissipation for any non-constant initial data in $H^1(\T^3)$.
In \cite{EL23}, Elgindi and Liss prove anomalous dissipation for any non-constant smooth initial data in any supercritical Yaglom's regime. The construction of their velocity field is a rescaled-in-time and in-space version of the one in \cite{ELM23} by Elgindi, Liss and Mattingly. Finally, in the Kraichnan model \cite{K68}, where the velocity field is a Gaussian random field which is white-in-time and rough-in-space (only H\"older continuous), anomalous dissipation has been proved by Bernard, Gawedzki and Kupiainen \cite{BGK98}, see also \cite{CK03,K03}. Recently, Rowan \cite{R23} provided a PDE-based proof of anomalous dissipation in the Kraichnan model.
Now we go through the techniques mentioned in the summary above and for each technique point out the challenges to overcome in order to apply them in the autonomous two-dimensional setting.

 {\em Balanced growth of norms.} The works based on balanced growth of Sobolev norms \cite{Gautam, EL23} construct smooth velocity fields  which satisfy for any $t \in (0,1)$
 \begin{equation}\label{balanced}
\| \theta_0 (t,\cdot ) \|_{H^1}^{\sigma} \sim \| \theta_0 (t,\cdot ) \|_{L^2}^{\sigma - 1} \| \theta_0 (t, \cdot ) \|_{H^{\sigma}}  \qquad \text{and } \int_0^1 \| \nabla \theta_0 (t, \cdot) \|_{L^2}^2 dt = \infty 
\end{equation}
 for some $\sigma \in (1,2]$,
where $\theta_0$ denotes the unique solution to the advection equation.
Such a condition is particularly well-suited for self-similar constructions and alternating shear flows which are smooth for any time less than a fixed singular time $T$. 
However, solutions to the advection equation may become non-unique with a $C^{1-}$ divergence-free velocity fields in general and the condition \eqref{balanced} must be adapted accordingly. This means that the conditions must be specified for the solution to the advection equation with a regularized velocity field. If the regularization is sufficiently small (for instance by mollification)  with respect to $\kappa$, then the solution to the advection-diffusion equation with the velocity field is close to that with the regularized velocity field. However, the second condition with the regularized velocity field cannot hold.  A possible adapted assumption is to consider a quantitative growth explosion of such an integral depending on the regularization of the velocity field, which is not straightforward to satisfy. Therefore, finding sufficient conditions for anomalous dissipation based on balanced growth, which are satisfied by a two-dimensional autonomous velocity field, seems non-trivial. Nevertheless, it is an interesting mathematical problem in the opinion of the authors of this paper.
 
  {\em Mixing.} The works based on mixing, such as \cite{CCS22}, use a condition of the following type:
 \begin{align}\label{eq:mixing}
\| \theta_0 (t_q, \cdot) - \theta_{\initial} (\lambda_q \cdot ) \|_{L^2 } \leq \frac{1}{100} \| \theta_{\initial} \|_{L^2} \qquad \text{for  some } t_q \to 1\,, \lambda_q \to \infty \,.
\end{align}
This condition too is particularly well-suited for velocity fields which generate solutions to the advection equation with a high degree of self-similarity. It should be stressed that although this condition has similarities to the previous one, it is different. Indeed, \eqref{balanced} is not satisfied by the velocity field in \cite{CCS22}, but \eqref{eq:mixing} is. { In \cite{JS23}, the authors use this time-dependent, divergence-free velocity field in two dimensions and extend it to the three-dimensional case by replacing the time variable with the third spatial variable. This approach yields anomalous dissipation for certain initial data with a three-dimensional autonomous velocity field. Clearly, this trick does not aid in the construction of a two-dimensional autonomous velocity field exhibiting anomalous dissipation, which is a more subtle problem. }
 
  {\em Quantitative homogenization.} In the works based on quantitative homogenization \cite{AV23, BBS23}, one constructs a Cauchy sequence in some H\"older space of smooth divergence-free velocity fields $\{ u_m \}_m$ for which the corresponding solutions $\{ \theta_m \}_m$ to the advection-diffusion equation with diffusivity $\kappa_m \to 0$ exhibit anomalous dissipation.
To do so, it is necessary to control a transport term involving the term
$$ u_m \cdot  \nabla \tilde{\chi}_{m , k}\,,$$
where $ \tilde{\chi}_{m , k}$ is a corrector function coming from the homogenization technique.
The term can in fact not be controlled but cancelled by a time derivative. To put this idea into action, one composes $u_m$ with the flow map of $u_{m-1}$, which causes the velocity field to become time-dependent. This technical point is also the reason why the velocity field can only be constructed in $C^\alpha$ with $\alpha < 1/3$.

\subsection*{Outline of the paper}
We start by introducing some notation and provide preliminaries in Section~\ref{sec:NotAndPrel}. In Section~\ref{sec:spontaneous}, we review the concept of spontaneous stochasticity based on which we develop a criterion for anomalous dissipation in Section~\ref{sec:AD-Crit}. { In Section~\ref{sec:ideas} we explain the main ideas of the construction of the velocity field as well as of the proof of the main theorem.}  Then, we make a choice of parameters in Section~\ref{sec:choice}. Given this choice, we construct the velocity field $u$ of Theorem~\ref{thm-main} in Section~\ref{sec:construction}. 
{In Section~\ref{sec:stability}, we prove stability results needed for the proof of Theorem~\ref{thm-main}} 
Subsequently, we prove Theorem~\ref{thm-main} in Section~\ref{sec:ProofMainThm}. Finally, in Section~\ref{sec:NS}, we prove Theorem~\ref{thm:NS}.

\section*{Acknowledgements} 
MS and CJ are supported by the Swiss State Secretariat for Education, Research and Innovation (SERI) under contract number MB22.00034. MS acknowledges support from the Chapman Fellowship at Imperial College London. The authors are grateful to Maria Colombo for fruitful discussions, to Vlad Vicol for useful observations on continuous in time dissipation and Lucio Galeati for discussions about spontaneous stochasticity and references on the stochastic part.
The authors would like to thank the anonymous referees for careful reading of the manuscript and valuable comments.


\newcommand{\DTPipe}[7]{ 
\draw[black, thick] (#1, #2) -- (#1+#3+2*#4*#5, #2) -- (#1+#3+2*#4*#5, #2-3*#4-#6) -- (#1+#3+3*#4*#5+#7, #2-3*#4-#6);
\draw[black, thick] (#1, #2-#4) -- (#1+#3+#4*#5, #2-#4) -- (#1+#3+#4*#5, #2-4*#4-#6) -- (#1+#3+3*#4*#5+#7, #2-4*#4-#6);
}

\newcommand{\DTPipeThin}[7]{ 
\draw[black, thin] (#1, #2) -- (#1+#3+2*#4*#5, #2) -- (#1+#3+2*#4*#5, #2-3*#4-#6) -- (#1+#3+3*#4*#5+#7, #2-3*#4-#6);
\draw[black, thin] (#1, #2-#4) -- (#1+#3+#4*#5, #2-#4) -- (#1+#3+#4*#5, #2-4*#4-#6) -- (#1+#3+3*#4*#5+#7, #2-4*#4-#6);
}

\newcommand{\HalfDownPipe}[7]{ 
\foreach \y in {1, ..., #4} {
 \DTPipe{#1}{#2+#3*\y/#4-#3}{\y*#6/#4 - 4*\y*#3/#4 - 2*#5*#3/#4}{#3/#4}{#5}{#7}{#6-  \y*#6/#4 + 4*\y*#3/#4 - 4*#3}
}
\draw[black, thick] (#1 + #6 - 4*#3 + #5*#3/#4, #2-#3-#7-3*#3/#4) -- (#1+#6-#3, #2-#3-#7-3*#3/#4) -- (#1+#6-#3, #2-#3) -- (#1+#6, #2-#3);
\draw[black, thick] (#1 + #6 - 4*#3 + #5*#3/#4, #2-#7-3*#3/#4) -- (#1+#6-2*#3, #2-#7-3*#3/#4) -- (#1+#6-2*#3, #2) -- (#1+#6, #2);
}

\newcommand{\HalfDownPipeThin}[7]{ 
\foreach \y in {1, ..., #4} {
 \DTPipeThin{#1}{#2+#3*\y/#4-#3}{\y*#6/#4 - 4*\y*#3/#4 - 2*#5*#3/#4}{#3/#4}{#5}{#7}{#6-  \y*#6/#4 + 4*\y*#3/#4 - 4*#3}
}
\draw[black, thin] (#1 + #6 - 4*#3 + #5*#3/#4, #2-#3-#7-3*#3/#4) -- (#1+#6-#3, #2-#3-#7-3*#3/#4) -- (#1+#6-#3, #2-#3) -- (#1+#6, #2-#3);
\draw[black, thin] (#1 + #6 - 4*#3 + #5*#3/#4, #2-#7-3*#3/#4) -- (#1+#6-2*#3, #2-#7-3*#3/#4) -- (#1+#6-2*#3, #2) -- (#1+#6, #2);
}

\newcommand{\BMPipe}[7]{ 
 \HalfDownPipe{#1}{#2}{#3}{#4}{#5}{#6}{#7}
\begin{scope}[yscale=-1,xscale=1, yshift=- 2*#2 cm]
  \HalfDownPipe{#1}{#2}{#3}{#4}{#5}{#6}{#7}
\end{scope}
}

\newcommand{\BMPipeThin}[7]{ 
 \HalfDownPipeThin{#1}{#2}{#3}{#4}{#5}{#6}{#7}
\begin{scope}[yscale=-1,xscale=1, yshift=- 2*#2 cm]
  \HalfDownPipeThin{#1}{#2}{#3}{#4}{#5}{#6}{#7}
\end{scope}
}

\newcommand{\BMPipeWithRectangles}[7]{ 
 \BMPipe{#1}{#2}{#3}{#4}{#5}{#6}{#7}
 \foreach \y in {1, ..., #4} {
 	\ifodd\y{
 		\path[fill=teal, fill opacity=0.15] (\y*#6/#4 - 4*\y*#3/#4 - 0.5*#5*#3/#4 - 0.5*#6/#4 + 2*#3/#4, #3 + 2*#3/#4) rectangle ++(#6/#4 - 4*#3/#4, #7-#3);
		\draw[blue] (\y*#6/#4 - 4*\y*#3/#4 - 0.7*#5*#3/#4 - 0.5*#6/#4 + 2*#3/#4, #3 + 2*#3/#4 + 0.8*#7-0.8*#3) node[anchor=west]{$\widetilde{R}_{\y}$};
		\path[fill=teal, fill opacity=0.15] (\y*#6/#4 - 4*\y*#3/#4 - 0.5*#5*#3/#4 - 0.5*#6/#4 + 2*#3/#4, - #3 - 2*#3/#4) rectangle ++(#6/#4 - 4*#3/#4, -#7+#3);
		\draw[blue] (\y*#6/#4 - 4*\y*#3/#4 - 0.7*#5*#3/#4 - 0.5*#6/#4 + 2*#3/#4, - #3 - 2*#3/#4 - 0.8*#7 + 0.8*#3) node[anchor=west]{$\widetilde{R}_{\number\numexpr \y + #4 \relax}$};

 }
 	\else{
		\path[fill=teal, fill opacity=0.25] (\y*#6/#4 - 4*\y*#3/#4 - 0.5*#5*#3/#4 - 0.5*#6/#4 + 2*#3/#4, #3 + 2*#3/#4) rectangle ++(#6/#4 - 4*#3/#4, #7-#3);
		\draw[blue] (\y*#6/#4 - 4*\y*#3/#4 - 0.7*#5*#3/#4 - 0.5*#6/#4 + 2*#3/#4, #3 + 2*#3/#4 + 0.8*#7-0.8*#3) node[anchor=west]{$\widetilde{R}_{\y}$};
		\path[fill=teal, fill opacity=0.25] (\y*#6/#4 - 4*\y*#3/#4 - 0.5*#5*#3/#4 - 0.5*#6/#4 + 2*#3/#4, - #3 - 2*#3/#4) rectangle ++(#6/#4 - 4*#3/#4, -#7+#3);
		\draw[blue] (\y*#6/#4 - 4*\y*#3/#4 - 0.7*#5*#3/#4 - 0.5*#6/#4 + 2*#3/#4, - #3 - 2*#3/#4 - 0.8*#7 + 0.8*#3) node[anchor=west]{$\widetilde{R}_{\number\numexpr \y + #4 \relax}$};
	}
 \fi
}
}

\newcommand{\OnlyRectangles}[7]{ 
 \foreach \y in {1, ..., #4} {
 	\ifodd\y{
 		\path[fill=teal, fill opacity=0.15] (\y*#6/#4 - 4*\y*#3/#4 - 0.5*#5*#3/#4 - 0.5*#6/#4 + 2*#3/#4, #3 + 2*#3/#4) rectangle ++(#6/#4 - 4*#3/#4, #7-#3);
		\path[fill=teal, fill opacity=0.15] (\y*#6/#4 - 4*\y*#3/#4 - 0.5*#5*#3/#4 - 0.5*#6/#4 + 2*#3/#4, - #3 - 2*#3/#4) rectangle ++(#6/#4 - 4*#3/#4, -#7+#3);
 }
 	\else{
		\path[fill=teal, fill opacity=0.25] (\y*#6/#4 - 4*\y*#3/#4 - 0.5*#5*#3/#4 - 0.5*#6/#4 + 2*#3/#4, #3 + 2*#3/#4) rectangle ++(#6/#4 - 4*#3/#4, #7-#3);
		\path[fill=teal, fill opacity=0.25] (\y*#6/#4 - 4*\y*#3/#4 - 0.5*#5*#3/#4 - 0.5*#6/#4 + 2*#3/#4, - #3 - 2*#3/#4) rectangle ++(#6/#4 - 4*#3/#4, -#7+#3);
	}
 \fi
}
\draw[black, thick, <->] (#6 - 4*#3 - 0.5*#5*#3/#4 - 0.5*#6/#4 + 2*#3/#4, #3 + 2*#3/#4 + #7-#3 + 0.2) -- ++(#6/#4 - 4*#3/#4, 0);
\draw[black, thick, <->] (#6 - 4*#3 - 0.5*#5*#3/#4 - 0.5*#6/#4 + 2*#3/#4 + #6/#4 - 4*#3/#4 + 0.2, #3 + 2*#3/#4) -- ++(0, #7-#3);
\draw[black] (#6 - 4*#3 - 0.5*#5*#3/#4 - 0.5*#6/#4 + 2*#3/#4 + #6/#4 - 4*#3/#4 + 0.2, #3 + 2*#3/#4 + 0.5*#7 - 0.5*#3) node[anchor=west]{$L_1$};
\draw[black] (#6 - 4*#3 - 0.5*#5*#3/#4 - 0.5*#6/#4 + 2*#3/#4 + 0.5*#6/#4 - 2*#3/#4 , #3 + 2*#3/#4 + #7-#3 + 0.2) node[anchor=south]{$A_1 + B_1$};
}

\newcommand{\RectanglesWithPipes}[7]{ 
 \BMPipe{#1}{#2}{#3}{#4}{#5}{#6}{#7}
 \foreach \y in {1, ..., #4} {
 	\ifodd\y{
 		\path[fill=teal, fill opacity=0.15] (\y*#6/#4 - 4*\y*#3/#4 - 0.5*#5*#3/#4 - 0.5*#6/#4 + 2*#3/#4, #3 + 2*#3/#4) rectangle ++(#6/#4 - 4*#3/#4, #7-#3);
		\path[fill=teal, fill opacity=0.15] (\y*#6/#4 - 4*\y*#3/#4 - 0.5*#5*#3/#4 - 0.5*#6/#4 + 2*#3/#4, - #3 - 2*#3/#4) rectangle ++(#6/#4 - 4*#3/#4, -#7+#3);
 }
 	\else{
		\path[fill=teal, fill opacity=0.25] (\y*#6/#4 - 4*\y*#3/#4 - 0.5*#5*#3/#4 - 0.5*#6/#4 + 2*#3/#4, #3 + 2*#3/#4) rectangle ++(#6/#4 - 4*#3/#4, #7-#3);
		\path[fill=teal, fill opacity=0.25] (\y*#6/#4 - 4*\y*#3/#4 - 0.5*#5*#3/#4 - 0.5*#6/#4 + 2*#3/#4, - #3 - 2*#3/#4) rectangle ++(#6/#4 - 4*#3/#4, -#7+#3);
	}
 \fi
}
\draw[black, thick, <->, dashed] (#6 - 4*#3 - 0.5*#6/#4 + 2*#3/#4 - 0.5*#6/#4 + 2*#3/#4, #3 + 2*#3/#4 + 0.5*#7 - 0.5*#3) -- ++(#6/#4 - 4*#3/#4 - #5*#3/#4, 0);
\draw[black] (#6 - 4*#3 - 0.5*#6/#4 + 2*#3/#4 - 0.5*#6/#4 + 2*#3/#4 + 0.5*#6/#4 - 2*#3/#4 - 0.5*#5*#3/#4, #3 + 2*#3/#4 + 0.5*#7 - 0.5*#3) node[anchor=south]{$B_1$};
\draw[black, thick, <->, dashed] (#6 - 4*#3 - 0.5*#6/#4 + 2*#3/#4 - 0.5*#6/#4 + 2*#3/#4 + #6/#4 - 4*#3/#4 - #5*#3/#4, #3 + 2*#3/#4 + 0.7*#7 - 0.7*#3) -- ++(#5*#3/#4, 0);
\draw[black] (#6 - 4*#3 - 0.5*#6/#4 + 2*#3/#4 - 0.5*#6/#4 + 2*#3/#4 + #6/#4 - 4*#3/#4 - 0.5*#5*#3/#4, #3 + 2*#3/#4 + 0.7*#7 - 0.7*#3) node[anchor=south]{$A_1$};
\draw[black, thick, <->, dashed] (#6 - 4*#3 - 0.5*#5*#3/#4 - 0.5*#6/#4 + 2*#3/#4, #3 + 2*#3/#4 + 0.2) -- ++(#6/#4 - 4*#3/#4, 0);
\draw[black, thick, <->, dashed] (- 0.5*#5*#3/#4 - 0.5*#6/#4 + 2*#3/#4 + #6/#4 - 4*#3/#4 + 0.2, #3 + 2*#3/#4) -- ++(0, #7-#3);
\draw[black] (- 0.5*#5*#3/#4 - 0.5*#6/#4 + 2*#3/#4 + #6/#4 - 4*#3/#4 + 0.2, #3 + 2*#3/#4 + 0.5*#7 - 0.5*#3) node[anchor=east]{$L_1$};
\draw[black] (#6 - 4*#3 - 0.5*#5*#3/#4 - 0.5*#6/#4 + 2*#3/#4 + 0.5*#6/#4 - 2*#3/#4 , #3 + 2*#3/#4 + 0.2) node[anchor=north]{$A_1 + B_1$};
\draw[black, thick, <->, dashed] (#6 - 4*#3 - 0.5*#5*#3/#4 - 0.5*#6/#4 + 2*#3/#4 + #6/#4 - 4*#3/#4, #3 + 2*#3/#4 + 1.8) -- (#6, #3 + 2*#3/#4 + 1.8);
\draw[black] (#6 - 4*#3 - 0.5*#5*#3/#4 - 0.5*#6/#4 + 2*#3/#4 + #6/#4 - 4*#3/#4 + 0.8, #3 + 2*#3/#4 + 1.8) node[anchor=south]{$\frac{3 A_0}{2}$};
\draw[black, thick, <->, dashed] (0,- #3 - 2*#3/#4 - 0.5*#7 + 0.5*#3) -- (- 0.5*#5*#3/#4 - 0.5*#6/#4 + 2*#3/#4 + #6/#4 - 4*#3/#4, - #3 - 2*#3/#4 - 0.5*#7 + 0.5*#3);
\draw[black] (- 0.25*#5*#3/#4 - 0.25*#6/#4 + 1*#3/#4 + 0.5*#6/#4 - 2*#3/#4, - #3 - 2*#3/#4 - 0.5*#7 + 0.5*#3) node[anchor=south]{$\frac{A_0}{2}$};
   \ArrowsStructureOnlyVelocity{#1}{#2}{#3}{#4}{#5}{#6}{#7}{0.8}
 \begin{scope}[yscale=-1,xscale=1, yshift=- 2*#2 cm]
   \ArrowsStructureOnlyVelocity{#1}{#2}{#3}{#4}{#5}{#6}{#7}{0.8}
\end{scope}
}

\newcommand{\OnlyRectanglesDarker}[7]{ 
 \foreach \y in {1, ..., #4} {
 	\ifodd\y{
 		\path[fill=blue, fill opacity=0.20] (\y*#6/#4 - 4*\y*#3/#4 - 0.5*#5*#3/#4 - 0.5*#6/#4 + 2*#3/#4, #3 + 2*#3/#4) rectangle ++(#6/#4 - 4*#3/#4, #7-#3);
		\path[fill=blue, fill opacity=0.20] (\y*#6/#4 - 4*\y*#3/#4 - 0.5*#5*#3/#4 - 0.5*#6/#4 + 2*#3/#4, - #3 - 2*#3/#4) rectangle ++(#6/#4 - 4*#3/#4, -#7+#3);
 }
 	\else{
		\path[fill=blue, fill opacity=0.35] (\y*#6/#4 - 4*\y*#3/#4 - 0.5*#5*#3/#4 - 0.5*#6/#4 + 2*#3/#4, #3 + 2*#3/#4) rectangle ++(#6/#4 - 4*#3/#4, #7-#3);
		\path[fill=blue, fill opacity=0.35] (\y*#6/#4 - 4*\y*#3/#4 - 0.5*#5*#3/#4 - 0.5*#6/#4 + 2*#3/#4, - #3 - 2*#3/#4) rectangle ++(#6/#4 - 4*#3/#4, -#7+#3);
	}
 \fi
}
}

\newcommand{\OnlyRectanglesTwoGen}[7]{ 
 \foreach \y in {1, ..., #4} {
 	\ifodd\y{
 		\path[fill=teal, fill opacity=0.15] (\y*#6/#4 - 4*\y*#3/#4 - 0.5*#5*#3/#4 - 0.5*#6/#4 + 2*#3/#4, #3 + 2*#3/#4) rectangle ++(#6/#4 - 4*#3/#4, #7-#3);
		\path[fill=teal, fill opacity=0.15] (\y*#6/#4 - 4*\y*#3/#4 - 0.5*#5*#3/#4 - 0.5*#6/#4 + 2*#3/#4, - #3 - 2*#3/#4) rectangle ++(#6/#4 - 4*#3/#4, -#7+#3);
 }
 	\else{
		\path[fill=teal, fill opacity=0.25] (\y*#6/#4 - 4*\y*#3/#4 - 0.5*#5*#3/#4 - 0.5*#6/#4 + 2*#3/#4, #3 + 2*#3/#4) rectangle ++(#6/#4 - 4*#3/#4, #7-#3);
		\path[fill=teal, fill opacity=0.25] (\y*#6/#4 - 4*\y*#3/#4 - 0.5*#5*#3/#4 - 0.5*#6/#4 + 2*#3/#4, - #3 - 2*#3/#4) rectangle ++(#6/#4 - 4*#3/#4, -#7+#3);
	}
 \fi
 
 \begin{scope}[shift={(\y*#6/#4 - 4*\y*#3/#4 - 0.5*#5*#3/#4, #3 + 2*#3/#4)}, rotate=90]
	\OnlyRectanglesDarker{0}{0}{0.5*#5*#3/#4}{5}{2}{2.4}{0.58}
 \end{scope}
 
  \begin{scope}[shift={(\y*#6/#4 - 4*\y*#3/#4 - 0.5*#5*#3/#4, - #3 - 2*#3/#4)}, rotate=-90]
	\OnlyRectanglesDarker{0}{0}{0.5*#5*#3/#4}{5}{2}{2.4}{0.58}
 \end{scope}
}
}

\newcommand{\RectanglesTwoGenWithPipes}[7]{ 
 \BMPipe{#1}{#2}{#3}{#4}{#5}{#6}{#7}
     \ArrowsStructureOnlyVelocityTwo{#1}{#2}{#3}{#4}{#5}{#6}{#7}{0.8}
 \begin{scope}[yscale=-1,xscale=1, yshift=- 2*#2 cm]
   \ArrowsStructureOnlyVelocityTwo{#1}{#2}{#3}{#4}{#5}{#6}{#7}{0.8}
\end{scope}
 \foreach \y in {1, ..., #4} {
 	\ifodd\y{
 		\path[fill=white, fill opacity=1] (\y*#6/#4 - 4*\y*#3/#4 - 0.5*#5*#3/#4 - 0.5*#6/#4 + 2*#3/#4, #3 + 2*#3/#4) rectangle ++(#6/#4 - 4*#3/#4, #7-#3);
		\path[fill=white, fill opacity=1] (\y*#6/#4 - 4*\y*#3/#4 - 0.5*#5*#3/#4 - 0.5*#6/#4 + 2*#3/#4, - #3 - 2*#3/#4) rectangle ++(#6/#4 - 4*#3/#4, -#7+#3);
 }
 	\else{
		\path[fill=white, fill opacity=1] (\y*#6/#4 - 4*\y*#3/#4 - 0.5*#5*#3/#4 - 0.5*#6/#4 + 2*#3/#4, #3 + 2*#3/#4) rectangle ++(#6/#4 - 4*#3/#4, #7-#3);
		\path[fill=white, fill opacity=1] (\y*#6/#4 - 4*\y*#3/#4 - 0.5*#5*#3/#4 - 0.5*#6/#4 + 2*#3/#4, - #3 - 2*#3/#4) rectangle ++(#6/#4 - 4*#3/#4, -#7+#3);
	}
 \fi
 \begin{scope}[shift={(\y*#6/#4 - 4*\y*#3/#4 - 0.5*#5*#3/#4, #3 + 2*#3/#4)}, rotate=90]
 	\OnlyRectanglesDarker{0}{0}{0.5*#5*#3/#4}{5}{2}{2.4}{0.58}
	\BMPipe{0}{0}{0.5*#5*#3/#4}{5}{2}{2.4}{0.58}
 \end{scope}
 
  \begin{scope}[shift={(\y*#6/#4 - 4*\y*#3/#4 - 0.5*#5*#3/#4, - #3 - 2*#3/#4)}, rotate=-90]
  	\OnlyRectanglesDarker{0}{0}{0.5*#5*#3/#4}{5}{2}{2.4}{0.58}
	\BMPipe{0}{0}{0.5*#5*#3/#4}{5}{2}{2.4}{0.58}
 \end{scope}
}
}

\newcommand{\RectanglesTwoGenWithPipesPeriodicity}[7]{ 
 \BMPipe{#1}{#2}{#3}{#4}{#5}{#6}{#7}
     \ArrowsStructureOnlyVelocityTwo{#1}{#2}{#3}{#4}{#5}{#6}{#7}{0.65}
 \begin{scope}[yscale=-1,xscale=1, yshift=- 2*#2 cm]
   \ArrowsStructureOnlyVelocityTwo{#1}{#2}{#3}{#4}{#5}{#6}{#7}{0.65}
\end{scope}
 \foreach \y in {1, ..., #4} {
 	\ifodd\y{
 		\path[fill=white, fill opacity=1] (\y*#6/#4 - 4*\y*#3/#4 - 0.5*#5*#3/#4 - 0.5*#6/#4 + 2*#3/#4, #3 + 2*#3/#4) rectangle ++(#6/#4 - 4*#3/#4, #7-#3);
		\path[fill=white, fill opacity=1] (\y*#6/#4 - 4*\y*#3/#4 - 0.5*#5*#3/#4 - 0.5*#6/#4 + 2*#3/#4, - #3 - 2*#3/#4) rectangle ++(#6/#4 - 4*#3/#4, -#7+#3);
 }
 	\else{
		\path[fill=white, fill opacity=1] (\y*#6/#4 - 4*\y*#3/#4 - 0.5*#5*#3/#4 - 0.5*#6/#4 + 2*#3/#4, #3 + 2*#3/#4) rectangle ++(#6/#4 - 4*#3/#4, #7-#3);
		\path[fill=white, fill opacity=1] (\y*#6/#4 - 4*\y*#3/#4 - 0.5*#5*#3/#4 - 0.5*#6/#4 + 2*#3/#4, - #3 - 2*#3/#4) rectangle ++(#6/#4 - 4*#3/#4, -#7+#3);
	}
 \fi
 \begin{scope}[shift={(\y*#6/#4 - 4*\y*#3/#4 - 0.5*#5*#3/#4, #3 + 2*#3/#4)}, rotate=90]
	\ArrowsStructureOnlyVelocityTwoOnlySecondPart{0}{0}{0.5*#5*#3/#4}{6}{2}{2.4}{0.58}{0.30}
	\begin{scope}[yscale=-1,xscale=1, yshift=- 2*#2 cm]
   		\ArrowsStructureOnlyVelocityTwoOnlySecondPart{0}{0}{0.5*#5*#3/#4}{6}{2}{2.4}{0.58}{0.30}
	\end{scope}
	\BMPipeThin{0}{0}{0.5*#5*#3/#4}{6}{2}{2.4}{0.41}
 \end{scope}
 
  \begin{scope}[shift={(\y*#6/#4 - 4*\y*#3/#4 - 0.5*#5*#3/#4, - #3 - 2*#3/#4)}, rotate=-90]
	\ArrowsStructureOnlyVelocityTwoOnlySecondPart{0}{0}{0.5*#5*#3/#4}{6}{2}{2.4}{0.58}{0.30}
	\begin{scope}[yscale=-1,xscale=1, yshift=- 2*#2 cm]
   		\ArrowsStructureOnlyVelocityTwoOnlySecondPart{0}{0}{0.5*#5*#3/#4}{6}{2}{2.4}{0.58}{0.30}
	\end{scope}
	\BMPipeThin{0}{0}{0.5*#5*#3/#4}{6}{2}{2.4}{0.41}
 \end{scope}
}
}

\newcommand{\BMPipeWithRectanglesAndHitSet}[7]{ 
 \BMPipeWithRectangles{#1}{#2}{#3}{#4}{#5}{#6}{#7}
  \foreach \y in {1, ..., #4} {
  		\draw[purple, very thick] (\y*#6/#4 - 4*\y*#3/#4 - 0.5*#5*#3/#4 - 0.5*#6/#4 + 2*#3/#4, #3 + 1.4*#3/#4 - \y*#3/#4) -- ++(#6/#4 - 4*#3/#4, 0);
  		\draw[purple, very thick] (\y*#6/#4 - 4*\y*#3/#4 - 0.5*#5*#3/#4 - 0.5*#6/#4 + 2*#3/#4, - #3 - 1.4*#3/#4 + \y*#3/#4) -- ++(#6/#4 - 4*#3/#4, 0);
}
}

\newcommand{\BMPipeWithRectanglesTwo}[7]{ 
 \BMPipe{#1}{#2}{#3}{#4}{#5}{#6}{#7}
 \foreach \y in {2, ..., \number\numexpr#4-1\relax} {
 	\ifodd\y{
 		\path[fill=teal, fill opacity=0.15] (\y*#6/#4 - 4*\y*#3/#4 - 0.5*#5*#3/#4 - 0.5*#6/#4 + 2*#3/#4, #3 + 2*#3/#4) rectangle ++(#6/#4 - 4*#3/#4, #7-#3);
		\path[fill=teal, fill opacity=0.15] (\y*#6/#4 - 4*\y*#3/#4 - 0.5*#5*#3/#4 - 0.5*#6/#4 + 2*#3/#4, - #3 - 2*#3/#4) rectangle ++(#6/#4 - 4*#3/#4, -#7+#3);
 }
 	\else{
		\path[fill=teal, fill opacity=0.25] (\y*#6/#4 - 4*\y*#3/#4 - 0.5*#5*#3/#4 - 0.5*#6/#4 + 2*#3/#4, #3 + 2*#3/#4) rectangle ++(#6/#4 - 4*#3/#4, #7-#3);
		\path[fill=teal, fill opacity=0.25] (\y*#6/#4 - 4*\y*#3/#4 - 0.5*#5*#3/#4 - 0.5*#6/#4 + 2*#3/#4, - #3 - 2*#3/#4) rectangle ++(#6/#4 - 4*#3/#4, -#7+#3);
	}
 \fi
}
 \foreach \y in {1} {
 	\ifodd\y{
 		\path[fill=red, fill opacity=0.25] (\y*#6/#4 - 4*\y*#3/#4 - 0.5*#5*#3/#4 - 0.5*#6/#4 + 2*#3/#4, #3 + 2*#3/#4) rectangle ++(#6/#4 - 4*#3/#4, #7-#3);
		\path[fill=red, fill opacity=0.25] (\y*#6/#4 - 4*\y*#3/#4 - 0.5*#5*#3/#4 - 0.5*#6/#4 + 2*#3/#4, - #3 - 2*#3/#4) rectangle ++(#6/#4 - 4*#3/#4, -#7+#3);
 }
 	\else{
		\path[fill=red, fill opacity=0.25] (\y*#6/#4 - 4*\y*#3/#4 - 0.5*#5*#3/#4 - 0.5*#6/#4 + 2*#3/#4, #3 + 2*#3/#4) rectangle ++(#6/#4 - 4*#3/#4, #7-#3);
		\path[fill=red, fill opacity=0.25] (\y*#6/#4 - 4*\y*#3/#4 - 0.5*#5*#3/#4 - 0.5*#6/#4 + 2*#3/#4, - #3 - 2*#3/#4) rectangle ++(#6/#4 - 4*#3/#4, -#7+#3);
	}
 \fi
}
 \foreach \y in {#4} {
 	\ifodd\y{
 		\path[fill=red, fill opacity=0.25] (\y*#6/#4 - 4*\y*#3/#4 - 0.5*#5*#3/#4 - 0.5*#6/#4 + 2*#3/#4, #3 + 2*#3/#4) rectangle ++(#6/#4 - 4*#3/#4, #7-#3);
		\path[fill=red, fill opacity=0.25] (\y*#6/#4 - 4*\y*#3/#4 - 0.5*#5*#3/#4 - 0.5*#6/#4 + 2*#3/#4, - #3 - 2*#3/#4) rectangle ++(#6/#4 - 4*#3/#4, -#7+#3);
 }
 	\else{
		\path[fill=red, fill opacity=0.25] (\y*#6/#4 - 4*\y*#3/#4 - 0.5*#5*#3/#4 - 0.5*#6/#4 + 2*#3/#4, #3 + 2*#3/#4) rectangle ++(#6/#4 - 4*#3/#4, #7-#3);
		\path[fill=red, fill opacity=0.25] (\y*#6/#4 - 4*\y*#3/#4 - 0.5*#5*#3/#4 - 0.5*#6/#4 + 2*#3/#4, - #3 - 2*#3/#4) rectangle ++(#6/#4 - 4*#3/#4, -#7+#3);
	}
 \fi
}
}

\newcommand{\BMPipeWithRectanglesThree}[8]{ 
 \BMPipe{#1}{#2}{#3}{#4}{#5}{#6}{#7}
 \foreach \y in { \number\numexpr#8+1\relax, ..., \number\numexpr#4-#8\relax} {
 	\ifodd\y{
 		\path[fill=teal, fill opacity=0.15] (\y*#6/#4 - 4*\y*#3/#4 - 0.5*#5*#3/#4 - 0.5*#6/#4 + 2*#3/#4, #3 + 2*#3/#4) rectangle ++(#6/#4 - 4*#3/#4, #7-#3);
		\path[fill=teal, fill opacity=0.15] (\y*#6/#4 - 4*\y*#3/#4 - 0.5*#5*#3/#4 - 0.5*#6/#4 + 2*#3/#4, - #3 - 2*#3/#4) rectangle ++(#6/#4 - 4*#3/#4, -#7+#3);
 }
 	\else{
		\path[fill=teal, fill opacity=0.25] (\y*#6/#4 - 4*\y*#3/#4 - 0.5*#5*#3/#4 - 0.5*#6/#4 + 2*#3/#4, #3 + 2*#3/#4) rectangle ++(#6/#4 - 4*#3/#4, #7-#3);
		\path[fill=teal, fill opacity=0.25] (\y*#6/#4 - 4*\y*#3/#4 - 0.5*#5*#3/#4 - 0.5*#6/#4 + 2*#3/#4, - #3 - 2*#3/#4) rectangle ++(#6/#4 - 4*#3/#4, -#7+#3);
	}
 \fi
}
}
\newcommand{\BMPipeWithRectanglesAndDissipativeSet}[8]{ 
 \BMPipeWithRectanglesThree{#1}{#2}{#3}{#4}{#5}{#6}{#7}{#8}
 \foreach \y in { \number\numexpr#8+1\relax, ..., \number\numexpr#4-#8\relax} {
 	\ifodd\y{
 		\path[fill=teal, fill opacity=0.15] (\y*#6/#4 - 4*\y*#3/#4 - 0.5*#5*#3/#4 - 0.5*#6/#4 + 2*#3/#4, #3 + 2*#3/#4) rectangle ++(#6/#4 - 4*#3/#4, #7-#3);
		\path[fill=teal, fill opacity=0.15] (\y*#6/#4 - 4*\y*#3/#4 - 0.5*#5*#3/#4 - 0.5*#6/#4 + 2*#3/#4, - #3 - 2*#3/#4) rectangle ++(#6/#4 - 4*#3/#4, -#7+#3);
 }
 	\else{
		\path[fill=teal, fill opacity=0.25] (\y*#6/#4 - 4*\y*#3/#4 - 0.5*#5*#3/#4 - 0.5*#6/#4 + 2*#3/#4, #3 + 2*#3/#4) rectangle ++(#6/#4 - 4*#3/#4, #7-#3);
		\path[fill=teal, fill opacity=0.25] (\y*#6/#4 - 4*\y*#3/#4 - 0.5*#5*#3/#4 - 0.5*#6/#4 + 2*#3/#4, - #3 - 2*#3/#4) rectangle ++(#6/#4 - 4*#3/#4, -#7+#3);
	}
 \fi
 \path[fill=red, fill opacity=0.75] (\y*#6/#4 - 4*\y*#3/#4 + 0.1*#6/#4 - 0.4*#3/#4, #3 + 2*#3/#4+0.11*#7-0.11*#3) rectangle ++(0.1*#6/#4 - 0.4*#3/#4, 0.22*#7-0.22*#3);
 \path[fill=red, fill opacity=0.75] (\y*#6/#4 - 4*\y*#3/#4 - 0.2*#6/#4 + 0.8*#3/#4 - #5*#3/#4, - #3 - 2*#3/#4 -0.11*#7+0.11*#3) rectangle ++(0.1*#6/#4 - 0.4*#3/#4, -0.22*#7+0.22*#3);
}
}

\newcommand{\BMPipeWithRectanglesAndDissipativeSetTwo}[7]{ 
 \BMPipe{#1}{#2}{#3}{#4}{#5}{#6}{#7}
 \path[fill=red, fill opacity=0.75] (0.11*#6, - 1.3*#3) rectangle ++(0.22*#6, -0.3*#3);
 \path[fill=blue, fill opacity=0.5] (-0.07*#6, - 0.5*#3) rectangle ++(0.45*#6, #3);
 \draw[red] (0.5, - 1.6*#3) node[anchor=north]{$D_{q, R}$};
}

\newcommand{\ArrowsStructure}[8]{ 
  \foreach \y in {1, ..., #4} {
	\draw[red, thick, ->] (#1 + #6 - 4*#3 - \y*#6/#4 + 4*\y*#3/#4 + #5*#3/#4, #2 - 0.5*#3/#4 + \y*#3/#4) -- (#1 + #6 - 4*#3 - \y*#6/#4 + 4*\y*#3/#4 + #5*#3/#4 + #8, #2 - 0.5*#3/#4 + \y*#3/#4);
	\draw[blue, thick, ->] (\y*#6/#4 - 4*\y*#3/#4 - 0.5*#5*#3/#4, #2 + #3 + 2*#3/#4 + 0.15*#7) -- (\y*#6/#4 - 4*\y*#3/#4 - 0.5*#5*#3/#4, #3 + 2*#3/#4 + 0.15*#7 + #8/#5) ;
	\draw[black, thick, dashed, <->] (#1 - #6/#4 + 4*#3/#4 + \y*#6/#4 - 4*\y*#3/#4, #2 + #3 + 2*#3/#4 + 0.4*#7) -- (#1 - #6/#4 + 4*#3/#4 + \y*#6/#4 - 4*\y*#3/#4 + #6/#4 - 4*#3/#4 - #5*#3/#4, #2 + #3 + 2*#3/#4 + 0.4*#7);
	\draw[black] (#1 - #6/#4 + 4*#3/#4 + \y*#6/#4 - 4*\y*#3/#4 + 0.5*#6/#4 - 2*#3/#4 - 0.5*#5*#3/#4, #2 + #3 + 2*#3/#4 + 0.4*#7) node[anchor=north]{$B^{\prime}$};
	\draw[black, thick, dashed, <->] (#1 - #6/#4 + 4*#3/#4 + \y*#6/#4 - 4*\y*#3/#4 + #6/#4 - 4*#3/#4 - #5*#3/#4, #2 + #3 + 2*#3/#4 + 0.5*#7) -- (#1 - #6/#4 + 4*#3/#4 + \y*#6/#4 - 4*\y*#3/#4 + #6/#4 - 4*#3/#4, #2 + #3 + 2*#3/#4 + 0.5*#7);
	\draw[black] (#1 - #6/#4 + 4*#3/#4 + \y*#6/#4 - 4*\y*#3/#4 + #6/#4 - 4*#3/#4 - 0.5*#5*#3/#4, #2 + #3 + 2*#3/#4 + 0.5*#7) node[anchor=north]{$A^{\prime}$};
}
 \draw[red, thick, ->] (#1 + #6 - 4*#3 + #5*#3/#4, #2 + #7 + #3 + #3/#4) -- (#1 + #6 - 4*#3 + #5*#3/#4 + #8, #2 + #7 + #3 + #3/#4);
 \draw[red, thick, ->] (#1 + #6 - 1.5*#3, #2 + #3 + 2*#3/#4 + 0.5*#7) -- (#1 + #6 - 1.5*#3, #2 + #3 + 2*#3/#4 + 0.5*#7 - #8);
 \draw[red, thick, ->] (#1 + #6, #2 + 0.5*#3) -- (#1 + #6 + #8, #2 + 0.5*#3);
 \draw[black, thick, dashed, <->] (#1 + #6 - 4*#3, #2 + #3 + 2*#3/#4 + 0.15*#7) -- (#1 + #6, #2 + #3 + 2*#3/#4 + 0.15*#7);
 \draw[black] (#1 + #6 - 3*#3, #2 + #3 + 2*#3/#4 + 0.15*#7) node[anchor=north]{\footnotesize $\frac{3A}{2} + \frac{B^{\prime}}{2} $};
}

\newcommand{\ArrowsStructureOnlyVelocity}[8]{ 
  \foreach \y in {1, ..., #4} {
	\draw[red, thick, ->] (#1 + #6 - 4*#3 - \y*#6/#4 + 4*\y*#3/#4 + #5*#3/#4, #2 - 0.5*#3/#4 + \y*#3/#4) -- (#1 + #6 - 4*#3 - \y*#6/#4 + 4*\y*#3/#4 + #5*#3/#4 + #8, #2 - 0.5*#3/#4 + \y*#3/#4);
	\draw[blue, thick, ->] (\y*#6/#4 - 4*\y*#3/#4 - 0.5*#5*#3/#4, #2 + #3 + 2*#3/#4 + 0.15*#7) -- (\y*#6/#4 - 4*\y*#3/#4 - 0.5*#5*#3/#4, #3 + 2*#3/#4 + 0.15*#7 + #8/#5) ;
}
 \draw[red, thick, ->] (#1 + #6 - 4*#3 + #5*#3/#4, #2 + #7 + #3 + #3/#4) -- (#1 + #6 - 4*#3 + #5*#3/#4 + #8, #2 + #7 + #3 + #3/#4);
 \draw[red, thick, ->] (#1 + #6 - 1.5*#3, #2 + #3 + 2*#3/#4 + 0.5*#7) -- (#1 + #6 - 1.5*#3, #2 + #3 + 2*#3/#4 + 0.5*#7 - #8);
 \draw[red, thick, ->] (#1 + #6, #2 + 0.5*#3) -- (#1 + #6 + #8, #2 + 0.5*#3);
}

\newcommand{\ArrowsStructureOnlyVelocityTwo}[8]{ 
  \foreach \y in {1, ..., #4} {
	\draw[red, thick, ->] (#1 + #6 - 4*#3 - \y*#6/#4 + 4*\y*#3/#4 + #5*#3/#4, #2 - 0.5*#3/#4 + \y*#3/#4) -- (#1 + #6 - 4*#3 - \y*#6/#4 + 4*\y*#3/#4 + #5*#3/#4 + #8, #2 - 0.5*#3/#4 + \y*#3/#4);
	\draw[blue, thick, ->] (\y*#6/#4 - 4*\y*#3/#4 - 0.5*#5*#3/#4, #2 + #3) -- (\y*#6/#4 - 4*\y*#3/#4 - 0.5*#5*#3/#4, #2 + #3 + #8/#5) ;
}
 \draw[red, thick, ->] (#1 + #6 - 4*#3 + #5*#3/#4, #2 + #7 + #3 + #3/#4) -- (#1 + #6 - 4*#3 + #5*#3/#4 + #8, #2 + #7 + #3 + #3/#4);
 \draw[red, thick, ->] (#1 + #6 - 1.5*#3, #2 + #3 + 2*#3/#4 + 0.5*#7) -- (#1 + #6 - 1.5*#3, #2 + #3 + 2*#3/#4 + 0.5*#7 - #8);
 \draw[red, thick, ->] (#1 + #6, #2 + 0.5*#3) -- (#1 + #6 + #8, #2 + 0.5*#3);
}

\newcommand{\ArrowsStructureOnlyVelocityTwoOnlySecondPart}[8]{ 
  \foreach \y in {1, ..., #4} {
	\draw[violet, thick, ->] (\y*#6/#4 - 4*\y*#3/#4 - 0.5*#5*#3/#4, #2 + #3) -- (\y*#6/#4 - 4*\y*#3/#4 - 0.5*#5*#3/#4, #2 + #3 + #8/#5) ;
}
}

\newcommand{\ArrowsStructureIntersection}[8]{ 
  \foreach \y in {1, ..., #4} {
	\draw[red, thick, ->] (#1 + #6 - 4*#3 - \y*#6/#4 + 4*\y*#3/#4 + #5*#3/#4, #2 - 0.5*#3/#4 + \y*#3/#4) -- (#1 + #6 - 4*#3 - \y*#6/#4 + 4*\y*#3/#4 + #5*#3/#4 + #8, #2 - 0.5*#3/#4 + \y*#3/#4);
	\draw[blue, thick, ->] (\y*#6/#4 - 4*\y*#3/#4 - 0.5*#5*#3/#4, #2 + #3 + 2*#3/#4 + 0.15*#7) -- (\y*#6/#4 - 4*\y*#3/#4 - 0.5*#5*#3/#4, #3 + 2*#3/#4 + 0.15*#7 + #8/#5) ;
	\draw[black, thick, dashed, <->] (#1 - #6/#4 + 4*#3/#4 + \y*#6/#4 - 4*\y*#3/#4, #2 + #3 + 2*#3/#4 + 0.4*#7) -- (#1 - #6/#4 + 4*#3/#4 + \y*#6/#4 - 4*\y*#3/#4 + #6/#4 - 4*#3/#4 - #5*#3/#4, #2 + #3 + 2*#3/#4 + 0.4*#7);
	\draw[black] (#1 - #6/#4 + 4*#3/#4 + \y*#6/#4 - 4*\y*#3/#4 + 0.5*#6/#4 - 2*#3/#4 - 0.5*#5*#3/#4, #2 + #3 + 2*#3/#4 + 0.4*#7) node[anchor=north]{$B_j$};
	\draw[black, thick, dashed, <->] (#1 - #6/#4 + 4*#3/#4 + \y*#6/#4 - 4*\y*#3/#4 + #6/#4 - 4*#3/#4 - #5*#3/#4, #2 + #3 + 2*#3/#4 + 0.5*#7) -- (#1 - #6/#4 + 4*#3/#4 + \y*#6/#4 - 4*\y*#3/#4 + #6/#4 - 4*#3/#4, #2 + #3 + 2*#3/#4 + 0.5*#7);
	\draw[black] (#1 - #6/#4 + 4*#3/#4 + \y*#6/#4 - 4*\y*#3/#4 + #6/#4 - 4*#3/#4 - 0.5*#5*#3/#4, #2 + #3 + 2*#3/#4 + 0.5*#7) node[anchor=north]{$A_j$};
}
 \draw[red, thick, ->] (#1 + #6 - 4*#3 + #5*#3/#4, #2 + #7 + #3 + #3/#4) -- (#1 + #6 - 4*#3 + #5*#3/#4 + #8, #2 + #7 + #3 + #3/#4);
 \draw[red, thick, ->] (#1 + #6 - 1.5*#3, #2 + #3 + 2*#3/#4 + 0.5*#7) -- (#1 + #6 - 1.5*#3, #2 + #3 + 2*#3/#4 + 0.5*#7 - #8);
 \draw[red, thick, ->] (#1 + #6, #2 + 0.5*#3) -- (#1 + #6 + #8, #2 + 0.5*#3);
 \draw[black, thick, dashed, <->] (#1 + #6 - 4*#3, #2 + #3 + 2*#3/#4 + 0.15*#7) -- (#1 + #6, #2 + #3 + 2*#3/#4 + 0.15*#7);
 \draw[black] (#1 + #6 - 3*#3, #2 + #3 + 2*#3/#4 + 0.15*#7) node[anchor=north]{$2 A_{j-1}$};
}

\newcommand{\BMPipeWithArrows}[8]{ 
 \BMPipe{#1}{#2}{#3}{#4}{#5}{#6}{#7}
 \ArrowsStructure{#1}{#2}{#3}{#4}{#5}{#6}{#7}{#8}
 \begin{scope}[yscale=-1,xscale=1, yshift=- 2*#2 cm]
  \ArrowsStructure{#1}{#2}{#3}{#4}{#5}{#6}{#7}{#8}
\end{scope}
}

\newcommand{\BMPipeWithArrowsIntersection}[8]{ 
 \BMPipe{#1}{#2}{#3}{#4}{#5}{#6}{#7}
 \ArrowsStructureIntersection{#1}{#2}{#3}{#4}{#5}{#6}{#7}{#8}
 \begin{scope}[yscale=-1,xscale=1, yshift=- 2*#2 cm]
  \ArrowsStructureIntersection{#1}{#2}{#3}{#4}{#5}{#6}{#7}{#8}
\end{scope}
}

\newcommand{\BMPipeWithRectanglesAndCollection}[7]{ 
 \BMPipe{#1}{#2}{#3}{#4}{#5}{#6}{#7}
 \foreach \y in {1, ..., #4} {
 	\ifodd\y{
 		\path[fill=teal, fill opacity=0.15] (\y*#6/#4 - 4*\y*#3/#4 - 0.5*#5*#3/#4 - 0.5*#6/#4 + 2*#3/#4, #3 + 2*#3/#4) rectangle ++(#6/#4 - 4*#3/#4, #7-#3);
		\path[fill=teal, fill opacity=0.15] (\y*#6/#4 - 4*\y*#3/#4 - 0.5*#5*#3/#4 - 0.5*#6/#4 + 2*#3/#4, - #3 - 2*#3/#4) rectangle ++(#6/#4 - 4*#3/#4, -#7+#3);
		\path[fill=blue, fill opacity=0.35] (\y*#6/#4 - 4*\y*#3/#4 - 1.5*#5*#3/#4 - 0.5*#6/#4 + 2*#3/#4, #3 + 2*#3/#4) rectangle ++(2*#5*#3/#4, #7-#3);
		\path[fill=blue, fill opacity=0.35] (\y*#6/#4 - 4*\y*#3/#4 - 1.5*#5*#3/#4 - 0.5*#6/#4 + 2*#3/#4, - #3 - 2*#3/#4) rectangle ++(2*#5*#3/#4, -#7+#3);
		\path[fill=blue, fill opacity=0.35] (\y*#6/#4 - 4*\y*#3/#4 - 0.5*#5*#3/#4 - #5*#3/#4, #3 + 2*#3/#4) rectangle ++(2*#5*#3/#4, #7-#3);
		\path[fill=blue, fill opacity=0.35] (\y*#6/#4 - 4*\y*#3/#4 - 0.5*#5*#3/#4 - #5*#3/#4, - #3 - 2*#3/#4) rectangle ++(2*#5*#3/#4, -#7+#3);

 }
 	\else{
		\path[fill=teal, fill opacity=0.25] (\y*#6/#4 - 4*\y*#3/#4 - 0.5*#5*#3/#4 - 0.5*#6/#4 + 2*#3/#4, #3 + 2*#3/#4) rectangle ++(#6/#4 - 4*#3/#4, #7-#3);
		\path[fill=teal, fill opacity=0.25] (\y*#6/#4 - 4*\y*#3/#4 - 0.5*#5*#3/#4 - 0.5*#6/#4 + 2*#3/#4, - #3 - 2*#3/#4) rectangle ++(#6/#4 - 4*#3/#4, -#7+#3);
		\path[fill=blue, fill opacity=0.35] (\y*#6/#4 - 4*\y*#3/#4 - 1.5*#5*#3/#4 - 0.5*#6/#4 + 2*#3/#4, #3 + 2*#3/#4) rectangle ++(2*#5*#3/#4, #7-#3);
		\path[fill=blue, fill opacity=0.35] (\y*#6/#4 - 4*\y*#3/#4 - 1.5*#5*#3/#4 - 0.5*#6/#4 + 2*#3/#4, - #3 - 2*#3/#4) rectangle ++(2*#5*#3/#4, -#7+#3);
		\path[fill=blue, fill opacity=0.35] (\y*#6/#4 - 4*\y*#3/#4 - 0.5*#5*#3/#4 - #5*#3/#4, #3 + 2*#3/#4) rectangle ++(2*#5*#3/#4, #7-#3);
		\path[fill=blue, fill opacity=0.35] (\y*#6/#4 - 4*\y*#3/#4 - 0.5*#5*#3/#4 - #5*#3/#4, - #3 - 2*#3/#4) rectangle ++(2*#5*#3/#4, -#7+#3);
	}
 \fi
}
\path[fill=blue, fill opacity=0.35] (#6 - 4*#3 + #6/#4 - 4*#3/#4 - 1.5*#5*#3/#4 - 0.5*#6/#4 + 2*#3/#4, #3 + 2*#3/#4) rectangle ++(2*#5*#3/#4, #7-#3);
	\path[fill=blue, fill opacity=0.35] (#6 - 4*#3 + #6/#4 - 4*#3/#4 - 1.5*#5*#3/#4 - 0.5*#6/#4 + 2*#3/#4, - #3 - 2*#3/#4) rectangle ++(2*#5*#3/#4, -#7+#3);
}


\section{Notation and Preliminaries}\label{sec:NotAndPrel}

\subsection{Notation}
We explain the notation used throughout the paper. 
The two-dimensional torus is denoted by $\T^2$. 
The Euclidean norm is denoted by $| \cdot |$
For any $\eps > 0$, the $\eps$-restriction of a set $S \subseteq \T^2$ is denoted and defined by
\[
 S[\eps] = \left\{ x \in \T^2 : \dist (x, S^c) > \eps \right\}
\]
where $S^c$ denotes the complement of $S$. Similarly the $\eps$-extension of a set $S \subseteq \T^2$ is denoted and defined by
{
\[
 I_\eps (S)= \left\{ x \in \T^2 : \dist (x, S) < \eps \right\}.
\]}
The cardinality of a set $A$ is denoted by $\# A$. 
For any $a,b \in \N$, we write
\[
 a + b \N = \{ n \in \N : \exists j \in \N \text{ with } n = a + bj \}.
\]
For vector-valued maps $F$, we denote the $i-th$ component by $F^{(i)}$.
The projection map into the $j$-th component will be denoted by $\pi_j$.
{We use $\wedge$ to denote the minimum and $\vee$ to denote the maximum:
\[
 a \wedge b = \min \{a,b\} \quad \text{and} \quad a \vee b = \max \{ a,b\}.
\]
Moreover, we introduce the following convention for the infimum and supremum of the empty set
\[
 \inf \emptyset = + \infty \quad \text{and} \quad \sup \emptyset = - \infty.
\]

\subsection{Preliminaries}
We introduce some preliminaries required for the remainder of the paper.
\begin{proposition} \label{prop:BCC}
Let $u \in L^\infty ((0,1) \times \T^2)$ be a divergence-free velocity field and $\theta_{\initial} \in L^\infty (\T^2)$, then for any $\kappa >0$ there exists a unique $\theta_{\kappa} \in L^\infty \cap L^2_t H^1_x$ solution of \eqref{adv-diff} and it satisfies 
$$\int_{\T^2} |\theta_{\kappa  }  (t,x)|^2 dx + 2 \kappa \int_0^t \int_{\T^2} | \nabla \theta_{\kappa } (s,x) |^2 dx ds  = \int_{\T^2} |\theta_{\initial  }  (x)|^2 dx  \,. $$
\end{proposition}
This proposition is a corollary of \cite[Theorem 1.1]{BCC23}.

\begin{lemma} \label{lemma:energy}
Let $u_1, u_2 \in L^\infty ((0,1) \times \T^2)$ be two divergence-free velocity fields and $\theta_{\initial } \in L^\infty(\T^2)$. Let $\theta_{\kappa,1}$ be the unique solution of \eqref{adv-diff} with $u_1$ and $\theta_{\kappa,2}$ be the unique solution of \eqref{adv-diff} with $u_2$, then
\begin{align*}
 \int_{\T^2} |\theta_{\kappa ,1  } (t,x) - \theta_{\kappa, 2} (t,x)|^2 dx & + \kappa \int_0^t \int_{\T^2} | \nabla (\theta_{\kappa, 1} - \theta_{\kappa, 2}) |^2  
\\
& \leq \frac{\| \theta_{\initial} \|_{L^\infty}^2 }{ \kappa} \int_0^t \int_{\T^2} |u_1 (s,x) - u_2 (s,x)|^2 dx ds
\end{align*} 
for any $t \in (0,1)$.
\end{lemma}
\begin{proof}
 
The following identity holds in the sense of distribution thanks to the regularity of the solutions
\begin{align} \label{eq:proof-energy}
\partial_t \frac{|\theta_{\kappa, 1} - \theta_{\kappa, 2}|^2}{2} & + u_1 \cdot \nabla (\theta_{\kappa,1} - \theta_{\kappa, 2}) (\theta_{\kappa, 1} - \theta_{\kappa, 2}) + (u_1 - u_2 ) \cdot \nabla \theta_{\kappa, 2} (\theta_{\kappa, 1} - \theta_{\kappa, 2})   \notag
\\
& = \kappa \Delta ( \theta_{\kappa,1 } - \theta_{\kappa, 2}) (\theta_{\kappa, 1} - \theta_{\kappa, 2 })
\end{align}
and therefore integrating in space-time and using the divergence-free condition on $u_1$ of the velocity fields and integration by parts we obtain
\begin{align*}
\int_{\T^2} |\theta_{\kappa, 1}  - \theta_{\kappa , 2} |^2 & + 2 \kappa \int_0^t \int_{\T^2} | \nabla (\theta_{\kappa, 1} - \theta_{\kappa, 2}) |^2 dx ds
\\
&  \leq 2 \int_0^t \int_{\T^2} |u_1 - u_2| |\nabla (\theta_{\kappa, 1} - \theta_{\kappa, 2} )| | \theta_{\kappa, 2}| dx ds.
\end{align*}
Using Young's inequality,  we can can bound the right hand side with 
\begin{align*}
 \kappa \int_0^t \int_{\T^2} | \nabla (\theta_{\kappa, 1} - \theta_{\kappa, 2}) |^2 dx ds + \frac{1}{ \kappa} \int_0^t \int_{\T^2} |u_1 - u_2|^2 |\theta_{\kappa, 2}|^2 dx ds \,,
\end{align*}
and using the maximum principle  on $\theta_{\kappa, 2}$ we conclude the proof.
\end{proof}
We now recall the Strong Markov property and Doob's Maximal inequality which are classical and can be found in \cite{OB}.

\begin{theorem}[Strong Markov property] \label{thm:strong-markov}
Let $\{W (t) \}_{t \geq 0}$ be a standard Brownian motion, and let $\tau$ be a finite stopping time relative to the standard filtration, with associated stopping $\sigma$-algebra $\mathcal{F}_\tau$ . For $t \geq 0$, define the post-$\tau$ process
$$W^\star (t) = W(t + \tau) - W(\tau)$$ 
and let $\{ \mathcal{F}^\star_t \}_{t \geq 0}$ be the standard filtration for this process. Then, conditional on the event $\{ \tau < \infty \}$, we have
\begin{itemize}
\item ${W^\star (t)}_{t \geq 0}$ is a standard Brownian motion; 
\item  for each $ t > 0$, the $\sigma$-algebra $\mathcal{F}^\star_t$ is independent of $\mathcal{F}_\tau$.
\end{itemize}
\end{theorem}

\begin{theorem}[Doob Maximal inequality] \label{lemma:maximal}
Let $(\Omega, (\mathcal{F}_t)_t, \mathbb{P})$ be a filtered probability space and $W_t$ an adapted $\T^d$-valued Brownian motion. Then, for every $c,\kappa>0, T> \tilde T \geq0$ we have
\begin{equation}\label{stima:brownian}
\mathbb{P} \left ( \omega \in \Omega : \sup_{t \in [\tilde{T} , T]}  \sqrt{2 \kappa} | W_t - W_{\tilde T}| \leq c \right ) \geq 1-  2 e^{ \sfrac{-c^2}{2 \kappa (T-\tilde T)}}.
\end{equation}
\end{theorem}

The following ergodic property of the Brownian motion will also be of crucial importance in the proof of Theorem~\ref{thm-main}, see for instance \cite[Theorem 3]{KR53} for a proof.
\begin{theorem}\label{thm:ergodic}
Let $f \in L^\infty (\R) \cap L^1(\R)$ with $\overline{f} = \int f(x) dx \neq 0$ and let $(W_t)_t$ be a standard one dimensional Brownian motion on a probability space $(\Omega, \mathcal{F}, \mathbb{P})$, then  for any $C\geq 0$ we have
$$\lim_{t \to \infty} \mathbb{P} \left ( \frac{1}{\overline{f} \sqrt{t}} \int_0^t f { (W_s) ds}   \leq C \right ) = \sqrt{\frac{2}{\pi}} \int_0^C \exp (- y^2/2) dy  \,.$$
\end{theorem}

We introduce the backward stochastic flow as well as the Feynman-Kac formula which gives a stochastic formulation of solutions to the advection-diffusion equation.

\begin{definition}[Backward stochastic flow] \label{d:backward}
Let $(W_t)_t$ be a two-dimensional backward Brownian motion process, $u \in L^\infty ((0,T) \times \T^2)$ be a two-dimensional autonomous divergence free velocity field,  then we define $X_{T,t}^{\kappa}$ the backward stochastic flow on $[0,T] \times \T^2$ of $u$ with noise constant $\sqrt{2 \kappa}$ as
\begin{align} \label{eq:backward-stoch-kappa}
\begin{cases}
d X_{T,t}^{\kappa} = u (X_{T,t}^\kappa) dt + \sqrt{2 \kappa} d W_t \,,
\\
{ X_{T,T}^\kappa} (x, \omega) \equiv x \,.
\end{cases}
\end{align}
\end{definition}
For the sake of completeness we denote $X_{T,t}$ the backward Lagrangian flow, solution to \eqref{eq:backward-stoch-kappa} with $\kappa =0$.

The following representation formula is well-known as Feynman-Kac formula, see for instance \cite{K84} for a proof.

\begin{theorem}[Feynman-Kac formula]
Let $u \in C^\infty ([0,1] \times \T^2; \R^2)$ be a  velocity field. Then the unique solution  $\theta_{\kappa}$ to the advection-diffusion equation\eqref{adv-diff}  with diffusivity parameter $\kappa \geq 0$ and smooth initial datum $\theta_{\initial} \in C^\infty (\T^2)$ can be expressed with the backward stochastic flow map $X_{t,s}^{\kappa}$ as
$$\theta_{\kappa} (x,t) = \mathbb{E} [\theta_{\initial} (X_{t,0}^\kappa (x, \cdot)) ].$$ 
\end{theorem}

Finally, we introduce the so-called fluctuation-dissipation formula. It was introduced in \cite{DE17,DE172,DE173} and has been used in \cite{DCZ} for enhanced dissipation.

\begin{lemma}[Fluctuation-Dissipation equality]\label{lemma:spont-anom}
Let $u \in C^\infty ([0,1] \times \T^2)$ be a divergence-free velocity field and $X_{t,0}^{\kappa}$ be the backward stochastic flow and $\theta_\kappa$ be the solution to \eqref{adv-diff} with a bounded initial datum $\theta_{\initial}$, then
\begin{align} \label{eq:FL-DISS}
\int_{\T^2} \mathbb{E} \left[ \Big| \theta_{\initial} (X_{t,0}^{\kappa} (x, \cdot )) -  \mathbb{E}[ \theta_{\initial} (X_{t,0}^{\kappa} (x, \cdot))] \Big|^2 \right] dx & = 2 \kappa \int_0^t \int_{\T^2} |\nabla \theta_\kappa (x,s)|^2 dx ds. \tag{FL-DISS}
\end{align}
\end{lemma}

\begin{proof}
We have the energy equality
\begin{equation}\label{eq:EnergyEquality}
\begin{split}
\int_{\T^2} |\theta_\kappa (x,t)|^2 dx + 2 \kappa \int_0^t \int_{\T^2} &| \nabla \theta_\kappa (x,s)|^2 dx ds = \int_{\T^2} |\theta_{\initial } (x)|^2 dx \\
&{\color{black} = \int_{\T^2} \mathbb{E} [|\theta_{\initial}(X_{t,0}^{\kappa}(x, \cdot))|^2] \, dx}
\end{split}
\end{equation}
where the last equality follows from the fact that $x \mapsto X_{t,0}^{\kappa}(x, \omega)$ is measure-preserving for all $\omega \in \Omega$. In addition, by the Feynman-Kac formula $|\theta_\kappa (x,t)|^2 = |\mathbb{E} [\theta_{\initial}(X_{t,0}^{\kappa}(x, \cdot))]|^2$. Plugging this into \eqref{eq:EnergyEquality} gives \eqref{eq:FL-DISS}.
\end{proof} 
  
 \section{Spontaneous stochasticity}\label{sec:spontaneous}
The proof of Theorem \ref{thm-main} strongly relies on the idea of strong separation of different realizations of the stochastic flow in finite time due to a chaotic behaviour of the velocity field. 
This phenomenon { resembles} spontaneous stochasticity. This section is devoted to clarifying the connection between our proof of Theorem \ref{thm-main} and spontaneous stochasticity.
 
Even though spontaneous stochasticity has been studied over the last few decades, see for instance \cite{BGK98,FGV01,CK03,DE17}, we could not find a precise definition of spontaneous stochasticity. Therefore, consistently with previous papers, we give Definition \ref{d:spontaneous-weak}. 
Then, we state a sufficient condition proving spontaneous stochasticity based on separation in finite time of different realizations of the stochastic flow, see Proposition \ref{prop:relation-spont}. This is essentially the condition we will prove to conclude Theorem \ref{thm-main}, see Proposition \ref{prop:criterion}.
\\
\\
Spontaneous stochasticity is a macroscopic effect  caused by  the explosive dispersion of particle pairs in a turbulent flow predicted by Richardson \cite{R26} with arbitrarily small stochastic perturbation. This phenomenon leads, in the zero noise limit, to non-uniqueness of Lagrangian trajectories. In \cite{BGK98,CK03,R23}, spontaneous stochasticity has been proved for the Kraichnan model \cite{K68}, where the
velocity is a Gaussian random field irregular in space and white-noise correlation in time, see \cite{FGV01} for a review.
Spontaneous stochasticity has been  further developed in \cite{DE17,DE172,DE173}, where the authors introduced the fluctuation dissipation formula and studied the relation between anomalous dissipation and spontaneous stochasticity both for passive and active scalars, as for instance Burgers equation.

Using the description given in \cite{BGK98,CK03,DE17}  we consistently define spontaneous stochasticity as follows.
Notice that we consider a continuous modification of the Brownian motion which is $C^{\beta }$ with $\beta < 1/2$ and  $X_{1,s}^{\kappa}$ is the path-wise unique solution (see for instance \cite{F15}) which is also continuous in time.

\begin{definition}[Spontaneous stochasticity] \label{d:spontaneous-weak}
Let $u \in L^\infty ((0,T) \times \T^d ; \R^d)$ be a divergence-free velocity field and let $(\Omega, \mathcal{F}, \mathbb{P})$ be a probability space and $X_{t,s}^\kappa$ be the backward stochastic flow. For any $x \in \T^d$ we have  $X_{T, \cdot }^{\kappa} (x, \cdot) : (\Omega, \mathbb{P}) \to ( C([0,T]; \T^d) , \eta_x^\kappa )$, where $\eta_x^\kappa = X_{T, \cdot }^{\kappa} (x, \cdot)_{\#} (\mathbb{P}) \in \mathcal{M} ( C([0,T]; \T^d) )$ is the pushforward of the backward stochastic flow  and finally we set $\eta^\kappa (d \gamma, d x) = \eta_x^\kappa (d \gamma) \otimes { \mathcal{L}^d ( dx)} \in \mathcal{M} (C([0,T]; \T^d) \times \T^d )$, namely 
\begin{equation}\label{eq:SpontStochFormulaDisintegration}
\int_{C([0,T]; \T^d) \times \T^d} f(\gamma, x)  \eta^\kappa (d \gamma, d x) = \int_{\T^d} \int_{C([0,T]; \T^d )} f(\gamma, x)  \eta^\kappa_x (d \gamma) d x \,,
\end{equation}
for any continuous  $f :  C([0,T]; \T^d) \times \T^d \to \R$. We say that $u$  exhibits spontaneous stochasticity if there exist $A \subset \T^d$ with $\mathcal{L}^d (A) >0 $ and a narrowly converging subsequence $\eta^{\kappa_q} \rightharpoonup \eta$  such that
$$\eta   =  \eta_x \otimes \mathcal{L}^d$$
and for all $x \in A$, $\eta_x$ is not a Dirac delta (meaning that there exists no continuous curve $\overline{\gamma}_x$ such that $\eta_x = \delta_{\overline{\gamma}_x}$).
\end{definition}

Related to the definition above, we point out the recent paper \cite{P24} where the author considers the inviscid case via regularisation by mollification of the velocity field.

\begin{lemma} \label{lemma:compactness}
Let $\eta^\kappa \in \mathcal{M} ( C([0,T]; \T^d) \times \T^d )$ defined as above  $\eta^\kappa (d \gamma, d x) = \eta_x^\kappa (d\gamma) \otimes \mathcal{L}^d (dx)$ where $ \eta_x^\kappa = X_{T, \cdot }^{\kappa} (x, \cdot)_{\#} (\mathbb{P}) \in \mathcal{M} ( C([0,T]; \T^d) )$. Then $\{ \eta^{\kappa} \}_{\kappa >0}$ admits a narrowly converging subsequence.
\end{lemma}

\begin{proof} 
This is an application of Prokhorov's Theorem, see for instance \cite[Theorem 5.1.3 and Remark 5.1.5]{AGS08}.  It is sufficient to find $F: C([0,T]; \T^d) \times \T^d  \to \R$ such that  the sublevel sets $\{ F \leq L \}$ are compact and 
$$\int_{C([0,T]; \T^d) \times \T^d} F d \eta^\kappa \leq 1 \qquad \forall \kappa>0 \,. $$
Let $\beta < 1/2$ and observe that
$$\mathbb{E} | X_{T,t}^{\kappa} (x, \omega) - X_{T,s}^{\kappa} (x, \omega) | \leq ( 2  \| u \|_{L^\infty} + \sqrt{2 \kappa} c ) |t-s|^{\beta} \leq C |t-s|^\beta $$
for any $|t-s| \leq 1$ and $\kappa \in (0,1)$ for some $C$ independent of $\kappa$. 
We define $ F(\gamma , x ) =  { C^{-1}} \| \gamma \|_{C^{\beta}} $
which is set $= + \infty$ if $\gamma \in C([0,T] ) \setminus C^\beta ([0,T])$. Thanks to Ascoli-Arzela's theorem we can show that the sublevel sets of the functional $F$ are compact. Finally,  by a direct computation we have 
$$\int_{C([0,T]; \T^d) \times \T^d} F(x, \gamma) \eta^\kappa (d \gamma, dx) = \dfrac{1}{C} \int_{\T^d} \int_{\Omega} \| X_{T, \cdot}^{\kappa} (x, \omega) \|_{C^\beta} d \mathbb{P}(\omega) dx \leq 1$$
for any $\kappa \in (0,1)$, concluding the proof.
\end{proof}

\begin{proposition}[Sufficient condition for spontaneous stochasticity]\label{prop:relation-spont}
Let  $d\geq 2$ and $u \in L^\infty ((0,T) \times \T^d ; \R^d)$ be a divergence free velocity field and  $X_{t,s}^{\kappa}$ be its backward stochastic flow as in Definition  \ref{d:backward}. Suppose that there exist a constant $c >0$ and $\{ \kappa_q \}_{q \in \N}$ with $\kappa_q \to 0$ as $q \to \infty$ such that 
\begin{align} \label{eq:spontaneous-strong}
\inf_{\kappa \in \{ \kappa_q\}_q }  \int_{\T^d} \, \, \mathbb{E} \left [  \left| X_{T, 0}^{\kappa} (x, \cdot )  - \mathbb{E} \left [  X_{T, 0}^{\kappa} (x, \cdot ) \right ] \right|^2 \right ] dx \geq c \,,
\end{align}
then $u$ exhibits spontaneous stochasticity.
\end{proposition}

\begin{proof}
 We fix a sequence $\{ \kappa_q \}_{q}$  for which the assumption holds. Suppose by contradiction that there is no spontaneous stochasticity, i.e. for any converging  subsequence $\eta^{\kappa_j} \rightharpoonup^* \eta$ we have $\eta (d \gamma, dx)= \delta_{\overline{\gamma}_x} (d \gamma) \otimes dx$.
 Let $\eps > 0$ be small and approximate the mapping $x \mapsto \overline{\gamma}_{x}(0)$ by a continuous function $c_{\eps} \colon \T^d \to \T^d$ such that $c_{\eps}(x) = \overline{\gamma}_{x}(0)$ except on a set of measure $\eps$. Note that $F \colon C([0,T]; \T^d) \times \T^d \to \R$ defined by $F(\gamma, x) = \dist(\gamma(0), c_{\eps}(x))^2$ is continuous and since $\eta^{\kappa_j} \rightharpoonup^* \eta$, we have
 \begin{equation}\label{eq:EquationWithF}
  \lim_{\kappa_q \to 0} \int_{\T^d} \int_{C([0,T]; \T^d)} F(\gamma, x) \, d \eta^{\kappa_q} = \int_{\T^d} \int_{C([0,T]; \T^d)} F(\gamma, x) \, d \eta.
 \end{equation}
 Since $\eta = \delta_{\overline{\gamma}_x} \otimes \Leb^d$ the right-hand side is bounded by $\eps$.
 Thus, by \eqref{eq:EquationWithF} and since $\eta^{\kappa_q} = X_{T, \cdot}(x, \cdot)_{\#} \mathbb{P}$, by taking $\kappa_q$ small enough 
  \begin{equation*}
   \int_{\T^d} \mathbb{E}[\dist(X^{\kappa_q}_{T,0}(x, \cdot), c_{\eps}(x))^2] \, dx \leq 2 \eps \,.
 \end{equation*}
Therefore, by Markov inequality
 we deduce that for $\kappa_q  $ small enough
  \begin{equation*}
  \int_{\T^d}  \dist(\mathbb{E}[X^{\kappa_q}_{T,0}(x, \cdot)] , c_\eps (x))^2 \, dx \leq C  \eps^{1/4} \,,
 \end{equation*}
 and hence, since $\eps$ is arbitrary, we have a contradiction.
\end{proof}


\section{A criterion for anomalous dissipation and continuous in time dissipation}\label{sec:AD-Crit}
In this section, we state and prove a criterion for anomalous dissipation which considers the stochastic trajectories of suitable approximations of the velocity field. The criterion essentially states that anomalous dissipation occurs if trajectories evaluated at time one (or any other fixed time for that matter) remain stochastic in the vanishing noise limit.
In addition, we prove that anomalous dissipation with autonomous velocity fields necessarily occurs continuously { in time} by which we mean that up to nonrelabeled subsequences
\[
 \kappa \int_{\T^d} |\nabla \theta_{\kappa}|^2 \, dx \overset{*}{\rightharpoonup} f \in L^\infty (0,T)
\]
 as $\kappa \to 0$, where the convergence is $L^\infty$ weak$*$.
\begin{definition}[Anomalous dissipation] \label{d:anomalous}
Let $u \in L^\infty ((0,T) \times \T^d; \R^d) $ be a divergence free velocity field. We say that $u$ exhibits anomalous dissipation if there exists $\theta_{\initial} \in C^\infty (\T^d)$ for which  the following holds
$$ \limsup_{\kappa \to 0} \kappa \int_0^T \int_{\T^d} | \nabla \theta_{\kappa} (x, s) |^2 dx ds >0 \,. $$
\end{definition}
The following is a new criterion to prove anomalous dissipation.
\begin{proposition}[Criterion for anomalous dissipation]\label{prop:criterion}
Let $\alpha \in (0,1)$ and  $u \in C^{\alpha} ((0,T) \times \T^d; \R^d)$ be a divergence free velocity field such that there exist $\{ u_q \}_{q \in \N} \subset C^\infty ( [0,T] \times \T^d; \R^d)$ divergence free and $\{ \kappa_q \}_{q \in \N}$  with the following properties
\begin{itemize}
\item $\frac{1}{\sqrt{\kappa_{q}}} \| u - u_{q}  \|_{L^2_{t,x}} \to 0 $ as $q \to \infty$,
\item we denote the stochastic backward flow  of $u_{q}$ with noise parameter $\sqrt{2 \kappa_{q}}$ as $ X_{t, 0}^{q, \kappa_{q}}$.  We suppose that there exist a constant $c>0$ and $\{ D_q \}_{q \in \N} \subset \T^d$ with $\mathcal{L}^d (D_q)\geq c$ for any $q \in \N$ with the following property 
\begin{align} \label{eq:spontaneous-stoch-prop}
 \inf_{q \in \N }  \inf_{x \in D_q} \mathbb{E} \left [  \left| X_{T, 0}^{q, \kappa_{q} } (x, \omega)  - \mathbb{E} \left [  X_{T, 0}^{q, \kappa_{q}} (x, \cdot ) \right ] \right|^2 \right ] \geq c \,.
\end{align}
\end{itemize}
Then $u $ exhibits anomalous dissipation. 
\end{proposition}
 
\begin{proof}
We suppose, without loss of generality, that the first component of $X_{T,0}^{q, \kappa_q} = ( x_{T,0}^{q, \kappa_{q}} , y_{T,0}^{q, \kappa_{q}})$ satisfies \eqref{eq:spontaneous-stoch-prop} with constant $\frac{c}{2}$, namely
\begin{align*}
 \inf_{q \in \N }  \inf_{x \in D_q} \mathbb{E} \left [  \left| x_{T, 0}^{q, \kappa_{q} } (x, \omega)  - \mathbb{E} \left [  x_{T, 0}^{q, \kappa_{q}} (x, \cdot ) \right ] \right|^2 \right ] \geq \frac{c}{2} \,,
\end{align*}
We now define  $\theta_{\initial } (x,  y) = x$ for any $x, y \in [0,1)$, where we slightly abuse the notation identifying $[0,1)$ with the one dimensional torus and let $\theta_{\kappa_{q}}^{(q)}$ be the solution of the advection diffusion equation with velocity field $u_q$ and diffusivity parameter $\kappa_{q}$. Applying Lemma \ref{lemma:spont-anom} we have
 \begin{align*}
 2 \kappa_{q} \int_0^T \int_{\T^d} |\nabla {\theta}_{\kappa_{q}}^{(q)} (x,s)|^2 dx ds & \geq \int_{\T^d} \mathbb{E} \left[ \Big| \theta_{\initial} (X_{T,0}^{q, \kappa_{q}} (x, \cdot )) -  \mathbb{E}[ \theta_{\initial} (X_{T,0}^{q, \kappa_{q}} (x, \cdot))] \Big|^2 \right] dx 
 \\
 & = \int_{\T^d} \mathbb{E} \left[ \Big|  x_{T,0}^{q, \kappa_{q}} (x, \cdot ) -  \mathbb{E}[ x_{T,0}^{q, \kappa_{q}} (x, \cdot) ] \Big|^2 \right] dx 
 \\
 & \geq \int_{D_q} \mathbb{E} \left[ \Big|  x_{T,0}^{q, \kappa_{q}} (x, \cdot ) -  \mathbb{E}[ x_{T,0}^{q, \kappa_{q}} (x, \cdot) ] \Big|^2 \right] dx 
 \\
 & \geq \frac{c}{2} \,, \qquad \forall q \in \N \,.
 \end{align*}
 Let $\theta_{\kappa_{q}}$ be the solution of the advection diffusion with velocity field $u$ and diffusivity parameter $\kappa_{q}$. Applying Proposition \ref{prop:BCC} and Lemma~\ref{lemma:energy} we finally have 
 $$\limsup_{q \to \infty} \kappa_{q} \int_0^T \int_{\T^d} |\nabla {\theta}_{\kappa_{q}} (x,s)|^2 \, dx ds = \limsup_{q \to \infty} \kappa_{q} \int_0^T \int_{\T^d} |\nabla {\theta}_{\kappa_{q}}^{(q)} (x,s)|^2 dx ds \geq \frac{c}{2} \,.$$
 Mollifying the initial datum $\theta_{\initial, \gamma} = \theta_{\initial} \star \varphi_{\gamma} $ with a parameter $\gamma >0$ so that 
 $$\| \theta_{\initial} - \theta_{\initial, \gamma} \|_{L^2} < c/2 $$
 we conclude that  anomalous dissipation holds for the advection diffusion equation with the smooth initial datum $\theta_{\initial, \gamma} \in C^\infty$.
\end{proof}

We define an anomalous dissipation measure $\mu \in \mathcal{M} (  (0,T) \times \T^d)$, if there exists a subsequence $\{ \kappa_q \}_{q \in \N}$ with $\kappa_q \to 0$ as $q \to \infty$ such that 
$$ \kappa_q  |\nabla \theta_{\kappa_q} (x,t)|^2 dx \rightharpoonup \mu \,, $$
where the convergence is weakly* in the sense of measures.
Due to the bound  $2\kappa \int_0^T \int_{\T^d} |\nabla \theta_{\kappa} (x,t)|^2 dx dt \leq \| \theta_{\initial} \|_{L^2 }^2 $, an anomalous dissipation measure always exists. We now show that for any autonomous divergence-free velocity field, the time marginal of any anomalous dissipation measure is absolutely continuous with respect to the one dimensional Lebesgue measure. In particular, the time marginal of any non-trivial anomalous dissipation measure with autonomous velocity fields must be supported on a subset of the time interval $(0,T)$ with full Lebesgue measure.

\begin{proposition} \label{prop:absolute}
 Let $d \geq 2$ and $u \in L^\infty (\T^d)$ an autonomous divergence free velocity field. Let $\theta_\kappa : [0,T] \times \T^d \to \R$ be the solution of the advection diffusion equation with $\theta_{\initial} \in C^2 (\T^d)$. Then 
 $$ \sup_{\kappa \in (0,1)} \| \partial_t \theta_\kappa \|_{L^\infty ((0,T) \times \T^d)) } \leq { \| u \|_{L^\infty} } \| \theta_{\initial} \|_{C^1} + \| \theta_{\initial} \|_{C^2} \,. $$
 In particular, up to not relabelled subsequence, 
 $$\kappa \int_{\T^d} |\nabla \theta_{\kappa} (x,\cdot)|^2 dx \overset{*}{\rightharpoonup} f \in L^\infty (0,T) $$
 as $\kappa \to 0$, where the convergence is $L^\infty$ weak$*$. Finally, if $\theta_\kappa \rightharpoonup^* \theta$ as $\kappa \to 0$ (up to not relabelled subsequence), then 
 $$ e(t) = \int_{\T^d} |\theta (x,t)|^2 dx \in W^{1, \infty} (0,T) \,.$$
 \end{proposition}
 
 \begin{proof}
 We notice that the solution $g_\kappa^{ \varepsilon}$
 \begin{align}  \label{ADV-DIFF-t}
\begin{cases} \tag{ADV-DIFF-t}
\partial_t  g_\kappa^{ \varepsilon} +  u_{\color{black} \varepsilon}  \cdot \nabla  g_\kappa^{\color{black} \varepsilon} = \kappa \Delta  g_\kappa^{\color{black} \varepsilon},
\\
 g_{\kappa}^{\color{black} \varepsilon} (0, \cdot ) = - u_{\color{black} \varepsilon} \cdot \nabla \theta_{\initial} (\cdot ) + \kappa \Delta \theta_{\initial} (\cdot ) \, . \notag
 \end{cases}
\end{align}
satisfies the maximum principle, from which we deduce that
$$ \| g_\kappa^{\color{black} \varepsilon} \|_{L^\infty ((0,T) \times \T^d ) }  \leq { \| u \|_{L^\infty} } \| \theta_{\initial} \|_{C^1} + \| \theta_{\initial} \|_{C^2}  \,.$$
 We claim that the last bound holds also for $\partial_t \theta_\kappa$. 
First, we introduce $\theta_\kappa^{\varepsilon}$ to be the solution to
\begin{align*} 
\begin{cases}
\partial_t  \theta^{\varepsilon}_\kappa +  u_\varepsilon  \cdot \nabla  \theta^\varepsilon_\kappa = \kappa \Delta  \theta^\varepsilon_\kappa,
\\
 \theta^\varepsilon_\kappa (0, \cdot ) = \theta_{\initial} (\cdot),
 \end{cases}
\end{align*}
with $u_\varepsilon = u \star \varphi_\varepsilon$ a space mollification. It is possible to prove that $ \theta_\kappa^\varepsilon$ is smooth  and $\partial_t \theta_\kappa^\varepsilon (t, \cdot) \to - u_{\color{black} \varepsilon} \cdot \nabla \theta_{\initial} (\cdot) + \kappa \Delta \theta_{\initial} (\cdot) $ in $L^2(\T^d)$ as $t \to 0^+$. Furthermore, from standard energy estimate we have 
$$ \sup_{t \in [0,1]}  \| \theta_\kappa (t, \cdot) - \theta_\kappa^\varepsilon (t, \cdot) \|_{L^2 (\T^d)} \leq \frac{\| u - u_\varepsilon \|_{L^2 ( (0,T) \times \T^d )}^2}{\kappa} \| \theta_{\initial} \|_{L^\infty}^2 \,,$$
and therefore $\theta_\kappa^\varepsilon \to \theta_\kappa$ in $C([0,T); L^2( \T^d))$ as $\varepsilon \to 0$.  Since $\partial_t \theta_\kappa^\varepsilon$ solves \eqref{ADV-DIFF-t} and it enjoys $\| \partial_t \theta_\kappa^\varepsilon \|_{L^\infty ((0,1) \times \T^d ) }  \leq { \| u \|_{L^\infty} } \| \theta_{\initial} \|_{C^1} + \| \theta_{\initial} \|_{C^2}$ which implies
\begin{align*}
& \theta_\kappa^\varepsilon \to \theta_\kappa \quad   C([0,T); L^2( \T^d))\,,
\\
& \partial_t \theta_\kappa^{\varepsilon} \overset{\star}{\rightharpoonup} \partial_t \theta_\kappa \quad \text{ weak } *-L^\infty \,.
\end{align*}
In particular, we have obtained the bound 
$$ \sup_{\kappa \in (0,1)}\| \partial_t \theta_\kappa\|_{L^\infty ((0,T) \times \T^d)} \leq { \| u \|_{L^\infty} } \| \theta_{\initial} \|_{C^1} + \| \theta_{\initial} \|_{C^2} \,.$$
 From this and  a standard energy balance  we observe that 
 \begin{align*}
  2 \kappa  \int_{\T^d} |\nabla \theta_{\kappa} (x,\cdot)|^2 dx & = - \frac{d}{dt} \int_{\T^d} |\theta_\kappa (x,t)|^2 dx = \int_{\T^d} \partial_t \theta_\kappa (x,t) \theta_\kappa (x,t)   dx 
  \\
  & \leq \| \partial_t \theta_\kappa \|_{L^\infty } \| \theta_\kappa \|_{L^\infty }   
 \end{align*}
and the last is bounded independently on $\kappa$ thanks to the previous computations.
Similarly, if up to not relabelled subsequences we have
\begin{align*}
\theta_\kappa \overset{*}{\rightharpoonup} \theta \in L^\infty \quad 
\partial_t \theta_\kappa \overset{*}{\rightharpoonup} \partial_t \theta \in L^\infty
\end{align*}
we  can bound
$$ \left | \frac{d}{dt} \int_{\T^d} |\theta (x,t)|^2 dx \right | = 2 \left | \int_{\T^d} \partial_t \theta (x,t ) \theta(x,t) dx \right | \leq  \| \partial_t \theta \|_{L^\infty } \| \theta \|_{L^\infty }  $$
where the last term is bounded since $\| \partial_t \theta \|_{L^\infty} \leq { \| u \|_{L^\infty} } \| \theta_{\initial} \|_{C^1} + \| \theta_{\initial} \|_{C^2} $ thanks to the weak* convergence.
 \end{proof}

\newcommand{\BMotion}[9]{
\draw[#6] (#1,#2)
\foreach \x in {1,...,#3}
{   -- ++(rand*#5 + #8,0.7*rand*#5 + #9)
}
node[above] {#7} circle(0.01) circle(0.02) circle(0.03) circle(0.04);
}

{
\section{Main ideas and overview of the proof}\label{sec:ideas}
In this section we explain the main ideas of the construction of the velocity field with its crucial property and provide an overview of the proof of the main result.

\subsection{Heuristic derivation of the velocity field $u$}
An informal key property leading to the construction of the velocity field $u$ of Theorem~\ref{thm-main} is that the forward flow pushes a small ball (coloured in green in Figure~\ref{fig:FirstGenWithPipes}) into a fat Cantor set.
However, to use this property rigorously in the proof of Theorem~\ref{thm-main}, we study the backward stochastic flow for which we prove \eqref{eq:SpontStochHeuristics}. 
 The fat Cantor set will be obtained as
\[
 C = \bigcap_{q \geq 1} C_q \,,
\]
where $C_{q+1} \subseteq C_q$ and the sets $C_q$ are closed. 
The small green ball and the set $C_1$ are depicted in Figure~\ref{fig:FirstGenWithPipes}.
 \begin{figure}[ht]
\begin{tikzpicture}
\small
\path[fill=green, fill opacity=0.6] (0.8,0) circle (0.4);
\RectanglesWithPipes{0}{0}{0.6}{4}{3}{10}{3}{3}
\draw[black, thick , dotted] (0,4.05) rectangle (10,-4.05);
\draw[black, thick, <->, dashed] (-0.2,-0.6) -- (-0.2,0.6);
\draw[black] (-0.2,0) node[anchor=east]{$A_0$};
\draw[black, thick, <->, dashed] (0,-4.2) -- (10,-4.2);
\draw[black] (5,-4.4) node[anchor=north]{$L_0$};
\draw[black, thick, <->, dashed] (10.2,-4.05) -- (10.2,4.05);
\draw[black] (10.2,-1.5) node[anchor=west]{$A_0 + B_0$};
\draw[black, thick, <->, dashed] (7.2,-0.02) -- (7.2,0.16);
\draw[black] (7.08, 0.28) node[anchor=west]{$\overline{A}_1$};
\draw[black, thick, <->, dashed] (0.2,0.44) -- (0.2,0.62);
\draw[black] (0.25,0.52) node[anchor=south]{$\overline{A}_1$};
\end{tikzpicture}
\centering
\caption{ The velocity field $b_1$.  The set $C_1$, in water green, is composed by $N_1 = 2 n_1 = 8$ rectangles of size $L_1 \times (A_1 + B_1)$.  }\label{fig:FirstGenWithPipes}
\end{figure}
The set $C_1$ is composed of $N_1$ rectangles of size $L_1 \times (A_1 + B_1)$.
These rectangles themselves form two rectangles. 
 \begin{figure}[ht]
\begin{tikzpicture}
\OnlyRectanglesTwoGen{0}{0}{0.6}{4}{3}{10}{3}{3}
\draw[black, thick , dotted] (0,4.05) rectangle (10,-4.05);
\end{tikzpicture}
\centering
\caption{ We draw in water green  the set $C_1$ and in blue the set $C_2$. Each small blue rectangle has size $L_2 \times (A_2 + B_2)$.}\label{fig:SeconGen}
\end{figure}
\begin{figure}[ht]
\begin{tikzpicture}
\path[fill=green, fill opacity=0.6] (0.8,0) circle (0.4);
\RectanglesTwoGenWithPipes{0}{0}{0.6}{4}{3}{10}{3}{3}
\draw[black, thick , dotted] (0,4.05) rectangle (10,-4.05);
\end{tikzpicture}
\centering
\caption{The velocity field $b_2$. This velocity field transports points from the green ball into $C_2$, coloured in blue.}\label{fig:SeconGenWithPipes}
\end{figure}
We say that each of these two rectangles contain $n_1$ smaller rectangles which means that $N_1 = 2n_1$. We define the velocity field $u$ as a limit of suitably mollified velocity fields $b_q \in L^\infty(\T^2) \cap BV (\T^2)$ which we now construct iteratively. 
 We define the velocity field $b_1$ so that the trajectories starting from  the green ball spreads  in finite time into a union of  rectangles called $C_1$, see Figure~\ref{fig:FirstGenWithPipes}. The pipe structure in this figure is the support of the velocity field $b_1$ and the pipe of width $A_0$ is the zero-generation pipe, while the smaller pipes of width $A_1$ are the first-generation pipes.
 Moreover, note that the parameters introduced so far are linked by the following two equations
\begin{align*}
 L_0 = n_1 B_1 + n_1 A_1 + 2 A_0\,, \quad
 \frac{B_0}{2} = L_1 + \text{lower order terms}  \,.
\end{align*}
We now construct the set $C_2$. Recall that $C_1$ is made up of $N_1$ rectangles. In each of these rectangles, we choose $2 n_2$ rectangles of size $L_2 \times (A_2 + B_2)$. All these rectangles then make up the set $C_2$ which is composed of $N_2 = 2n_2 \cdot N_1$ rectangles. This is depicted in Figure~\ref{fig:SeconGen}. 
To construct $b_2$ we now adapt the previously constructed velocity field $b_1$ by replacing the straight pipes in all the $N_1$ rectangles by a suitable rotated version  of the velocity field of Figure~\ref{fig:FirstGenWithPipes}. The  the new velocity field is depicted in Figure~\ref{fig:SeconGenWithPipes}. The smaller pipes in this figure represent the second generation of the pipes.
We then continue constructing velocity fields in the same way and the limit (up to suitable mollifications) will be the desired velocity field $u$ of Theorem~\ref{thm-main}.
We can now explain the meaning of the parameters involved in the construction:
\begin{itemize}
 \item $A_q$ denotes the width of the $q$-th generation pipes.
 \item $B_q$ denotes the distance between two adjacent $q$-th generation pipes.
 \item $L_q$ denotes the length of $q$-th generation rectangles. The width of these rectangles is $A_q + B_q$.
 \item $n_q$ denotes how many $q$-th generation pipes each $(q-1)$-the generation pipe splits into on each side of the $(q-1)$-th pipe.
 \item $v_q$ denotes the magnitude of the velocity field in any $q$-th generation pipe.
 \item $\overline{A}_q$ denotes the width of the $q$-th generation pipes before having become wider, i.e.
 \[
  \overline{A}_q = \dfrac{A_{q-1}}{2n_q}.
 \]  
 \item $N_q$ denotes how many $q$-th generation rectangles appear in the construction, which in particular is
 \[
  N_q = 2n_q \cdot N_{q-1} = \prod_{j = 1}^{q} (2 n_j).
 \]
\end{itemize}
These parameters are linked by the following two equations
\begin{align}
 L_q &= n_{q+1} B_{q+1} + n_{q+1} A_{q+1} + 2 A_q \,, \notag \\
 \frac{B_q}{2} &= L_{q+1} + \text{l.o.t.} \ll L_{q+1} +  \frac{1}{2^q}
 B_q \,. \label{eq:lot}
\end{align}
We highlight that the width of the $q$-th generation pipe $A_q$ is larger than $\overline{A}_q$ which is just the ratio between the width of the $(q-1)$-th pipe and the number of $q$-th generation pipes generated by a single $(q-1)$-th pipe. This is due to the widening procedure given by the enlarging pipes defined in Subsection~\ref{subsec:widening}. 
 The widening of the pipes is necessary to make the magnitude of the velocity field in the $q$-th generation pipe smaller as $q \to \infty$, in order to achieve the H\"older regularity of the velocity field after a suitable mollification procedure. We will denote by $u_q$ a suitable regularization of the velocity field $b_q$ described above. The mollification procedure is the main reason for some technical points in the proofs of Section~\ref{sec:stability}, such as the so-called {\em intersection problem}, which we explain in more detail in that section.
 \begin{remark}
 A particular choice of the l.o.t. implies \eqref{eq:disjoint-support-w}, which is needed to obtain the desired regularity of the forces in Theorem~\ref{thm:NS}. More precisely, this choice is made so that $\supp ( b - b_{q+2} )$ and  $\supp ( b_{q} )$ are suitably separated, so that after mollifying one can prove \eqref{eq:disjoint-support-w}.
\end{remark} 
{  \subsection{Hierarchy of parameters} 
In the next subsection we will introduce a super-exponential sequence $ a_q = a_0^{(1+\delta)^q}  \ll 1$ for which the following hierarchy  holds true 
\begin{align}\label{eq:hierarchy-parameters}
a_q^{\frac{1}{2} +} \approx \overline{A}_q \ll  A_q \ll B_q \ll L_q \ll v_q \approx a_q^{\frac{1}{2} -} \,.
\end{align}
We now explain the origin of these constraints.
To ensure that the set $\partial \{ u \equiv 0 \}$ has positive Lebesgue measure and that our velocity field does \emph{not} satisfy the weak Sard property, we require that the space between two consecutive pipes, denoted by \( B_q \), is larger than the width of the pipes. This gives the constraint:
\[
\overline{A}_q \leq A_q \leq B_q.
\] 

Moreover, due to the iterative nature of the construction, we also impose the relation \( L_q \approx B_{q-1} \). Since the sequences are decreasing, it follows that:
\[
B_q \leq L_q.
\]
Suppose we implement the construction without introducing the widening of the pipe--more precisely, without introducing the parameter $A_q$. In that case, incompressibility implies that the magnitude of the velocity field remains constant across scales, namely $v_q = v_0$ for any $q$. As a consequence, the resulting velocity field would not even be continuous; it would merely be bounded. To address this issue and aim for H\"older regularity, we introduce the widening of the pipes, i.e. the parameter $A_q$. Due to incompressibility, this enlargement {leads to smaller magnitude of the velocity field as scales become smaller} following the relation
\[
v_q = v_{q-1} \frac{\overline{A}_q}{A_q}.
\]
 {To reach an arbitrary but fixed H\"older regularity, the} idea is to {impose} the largest possible widening of the pipes compatible with  the construction
 \[
A_q \leq B_q \leq L_q \,.
\]
We sharpen these inequalities as much as possible in order to obtain the optimal H\"older regularity of the constructed velocity field via a suitable mollification scheme. The hierarchy of parameters in \eqref{eq:hierarchy-parameters} then follows as a direct consequence of the above relations.
}

\subsection{Insights into the proof of Theorem~\ref{thm-main}}\label{subsec:InsightsMainTheorem}
We present the main ideas behind the proof of Theorem~\ref{thm-main}.
 Due to Proposition~\ref{prop:criterion}, it is enough to find a subset $D_q \subseteq C_q$, which is introduced in Section \ref{sec:stability} for technical reasons, with $\inf_q \Leb^d(D_q)  > 0$ and such that
\begin{equation}\label{eq:SpontStochHeuristics}
 \inf_{q \in \N }  \inf_{x \in D_q} \mathbb{E} \left [  \left| X_{{ 1}, 0}^{\kappa_{q} } (x, \omega)  - \mathbb{E} \left [  X_{{ 1}, 0}^{\kappa_{q}} (x, \cdot ) \right ] \right|^2 \right ]  > 0\,.
\end{equation}
where $X^{\kappa_q}_{{ 1 ,t} }$ is the backward stochastic flow associated with $u_q$---which is an approximation of $u$ described in the previous subsection---and noise parameter $\sqrt{2 \kappa_q}$. 
 Note that the distance between $q$-th generation pipes is $B_q$. 
For this reason, we consider $\kappa_q \simeq B_q^2$.
{We will select $D_q$ to be a union of sets of width $\frac{\sqrt{\kappa_q}}{100}$ and suitable length located between $q$-th generation pipes so that $u_q \equiv 0$ on $D_q$, see the forthcoming Figure \ref{fig:TheSetsMWithDissipative} to visualize this set.}
In order to prove \eqref{eq:SpontStochHeuristics}, it is sufficient to show that there exist $c_1, c_2>0$ independent on $q$ such that for any $x \in D_q$  there are two sets $\Omega_1, \Omega_2 \subseteq \Omega$ such that  
\begin{align}
\min (\Prob(\Omega_1), \Prob(\Omega_2)) \geq c_1 > 0 \,, \label{eq:prob-uniform}
\\
X_{1,0}^{\kappa_q}(x, \omega) \in B_{c_2}(x) \quad \forall \omega \in \Omega_1\,, \label{eq:HeuristicsParticleStays} \\
X_{1,0}^{\kappa_q}(x, \omega) \not\in B_{2c_2}(x) \quad \forall \omega \in \Omega_2. \label{eq:HeuristicsParticleGoAway}
\end{align}
Since the variance of the Brownian motion at time $t$ is comparable to $2 \kappa_q t$ { and the velocity field $u_q \equiv 0$ on $B_{\sqrt{\kappa_q}/100} (x)$ for any $x \in D_q$ by construction}, it is clear that there exists a set $\Omega_1$ so that \eqref{eq:HeuristicsParticleStays} holds. Indeed, it suffices to take
\[
 \Omega_1 = \left\{ \omega \in \Omega : \sup_{t \in [0,{ 1}]} |W_t - W_1| \leq \dfrac{1}{200} \right\}
\]
to satisfy \eqref{eq:HeuristicsParticleStays}.
Similarly, due to the variance of the Brownian motion, the probability of $X_{1,t}(x, \omega)$ hitting the middle of the closest $q$-th generation pipe with some $t \geq \frac{3}{4}$ is bounded from below by a positive constant independent of $q$.   
Once this has occured, one can show that the backward stochastic flow will travel through the whole pipe structure with probability bounded from below by a strictly positive constant independent of $q$  using an ergodic property of the Brownian motion (see Theorem \ref{thm:ergodic} and Equation~\eqref{eq:SequenceOfIneqs}), and the stability results given in Section \ref{sec:stability}. 
The underlying heuristics of the proof of \eqref{eq:prob-uniform}, \eqref{eq:HeuristicsParticleStays} and \eqref{eq:HeuristicsParticleGoAway} which we have just described are depicted in Figure~\ref{fig:IdeaMainTheorem}.
\begin{figure}[h!]
\begin{tikzpicture}[scale=0.70]
\draw[black, thick] (-1,-1) -- (-1,3);
\draw[black, thick] (1,-1) -- (1,3);
\draw[black, thick, dotted] (0,-1) -- (0,3);
\path[fill=red, fill opacity=0.5] (-7,3) rectangle ++(3,-4);
\pgfmathsetseed{733}
\BMotion{-6}{0.5}{288}{0.02}{0.2}{blue, thick}{$\omega \in \Omega_2$}{0.013}{0.008}
\pgfmathsetseed{933}
\BMotion{-6}{0.5}{300}{0.02}{0.2}{darkgray, thick}{$\omega \in \Omega_1$}{0}{0}
\draw[black, thick, ->] (0.5,-0.3) -- (0.5,-0.9);
\draw[black, thick, ->] (-0.5,-0.3) -- (-0.5,-0.9);
\draw[black, thick, ->] (0.5,2.9) -- (0.5,2.3);
\draw[black, thick, ->] (-0.5,2.9) -- (-0.5,2.3);
\draw[black, thick, <->, dashed] (-1,-1.2) -- (1,-1.2);
\draw[black, thick, <->, dashed] (-4,-1.2) -- (-1,-1.2);
\draw[black, thick, <->, dashed] (-7,-1.2) -- (-4,-1.2);
\draw[black] (0,-1.2) node[anchor=north]{$A_q$};
\draw[black] (-2.5,-1.2) node[anchor=north]{$\frac{\sqrt{\kappa_q}}{100}$};
\draw[black] (-5.5,-1.2) node[anchor=north]{$\frac{\sqrt{\kappa_q}}{100}$};
\draw[black] (-7,-0.7) node[anchor=west]{$D_q$};
\draw[black] (0.75,-0.5) node[anchor=south]{$v_q$};
\end{tikzpicture}
\centering
\caption{The initial steps of the proof of Theorem~\ref{thm-main}. Note that for $\omega \in \Omega_2$ the backward stochastic trajectory is only depicted until it hits the middle of the pipe.
}\label{fig:IdeaMainTheorem}
\end{figure}
There are technical challenges which we have not described in this analysis. Those arise from the fact that with diffusivity parameter $\kappa_q \simeq B_q^2$, the solution to the advection-diffusion equation with velocity field $u_q$ is not a good approximation of the one with velocity $u$. Indeed, to apply Lemma \ref{lemma:energy} at the beginning of Section \ref{sec:ProofMainThm}, we need to use $u_{q+2}$ with diffusivity parameter $\kappa_q$.
{ To prove that the qualitative properties of \( u_{q+2} \) are the same as those of \( u_q \), we exploit the fact that  one  component of \( u_{q+2} - u_q \) is highly periodic and zero average inside the dissipative set $D_q$.
{ This is achieved through an It\^o--Tanaka-type trick in Section~\ref{sec:ProofMainThm}.}}
Here, we summarize the strategy to prove claims \eqref{eq:HeuristicsParticleStays} and \eqref{eq:HeuristicsParticleGoAway}.
\begin{table}[h!]
\begin{center}
\begin{tabular}{||c || c||} 
 \hline
 \multicolumn{2}{||c||}{\textbf{Backward stochastic trajectories}}
 \\
  \hline
 \textbf{remain close} & \textbf{are whisked away} \\ [0.5ex]
 $X_{1,0}^{\kappa_q}(x, \omega) \in B_{c_2}(x) \quad \forall \omega \in \Omega_1$ & $X_{1,0}^{\kappa_q}(x, \omega) \not\in B_{2 c_2}(x) \quad \forall \omega \in \Omega_2$ \\
 \hline\hline
 It\^o-Tanaka trick \eqref{eq:itotanaka} & Ergodic property: Theorem \ref{thm:ergodic}  \\ 
  & Stability results: Proposition \ref{lemma:stab-2}  \\ 
 \hline
\end{tabular}
\end{center}
\end{table}
}

\section{Choice of the parameters}\label{sec:choice}
In this section we define the parameters we will use in the construction of the velocity field. Let us fix $\alpha <1 $ as in the statement of Theorem \ref{thm-main}, then we can find $\delta , \eps >0 $ such that 
\begin{equation}\label{d:eps-delta-alpha}
 2 \eps + (1 + \delta)^8 \left( 3 \eps + 3 \delta + \frac{(1+\delta)^2}{2 + \delta} \right) < \min \left( \frac{1}{(2 + \delta) \alpha}, \frac{1 + (1+\delta)^4}{(1 + \alpha)(2 + \delta)}, \frac{1 + 2(1+\delta)^5}{(2 + \alpha)(2 + \delta)} \right)
\end{equation}
noticing that for $\eps = \delta = 0$ the {inequality reduces} to $\alpha <1$ and therefore (since the previous {inequality depends} continuously on $\eps, \delta$) we can find a ball of radius $r>0$, i.e. $B_r (0) = \{ (\eps, \delta) \in \R^2: |\eps| + |\delta| < r\}$, such that the inequalities are true. 
 We also require  that $\eps \ll \delta$, more precisely we require  
\begin{align}\label{d:eps-delta-1}
2 \delta^2 + \delta^3 \geq \eps (1+ \delta)^2 + 2 \eps \,.
\end{align}
The role of $\delta >0$ is to define the super-exponential sequence and $\varepsilon>0$ is a useful parameter needed to close some estimates.
Further, we assume that $\delta$ is a multiple of $\eps$.

We now inductively define the following sequence  of parameters, needed for the construction of the velocity field in Section \ref{sec:construction}.
\begin{subequations} \label{d:parameter}
\begin{align}
a_{q+1} &= a_q^{1+ \delta} \,, \label{d:a_q+1}
\\
n_{q+1} &= \frac{a_q}{2 a_{q+1}} \,,
\\
\overline{A}_{q+1} &= \frac{A_q}{2 n_{q+1}} \,, \label{d:overlineA_q+1}
\\
A_{q+1} &= \overline{ A}_{q+1} a_q^{\delta \eps - \frac{\delta}{2 + \delta}} \,, \label{d:A_q+1}
\\
{B}_{q+1} &= \frac{L_q - n_{q+1}A_{q+1} -2 A_q}{ n_{q+1}} \,, \label{d:B-q+1}
\\
L_{q+1} &= \frac{B_q}{2} - 4 \overline{A}_{q+1} - \frac{ B_{q}}{2} a_q^{\eps \delta} \,, \label{d:L-q+1} 
\\
v_{q+1} &= v_q \frac{\overline{A}_{q+1}}{A_{q+1}} \label{d:v-q+1}\, ,
\end{align}
\end{subequations} 
with initial conditions
{
\begin{align} \label{eq:par-initial}
 A_0 = a_0^2\,, \quad  v_0 =1/8 \,, \quad  L_0 =\frac{a_0^{\frac{2}{2+\delta}}}{2}  \,, \quad   B_0 = a_0\,.
\end{align}  }
 We now choose 
$a_0$ sufficiently small in terms of $\delta$ and $\eps$ so that 
\begin{align} \label{summability-eps}
\sum_{k =q}^\infty a_k^{\eps^2} \leq 2 a_q^{\eps^2} \qquad \forall q \in \N \,,
\end{align} 
which is verified if we choose $a_0$ such that  $a_0^{\eps^2 \delta} \leq 1/2$.
{ Further, we assume that $a_q^{- \eps}$ is a natural number for all $q$. This affects the proof only in terms of numerical constants.}
We finally define the sequence of diffusivity parameters  as
\begin{align}\label{d:parameter:diffusive}
\kappa_{q} = a_0^{\delta/2}B_{q}^2.
\end{align}

The following lemma gives us the asymptotic behaviour of the sequences defined above as $q \to \infty$.

\begin{lemma}\label{lemma:parameter-lemma}
Given the definition of the sequences in \eqref{d:parameter}, 
the following bounds hold true for every $q \geq 0$
\begin{subequations} \label{parameter}
\begin{align}
n_{q} = \frac{a_{q-1}^{- \delta}}{2 } \,, \label{parameter:n-q}
\\
N_q = a_0 a_q^{-1} \,, \label{parameter:N-q}
\\
A_{q} = a_0^{ \frac{3+ \delta}{2+\delta} - \eps}  a_{q}^{\eps + \frac{1+ \delta}{2+\delta}} \,, \label{parameter:A-q}
\\
\overline{A}_{q} = a_0^{ \frac{3+ \delta}{2+\delta} - \eps}  a_{q-1}^{\eps + \delta + \frac{1+ \delta}{2 + \delta}} \, ,\label{parameter:overlineA-q}
\\
\frac{1}{2}  a_{0}^{ \frac{1}{2+\delta}}  a_{q}^{\frac{1+ \delta}{2+\delta}} \leq {B}_{q}  \leq a_{0}^{ \frac{1}{2+\delta}} a_{q}^{\frac{1+ \delta}{2+\delta}} \,, \label{parameter:B_q}
\\
  \frac{1}{4} a_{0}^{ \frac{1}{2+\delta}}  a_{q}^{\frac{1}{2+\delta}}  \leq L_{q}  \leq \frac{1}{2} a_{0}^{ \frac{1}{2+\delta}}  a_{q}^{\frac{1}{2+\delta}} \,, \label{parameter:L_q}
\\
\frac{1}{4} a_{0}^{ \frac{2}{2+\delta} + \frac{\delta}{2}}  a_{q}^{\frac{2+ 2 \delta}{2+\delta}} \leq \kappa_{q}  \leq a_{0}^{ \frac{2}{2+\delta} + \frac{\delta}{2}}  a_{q}^{\frac{2+ 2 \delta}{2+\delta}} \,, \label{parameter:kappa-q}
\\
v_q = \frac{1}{8} a_0^{\eps - \frac{1}{2+ \delta}} a_q^{- \eps + \frac{1}{2+ \delta}} \, .\label{parameter:v-q}
\end{align}
\end{subequations} 
\end{lemma}
\begin{proof}
The equality \eqref{parameter:n-q} follows directly from the choice of parameters. Note that
\[
 N_q = \prod_{j =1}^q 2 n_j = \prod_{j=1}^q a_{j-1}^{- \delta} =a_0 a_q^{-1},
\] 
proving \eqref{parameter:N-q}, where in the last identity we used the relation $a_{j}= a_{0}^{(1+ \delta)^j}$ and an explicit computation using the power series identity. 
We now prove \eqref{parameter:A-q}. The equation clearly holds for $q = 0$. 
The case $q \geq 1$ follows by induction combined with \eqref{d:overlineA_q+1} and \eqref{d:A_q+1}.
Equation \eqref{parameter:overlineA-q} then follows from \eqref{parameter:n-q}, \eqref{parameter:A-q} and \eqref{d:A_q+1}.
We now prove estimate \eqref{parameter:B_q} by induction. 
We start by noticing that 
\eqref{parameter:B_q} 
holds for $q=0,1$ with a factor $(1-a_0)$ in place of $1/2$ in the lower bound. 
By \eqref{d:B-q+1},
\begin{align*}
B_{q+1} \leq \frac{L_q}{n_{q+1}} \leq \frac{B_{q-1}}{2 n_{q+1}}  = a_{q}^\delta B_{q-1}.
\end{align*}
Now, 
using \eqref{parameter:A-q}, we have
\begin{align*}
B_{q+1}
&= 2L_q a_q^{\delta} - 5 A_0 a_0^{-\eps - \frac{1 + \delta}{2 + \delta}} a_{q}^{\eps + \frac{1 + \delta}{2 + \delta}} a_{q}^{\eps \delta + \frac{\delta(1 + \delta)}{2 + \delta}}.
\end{align*}
Using \eqref{d:L-q+1} and \eqref{d:eps-delta-1} we continue
\begin{equation}
\begin{split}\label{eq:BoundOnBqPlusOne}
B_{q+1} &= B_{q-1} a_q^{\delta} (1- a_{q-1}^{\eps \delta}) - 8 a_q^{\delta} A_0 a_0^{- \eps - \frac{1 + \delta}{2 + \delta}} a_{q-1}^{\eps + \delta + \frac{1 + \delta}{2 + \delta}} - 5 a_q^{\delta} A_0 a_0^{-\eps - \frac{1 + \delta}{2 + \delta}} a_{q}^{\eps + \frac{1 + \delta}{2 + \delta}+ \eps \delta + \frac{\delta(1 + \delta)}{2 + \delta} - \delta} \\
&\geq a_q^{\delta} \left[ B_{q-1}  (1- a_{q-1}^{\eps \delta}) - 16 A_0 B_0^{-1} B_{q-1} a_0^{- \eps} a_{q-1}^{\eps + \delta} - 10 A_0 B_0^{-1} B_{q-1} a_0^{-\eps} a_{q-1}^{\eps (1 + \delta)^2} \right] \\
&\geq a_q^{\delta} B_{q-1}   \left( 1 - 2 a_{q-1}^{\eps \delta}  \right)
\end{split}
\end{equation}
where we used $40 a_0^{\eps \delta} \leq A_0^{-1} B_0$ in the last inequality.
We observe that, using $1- x \geq \e^{-2 x}$ for any $x \in (0, 1/2)$
\[
 \prod_{j = 1}^{\overline{q}} (1 - 2 a_j^{\eps \delta}) \geq  \prod_{j = 1}^{\overline{q}} \e^{- 4 a_j^{\eps \delta}} =  \e^{- 4 \sum_{j = 1}^{\overline{q}} a_j^{\eps \delta}} \geq \dfrac{3}{4} \quad \forall \overline{q} \geq 1.
\]
so that the factor $\left( 1 - 2 a_{q-1}^{\eps \delta}  \right)$ in \eqref{eq:BoundOnBqPlusOne} only results in a factor $3/4$ when iterated. Therefore, iterating the upper and lower bound on $B_q$ gives
\[
 \frac{3}{4} B_{q - 2 \lfloor q/2 \rfloor} \prod_{j = 1}^{\lfloor q/2 \rfloor} a_{q+1-2j}^{\delta} \leq B_q \leq B_{q - 2 \lfloor q/2 \rfloor} \prod_{j = 1}^{\lfloor q/2 \rfloor} a_{q+1-2j}^{\delta}.
\]
Observing that $B_{q - 2 \lfloor q/2 \rfloor}$ equals $B_0$ with $q$ even and equals $B_1$ with $q$ odd, combined with the fact that
\[
 \prod_{j = 1}^{\lfloor q/2 \rfloor} a_{q+1-2j}^{\delta} = a_{q - 2 \lfloor q/2 \rfloor}^{- \frac{1 + \delta}{2 + \delta}} a_q^{\frac{1 + \delta}{2 + \delta}}
\]
gives \eqref{parameter:B_q}.
Inequalities \eqref{parameter:L_q} follow from \eqref{d:B-q+1}, \eqref{d:L-q+1}, \eqref{parameter:n-q} and \eqref{parameter:B_q}.
Inequalities \eqref{parameter:kappa-q} follows immediately from \eqref{parameter:B_q} due to the definition of $\kappa_q$ in \eqref{d:parameter:diffusive}.
Finally, from \eqref{d:v-q+1}, \eqref{parameter:overlineA-q} and \eqref{parameter:A-q} we obtain
$$v_{q+1} = v_q a_q^{- \delta \eps + \frac{\delta}{2+\delta}}.$$ 
An explicit computation using that $a_{q}= a_0^{(1+\delta)^q}$ yields \eqref{parameter:v-q}.
\end{proof}


 \section{Construction of the velocity field} \label{sec:construction}
 { 
In this section, we construct the velocity field $u$ from Theorem~\ref{thm-main} as the limit of a sequence of smooth velocity fields $\{ u_q \}_q$, following a multi-step process. We begin by defining a velocity field supported on a pipe that rotates trajectories by $\pi/2$ while decreasing the $L^\infty$ norm - a key step in achieving the optimal $C^\alpha$ regularity. With this in place, we firstly construct the  building block, which we then use to define a sequence of $L^\infty$ divergence-free velocity fields $\{ b_q \}_q$. Finally, we obtain the regularized sequence $\{ u_q \}_q$ through a mollification procedure.

\subsection{Twice rotating and enlarging pipe} \label{subsec:widening}
We define the {\em enlarging pipe} velocity field $\bar W \in BV (\T^2)$. The key properties of this velocity field is that it is divergence-free, autonomous  and allows one to trade bigger spacial support for a strictly decreasing $L^\infty$ norm. For any $v>0$ and $\lambda >1$ we define 
\begin{align}
\bar W (x,y) = \begin{cases}
(v,0)  & \text{ if } y \in [r, R] \text{ and } y \geq \frac{x}{\lambda} 
\\
(0, { -} \frac{v}{\lambda})  & \text{ if } x \in [\lambda r, \lambda R] \text{ and } y < \frac{x}{\lambda} 
\\
(0, 0)  & \text{ otherwise}\,.
\end{cases}
\end{align}

The velocity field enjoys the following property
\begin{align} \label{divergence:hausdorff}
\divergence \bar W = v \Haus^1 |_{\{ 0 \} \times [r,R]} - \frac{v}{\lambda} \Haus^1 |_{ [\lambda r, \lambda R] \times \{ 0 \} }. 
\end{align}
\begin{figure}[ht]
\begin{tikzpicture}
\small
\draw[black, thick , dotted] (3,0) rectangle (7,-2);
\draw[black, thick, <->, dashed] (2.75,-1) -- (2.75, -2);
\draw[black] (2.75,-1.5) node[anchor=east]{$r$};
\draw[black, thick, <->, dashed] (2.25,-0) -- (2.25, -2);
\draw[black] (2.25,-1) node[anchor=east]{$R$};
\draw[black, thick, <->, dashed] (3,-2.25) -- (5, -2.25);
\draw[black] (4,-2.25) node[anchor=north]{$\lambda r$};
\draw[black, thick, <->, dashed] (3,-2.75) -- (7, -2.75);
\draw[black] (5,-2.75) node[anchor=north]{$\lambda R$};
\draw[black, thick, -] (3,-1.00) -- (5, -1.00) -- (5, -2.00);
\draw[black, thick, -] (3,-0.00) -- (7, -0.00) -- (7, -2.00);
\draw[black, thick, dotted] (5, -1.00) -- (7, -0.00);
\draw[black, thick, ->] (3.0,-0.5) -- (3.6, -0.5);
\draw[black, thick, ->] (3.0,-0.25) -- (3.6, -0.25);
\draw[black, thick, ->] (3.0,-0.75) -- (3.6, -0.75);
\draw[black, thick, ->] (6,-2) -- (6, -2.3);
\draw[black, thick, ->] (6.5,-2) -- (6.5, -2.3);
\draw[black, thick, ->] (5.5,-2) -- (5.5, -2.3);
\draw[black] (6.9,-2.0) node[anchor=north]{$v / \lambda$};
\draw[black] (3.6,-0.5) node[anchor=west]{$v$};
\end{tikzpicture}
\centering
\caption{An enlarging pipe.}\label{fig:RotatingPipe}
\end{figure}
}
It is now possible to construct, thanks to the previous construction of an {\em enlarging} pipe, a pipe with a velocity field which starts out horizontal with magnitude $v$ and the width of the pipe $\widehat{A}$. 
Then, it turns and becomes a vertical  pipe with width $A^{\prime}$ and the magnitude of the velocity field is $v^{\prime}$. To ensure the divergence-free condition we have the relation $v^\prime = v \frac{\widehat{A}}{A^\prime}$.
 Afterwards, the pipe again turns and becomes horizontal with width $\widehat{A}$ and the magnitude of the velocity field is again $v$ (see Figure~\ref{fig:DoubleRotatingPipe}).  
\begin{figure}[h!]
\begin{tikzpicture}[scale= 0.5]
\small
\draw[black, thick , dotted] (0,0) rectangle (12,-6);
\draw[black, thick] (0,0) -- (7,0) -- (7,-5) -- (12,-5);
\draw[black, thick] (0,-1) -- (5,-1) -- (5,-6) -- (12,-6);
\draw[black, thick, <->, dashed] (2,0) -- (2,-1);
\draw[black, thick, <->, dashed] (5,-3.5) -- (7,-3.5);
\draw[black, thick, ->] (0,-0.25) -- (1,-0.25);
\draw[black, thick, ->] (0,-0.5) -- (1,-0.5);
\draw[black, thick, ->] (0,-0.75) -- (1,-0.75);
\draw[black, thick, ->] (6,-2.5) -- (6,-3.0);
\draw[black, thick, ->] (6.5,-2.5) -- (6.5,-3.0);
\draw[black, thick, ->] (5.5,-2.5) -- (5.5,-3.0);
\draw[black, thick, ->] (12,-5.25) -- (13,-5.25);
\draw[black, thick, ->] (12,-5.5) -- (13,-5.5);
\draw[black, thick, ->] (12,-5.75) -- (13,-5.75);
\draw[black, thick, <->, dashed] (11,-5) -- (11,-6);
\draw[black] (2,-0.5) node[anchor=west]{$\widehat{A}$};
\draw[black] (1,-0.5) node[anchor=west]{$v$};
\draw[black] (11,-5.5) node[anchor=west]{$\widehat{A}$};
\draw[black] (13,-5.5) node[anchor=west]{$v$};
\draw[black] (6,-3.5) node[anchor=north]{$A^{\prime}$};
\draw[black] (6,-1.3) node[anchor=north]{$v^{\prime}$};
\end{tikzpicture}
\centering
\caption{A twice rotating and enlarging pipe.}\label{fig:DoubleRotatingPipe}
\end{figure}

\subsection{Building blocks: Branching-merging pipe}\label{subsec:BranchingMerging}
In this subsection, we will use the velocity field of the previous subsection in order to build a velocity field whose fundamental property is that it starts out as a unique pipe of width $A$ and velocity $v$ and then branches into $2n$ (with $n$ large) smaller pipes of width 
$$\widehat{A} = \frac{A}{2n}$$ 
 and each of this pipe width enlarges to $A^{\prime}$ but with reduced intensity of the velocity which is given by 
 $$v^{\prime} = v \frac{\widehat{A}}{A^{\prime}} \ll v \,.$$
  Afterwards, all these smaller pipes shrink back into small pipes of width $\widehat{A}$ with corresponding velocity $v$ and  merge back into one pipe of width $A$ as in Figure~\ref{fig:BranchingMergingNew}.
\begin{figure}[ht]
\begin{tikzpicture}[scale=0.75]
\small
\BMPipeWithArrows{0}{0}{0.8}{5}{3}{15}{3}{1}
\draw[black, thick , dotted] (0,4.3) rectangle (15,-4.3);
\draw[black, thick, dashed, <->] (0, 4.5) -- (15,4.5);
\draw[black] (7.5, 4.5) node[anchor=south]{$L$};
\draw[black, thick, dashed, <->] (-0.5, 0.8) -- (-0.5, 4.3);
\draw[black] (-0.5, 2.5) node[anchor=east]{$\frac{B}{2}$};
\draw[black, thick, dashed, <->] (-0.5, -0.8) -- (-0.5, -4.3);
\draw[black] (-0.5, -2.5) node[anchor=east]{$\frac{B}{2}$};
\draw[black, thick, dashed, <->] (-0.25, 0.8) -- (-0.25, -0.8);
\draw[black] (-0.25, 0) node[anchor=east]{$A$};
\path[fill=white, fill opacity=1] (0.7,1.75) rectangle ++(0.5,0.5);
\path[fill=white, fill opacity=1] (0.7,-2.4) rectangle ++(0.5,-0.5);
\draw[black] (1.0,2.25) node[anchor=north]{\footnotesize $\frac{A}{2} + \frac{B^{\prime}}{2}$};
\draw[black] (1.0,-2.4) node[anchor=north]{\footnotesize $\frac{A}{2} + \frac{B^{\prime}}{2}$};
\end{tikzpicture}
\centering
\caption{This figure shows a branching-merging pipe $W_{L, A, B, A^{\prime}, B^{\prime}, 5, v}$. One-ended arrows represent the velocity field: red ones are of magnitude $v$ and blue ones are of magnitude $v^{\prime} = \frac{A}{2 n A^{\prime}} v = \frac{A}{10 A^{\prime}} v$. Dashed double-ended arrows represent distances which are written next to the arrow in question.}\label{fig:BranchingMergingNew}
\end{figure}
The quantity $L$ corresponds to the length of the whole construction, $(A+ B)$ corresponds to the width of the whole construction
and $B^{\prime}$ corresponds to the distance in between smaller pipes. Note that these parameters are linked by the following relation:
\[
 L = n A^{\prime} + n B^{\prime} + 2 A.
\]
Moreover, note that the length of the pipes { with width $A'$} is given in terms of $B$ and $\widehat{A}$ as
\[
 \dfrac{B - 3 \widehat{A}}{2}.
\]
This velocity field is our building block, needed in Subsection~\ref{subsec:LInftyVelocityFields}, and we denote it as
\[
 W = W_{L, A, B, A^{\prime}, B^{\prime}, n, v}.
\]
We will regularly use rotated (by some multiple of $\sfrac{\pi}{2}$) versions of this velocity field. This will always be clear from the context and the properties we would like to achieve, hence we will slightly abuse the notation using always the same notation independently of whether the velocity field is rotated or not. With the twice rotating enlarging pipe given in Section \ref{subsec:widening} we can construct our building block $W$. 
The key properties of $W$ are stated in the following lemma.
\begin{lemma}\label{lemma:PropertiesBranchingMerging}
Let $n \geq 10$ be an integer and let $A, B, L, A^{\prime}, B^{\prime}, v \in (0,1)$ be real numbers such that
\begin{equation}
 A \ll B, \quad A^{\prime} \ll B^{\prime}, \quad \text{and} \quad L = n A^{\prime} + n B^{\prime} + 2 A.
\end{equation}
Let $L^{\prime} \in (0,1)$ be an any real number such that
\begin{equation}
L^{\prime} < \frac{B - 3 \widehat{A} - A}{2}.
\end{equation}
Then there exists $W = W_{L, A, B, A^{\prime}, B^{\prime}, n, v} : R = [0,L] \times [- \frac{A+ B}{2}, \frac{A+ B}{2}] \to \R^2$ with the following properties:
\begin{enumerate}
\item $\| W \|_{L^\infty} = v$.
\item $\divergence W = \divergence (W \mathbbm{1}_{R}) = v \Haus^{1} |_{\{ 0 \} \times \left[ - \frac{A}{2}, \frac{A}{2} \right]} - v \Haus^{1} |_{\{ L \} \times \left[ - \frac{A}{2}, \frac{A}{2} \right]}$.
\item There are  $2n$ rectangles $\{ R_i \}_{i =1}^{2n}$ such that, up to a rotation of $\pi/2$ or $-\pi / 2$ and a translation,  
$$R_i = [0, L^\prime] \times \left[ - \frac{A^\prime + B^\prime}{2}, \frac{A^\prime + B^\prime}{2} \right]$$ 
and in the coordinates of $R_i$
$$ W (x, y)  = \begin{pmatrix}
    v^{\prime} \mathbbm{1}_{[- \frac{A^\prime }{2}, \frac{A^\prime}{2}]}(y) \\
    0
   \end{pmatrix}  \qquad \forall (x, y) \in R_i \,,$$
   where $v^\prime = v  \frac{\widehat{A}}{A^\prime}$.
   \item It holds that $ \| W \|_{L^\infty (R_i)} \leq v^{\prime}$ and in the coordinates of $R_i$ 
   $$\divergence (W \mathbbm{1}_{R_i}) =  v^\prime \Haus^{1} |_{\{ 0 \} \times \left[ - \frac{A^\prime}{2}, \frac{A^\prime}{2} \right]} - v^\prime \Haus^{1} |_{\{ L^\prime \} \times \left[ - \frac{A^\prime}{2}, \frac{A^\prime}{2} \right]}$$ for any $i =1 , ... , 2n$.
   \item The collection of rectangles $\{ R_i \}_{i = 1}^{2n}$ can be reordered so that
   \[
    R_i = R_1 + (i-1)(A^{\prime} + B^{\prime}, 0) \quad R_{i+n} = R_{n+1} + (i-1)(A^{\prime} + B^{\prime}, 0) \quad \forall i = 1, \ldots, n.
   \]
   \item \label{item:odd}  { The second component of the velocity field is odd with respect to $y$, namely $W^{(2)} (x,y) = - W^{(2)} (x, -y)$. In particular, it holds that 
   $$ \int_{- \frac{A + B}{2}}^{\frac{A + B}{2}} W^{(2)}(x,y) \, { dy}=0 \quad \forall x \in [0, L] \,. $$}
\end{enumerate}
\end{lemma}

\subsection{Construction of the $L^{\infty}$ velocity fields}\label{subsec:LInftyVelocityFields}
In this subsection, we use the velocity field from the previous subsection as a building block in the construction of a sequence of velocity fields $\{ b_q \}_{q \geq 1}$. The sequence is constructed iteratively by adding more and more suitably rescaled copies of branching-merging pipes at each step. To do this, we simultaneously define a sequence of sets of rectangles $ \{ \mathcal{R}_q \}_{q}$, where 
$ \mathcal{R}_q = \{ R_1, \ldots , R_{N_q} \}$ and $R_i$ is a rectangle for each $i = 1, \ldots, N_q$. Given  the velocity field  $b_q$ and the collection of mutually disjoint rectangles $\mathcal{R}_q$, we define $b_{q+1}$ as $b_q$, redefining it on $R_i$ as a branching merging pipe keeping the divergence free property.
For the construction we identify $\T^2$ with $[0,1]^2$.
 Recall the parameters selected in Section~\ref{sec:choice}, that satisfy \eqref{parameter} and  will be the parameters of our construction. Firstly, we define { a zero-average velocity field} $b_0  \colon \T^2 \to \R^2$ such that for all $(x,y) \in (A_0, 1 - A_0)^2$
 \begin{align} \label{d:b-0}
 b_0(x,y) = 
 \begin{cases}
  (v_0, 0) &\text{if $|y - \frac{1}{2}| < \frac{A_0}{2}$;} \\
  (0,0) &\text{otherwise.}
 \end{cases}
 \end{align}
In other words, $b_0$ is a straight pipe of velocity $v_0$ and width $A_0$ inside the square $(A_0, 1 - A_0)^2$.  We now define the set of rectangles at step 0 as $\mathcal{R}_0 = \{ [\frac{2}{3}, \frac{2}{3} + L_0] \times [- \frac{A_0 + B_0}{2}, \frac{A_0 + B_0}{2}] \} $.
 We now define $b_1$ as $b_0$ and redefine it on each rectangle of the collection $\mathcal{R}_0$   to be a branching-merging constructed in Subsection~\ref{subsec:BranchingMerging} with the parameters $L_0, A_0, B_0, A_1, B_1, n_1, v_0$ from Section~\ref{sec:choice}.  In particular, on each rectangle of the collection $\mathcal{R}_0  $ (which in this case has a single element) we replace the straight pipe given by $b_0 $ with a consistent branching merging pipe.
Then we note that for $b_1$ there exists a collection of $2 n_1 = N_1$ rectangles $\{ R_i \}_{i = 1}^{N_1}$ so that up to rotations and translations
\[
 R_i = [0, L_1] \times \left[- \frac{A_1}{2} - \frac{B_1}{2}, \frac{A_1}{2} + \frac{B_1}{2} \right]
\]
and in these rectangles the velocity field $b_1$ is composed by a single straight pipe of width $A_1$.
 Then we define this new collection of rectangles $\mathcal{R}_1 = \{ R_1, \ldots , R_{N_1 } \}$. 
To construct $b_{q+1}$ from $b_q$ for any $q$, the procedure is as follows. Assume that there exists a collection of rectangles $\mathcal{R}_q$ such that $\# \mathcal{R}_q = N_q$ and for any $R \in \mathcal{R}_q$ it holds that up to a rotation and a translation
\[
 R = [0, L_q] \times \left[ - \frac{A_q + B_q}{2}, \frac{A_q + B_q}{2} \right]
\]
and in the coordinates of $R$
\begin{equation} \label{d:b_q}
 b_q(x,y) = 
 \begin{pmatrix}
    v_q \mathbbm{1}_{[- \frac{A_q }{2}, \frac{A_q}{2}]}(y) \\
    0
 \end{pmatrix}
 \quad
 \forall (x,y) \in R.
\end{equation}
Now we define $b_{q+1}$ by redefining the velocity field $b_q$ on each $R \in \mathcal{R}_q$. Precisely we define $b_{q+1}$  to be a branching-merging pipe with parameters $L_q, A_q, B_q, A_{q+1}, B_{q+1}, n_{q+1}, v_q$ i.e.
$$b_{q+1} (x,y) = W_{L_q, A_q, B_q, A_{q+1}, B_{q+1}, n_{q+1}, v_q}(x,y) $$
 \begin{figure}[ht]
\begin{tikzpicture}[scale=0.6]
\small
\BMPipeWithRectangles{0}{0}{0.8}{6}{4}{15}{3}{3}
\draw[black, thick , dotted] (0,4.2) rectangle (15,-4.2);
\draw[black] (15,3.8) node[anchor=west]{$R \in \mathcal{R}_q$};
\end{tikzpicture}
\centering
\caption{The support of the velocity field $b_{q+1}$  on $R \in \mathcal{R}_q$.
 The dotted rectangle represents any rectangle $R \in \mathcal{R}_q$. The water green rectangles rectangles represent the collection of rectangles $ \widetilde{R}_i \in  \mathcal{R}_{q+1}$. 
}\label{fig:InductiveProcess}
\end{figure}
in the coordinates of $R \in \mathcal{R}_q$ (see Figure~\ref{fig:InductiveProcess}).  Thanks to Lemma \ref{lemma:PropertiesBranchingMerging} and \eqref{d:b_q} we have that $b_{q+1}$ is still divergence free.
In a short-formula, we have
 \[ b_{q+1} (x,y) =
 \begin{cases}
 O_q^T W_{L_q, A_q, B_q, A_{q+1}, B_{q+1}, n_{q+1}, v_q} (\tau_q \circ O_q (x,y))
 & 
 \forall (x,y) \in \bigcup \mathcal{R}_q
 \\
 b_q (x,y)  & \text{otherwise}
 \end{cases}
\]
where  $O_q$ is the identity or a rotation by a multiple of $\pi/2$ and $\tau_q$ is a translation so that $b_{q+1}$ remains divergence free.  We now have to define the new collection $\mathcal{R}_{q+1}$ of rectangles.
Thanks to Lemma \ref{lemma:PropertiesBranchingMerging}, for each $R \in \mathcal{R}_q$, there exists a collection $\mathcal{R}_{q+1}(R)$ of $2 n_{q+1}$ mutually disjoint rectangles such that for all $\widetilde{R} \in \mathcal{R}_{q+1}(R)$ we have that $\widetilde{R} \subseteq R$ and
\[
 \widetilde{R} = [0, L_{q+1}] \times \left[ - \frac{A_{q+1} + B_{q+1}}{2}, \frac{A_{q+1} + B_{q+1}}{2} \right]
\]
up to a translation and a rotation of $\pi/2$ and in these coordinates
\[
 b_{q+1}(x,y) = 
 \begin{pmatrix}
    v_{q+1} \mathbbm{1}_{[- \frac{A_{q+1}}{2}, \frac{A_{q+1}}{2}]}(y) \\
    0
 \end{pmatrix}
 \quad
 \forall (x,y) \in \widetilde{R}.
\]
Note that thanks to Lemma \ref{lemma:PropertiesBranchingMerging}  $\# \mathcal{R}_{q+1}(R) = 2 n_{q+1}$ for all $R \in \mathcal{R}_{q}$.
Finally, we define
\[
 \mathcal{R}_{q+1} = \bigcup_{R \in \mathcal{R}_{q}} \mathcal{R}_{q+1}(R) \,,
\]
hence $\# \mathcal{R}_{q+1} = 2n_{q+1} \cdot N_q = N_{q+1}$.

 \begin{lemma}\label{lemma:LInfAndBVBounds}
For the sequence of velocity fields $\{ b_q \}_{q \geq 1}$ it holds that:
 \begin{align}
  \| b_q - b_{q-1} \|_{L^{\infty}(\T^2)} &\leq 2 v_{q-1} \leq C a_{q-1}^{- \eps + \frac{1}{2 + \delta}} \label{eq:LInftyFieldsLInfty}
 \end{align}
 where $C$ is a constant that depends only on $a_0$.
 \end{lemma}
 \begin{proof}
 By the construction above, there exists a collection $\mathcal{R}_{q-1}$ of $N_{q-1}$ rectangles such that up to a translation and a rotation
 \[
  R = [0, L_{q-1}] \times \left[ - \dfrac{A_{q-1} + B_{q-1}}{2}, \dfrac{A_{q-1} + B_{q - 1}}{2} \right] \quad \forall R \in \mathcal{R}_{q-1}.
 \]
 and the velocity field $b_{q-1}$ is a straight pipe whenever restricted to any of the rectangles in $\mathcal{R}_{q-1}$.  
 Equivalent to the above definition we have 
 \[
  b_q = b_{q-1} - \underbrace{\sum_{R \in \mathcal{R}_{q-1}} b_{q-1} \mathbbm{1}_{R}}_{\text{straight pipes}} \qquad + \underbrace{\sum_{R \in \mathcal{R}_{q-1}} W_{L_{q-1}, A_{q-1}, B_{q-1}, A_q, B_q, n_q, v_{q-1}} }_{\text{new branching-merging pipes}} \,.
 \]
  By construction $\| b_{q-1} \|_{L^\infty (R)} \leq  v_{q-1}$ for any $R \in \mathcal{R}_{q-1}$  and the new  velocities of the branching-merging pipes are bounded by $v_{q-1}$ thanks to Lemma \ref{lemma:PropertiesBranchingMerging}. As a consequence,
  \[
   \| b_{q} - b_{q-1} \|_{L^{\infty}(\T^2)} \leq 2 v_{q-1}.
  \]
\end{proof}

In the next lemma, we recall some of the most important properties of the velocity fields $\{ b_q \}_{q \geq 1}$.
{ We recall that we use $b_q^{(i)}$ to denote the $i$-th component of $b_q$.}
\begin{lemma}\label{lemma:PropertiesOfVelocityFieldsLInftyRectangles}
 The sequence of velocity fields $\{ b_q \}_{q \geq 1}$ constructed above satisfies the following properties for all $q \geq 2$
 \begin{enumerate}
  \item \label{item:LemmaTwoPointOne} 
  For all $R \in \mathcal{R}_{q-1}$ it holds that up to a rotation and a translation
  \[
   R = [0, L_{q-1}] \times \left[ - \dfrac{A_{q-1} + B_{q-1}}{2}, \dfrac{A_{q-1} + B_{q-1}}{2} \right]
  \]
  and in the coordinates of $R$ we have
  \[
   b_{q-1}(x,y) 
   =
   \begin{pmatrix}
    v_{q-1} \mathbbm{1}_{\left[- \frac{A_{q-1}}{2}, \frac{A_{q-1}}{2}\right]}(y) \\
    0
   \end{pmatrix}
   \quad 
   \forall (x,y) \in R.
  \]
  \item \label{item:LemmaTwoPointTwo} Each $R \in \mathcal{R}_{q-1}$ contains a collection of $2 n_q$ rectangles given by $\mathcal{R}_q(R)$ such that for all $\widetilde{R} \in \mathcal{R}_q(R)$ 
  \[
   \widetilde{R} = [0, L_{q}] \times \left[ - \dfrac{A_{q} + B_{q}}{2}, \dfrac{A_{q} + B_{q}}{2} \right]
  \]
  up to a rotation and a translation and in these coordinates
  \[
  b_q(x,y) 
   =
   \begin{pmatrix}
    v_{q} \mathbbm{1}_{\left[- \frac{A_{q}}{2}, \frac{A_{q}}{2}\right]}(y) \\
    0
   \end{pmatrix}
   \quad 
   \forall (x,y) \in \widetilde{R}.
  \]
  \item \label{item:LemmaTwoPointThree} Each $R \in \mathcal{R}_q$ contains a collection of rectangles denoted by $\mathcal{F}_q (R)$ of cardinality $\# \mathcal{F}_q (R) \leq 6 n_{q+1}$ such that{, in the coordinates of $R$,} for all $\overline{R} \in \mathcal{F}_q (R)$ { it holds that} up to  a translation { only}
  \[
   \overline{R} = { [- A_{q+1}, A_{q+1}] \times [0, L_{q+1}] \subseteq \bigcup_{\tilde{R} \in \mathcal{R}_{q+1}(R)} \tilde{R} }
  \]
  and
  \[
   \| b_{q+2}^{(2)} \|_{L^{\infty}(\overline{R})} \leq v_{q+1} \quad \forall \overline{R} \in \mathcal{F}_q (R) \quad \text{and} \quad b_{q+2}^{(2)} \equiv 0 \text{ in ${ \bigcup_{\tilde{R} \in \mathcal{R}_{q+1}(R)} \tilde{R}} \setminus \bigcup_{\overline{R} \in \mathcal{F}_q (R)} \overline{R}$} \,.
  \]
  Moreover,
  \[
   \inf_{\overline{R}_1, \overline{R}_2 \in \mathcal{F}_q (R)} \dist \left( \overline{R}_1, \overline{R}_2 \right) \geq \frac{B_{q+1}}{4}.
  \]
 \end{enumerate} 
\end{lemma}

{
\begin{remark}
 The collection $\mathcal{F}_q(R)$, see Figure~\ref{fig:CollectionFOfRectangles}, consists of mutually disjoint rectangles where the vertical component of $b_{q+2}$, in the coordinates of $R$, is supported.  This collection is used in the proof of the fact that some backward stochastic trajectories remain close to their starting point, see Equation~\eqref{eq:TrajectoriesRemainClose}. In these sets we use that the velocity field is bounded by $v_{q+1}$.
\end{remark}
}

\begin{figure}[ht]
\begin{tikzpicture}[scale=0.6]
\small
\BMPipeWithRectanglesAndCollection{0}{0}{0.8}{4}{1.5}{15}{3}{3}
\draw[black, thick , dotted] (0,4.4) rectangle (15,-4.4);
\draw[black] (15,3.8) node[anchor=west]{$R \in \mathcal{R}_q$};
\end{tikzpicture}
\caption{Depiction of the sets $\mathcal{F}_q(R)$ in the coordinates of $R \in \mathcal{R}_q$. The rectangles in water green represent the collection $\mathcal{R}_{q+1}(R)$. The blue rectangles represent $\mathcal{F}_q(R)$. The depicted { structure is the support of the} velocity field $b_{q+1}$. { The velocity field $b_{q+2}$ is obtained by replacing the straight pipes by branching-merging pipes, hence vertical velocities of $b_{q+2}$ are supported in the blue sets.}} \label{fig:CollectionFOfRectangles}
\centering
\end{figure}

\begin{proof}
The points \ref{item:LemmaTwoPointOne} and \ref{item:LemmaTwoPointTwo} follow immediately from the construction.
In order to prove \ref{item:LemmaTwoPointThree}, we note that each $R \in \mathcal{R}_{q}$ contains a collection of $2 n_{q+1}$ rectangles given by $\mathcal{R}_{q+1}(R)$ such that for all $\widetilde{R} \in \mathcal{R}_{q+1}(R)$
\[
  \widetilde{R} = [0, L_{q+1}] \times \left[ - \frac{A_{q+1} + B_{q+1}}{2}, \frac{A_{q+1} + B_{q+1}}{2} \right]
\]
and in these coordinates
\[
 b_{q+1}(x,y)
 =
 \begin{pmatrix}
  v_{q+1} \mathbbm{1}_{\left[ - \frac{A_{q+1}}{2}, \frac{A_{q+1}}{2} \right](y)} \\
  0
 \end{pmatrix}
 \quad
 \forall (x,y) \in \widetilde{R}.
\]
In order to construct $b_{q+2}$ inside $R$ from $b_{q+1}$ we replace the straight pipes in the rectangles of $\mathcal{R}_{q+1}(R)$ by branching-merging pipes. For any $R \in \mathcal{R}_{q}$ and any $\widetilde{R} \in \mathcal{R}_{q+1}(R)$ we define (in the coordinates of $\widetilde{R}$)
\begin{align*}
 F_{R, \widetilde{R}}^1 &= [0, L_{q+1}] \times \left[ - \frac{A_{q+1} + B_{q+1}}{2}, - \frac{B_{q+1} - A_{q+1} }{2} \right]; \\
 F_{R, \widetilde{R}}^2 &= [0, L_{q+1}] \times \left[ - A_{q+1}, A_{q+1} \right]; \\
 F_{R, \widetilde{R}}^3 &= [0, L_{q+1}] \times \left[ \frac{B_{q+1} - A_{q+1}}{2} ,  \frac{A_{q+1} + B_{q+1}}{2} \right]. \\
\end{align*}
From the construction, it is clear that
\[
 \supp \left( b_{q+2}^{(2)} \right) \cap (R \cap \widetilde{R}) \subseteq \bigcup_{\xi = 1}^3 F_{R, \widetilde{R}}^{\xi}.
\]
Now note that $F_{R, \widetilde{R}}^2$ is a rectangle of length $L_{q+1}$ and width $2 A_{q+1}$ whereas $F_{R, \widetilde{R}}^1$ and $F_{R, \widetilde{R}}^3$ are rectangles of length $L_{q+1}$ and width $A_{q+1}$. Moreover, $F_{R, \widetilde{R}}^1$ and $F_{R, \widetilde{R}}^3$ are situated at the boundary of $\widetilde{R}$. Hence, by merging sets $F_{R, \widetilde{R}}^1$ and $F_{R, \widetilde{R}}^3$ where $\widetilde{R} \in \mathcal{R}_{q+1}(R)$ with each other when adjacent we obtain $2 n_{q+1} - 1$ rectangles of length $L_{q+1}$ and width $2 A_{q+1}$ as well as 2 rectangles of length $L_{q+1}$ and width $A_{q+1}$. By including these two rectangles in larger rectangles and adding to this collection the collection of rectangles given by $\{ F_{R, \widetilde{R}}^2 \}_{\widetilde{R} \in \mathcal{R}_{q+1}(R)}$ we obtain a collection denoted by $\mathcal{F}_{q+1}(R)$ which is of cardinality $\# \mathcal{F}_{q+1}(R) = 4 n_{q+1} + 1 \leq 6 n_{q+1}$. 
\end{proof}

\subsection{Construction of $C^{\alpha}$ velocity fields $u_q$}\label{subsec:HolderConstruction}
The purpose of this subsection is to construct a sequence of $C^{\alpha}$-velocity fields $\{ u_q \}_{q \geq 1}$ whose limit $u \in C^{\alpha}(\T^2, \R^2)$ will be the velocity field in Theorem~\ref{thm-main}. We start by defining a sequence of mollification parameters $\{ \ell_q \}_{q \geq 1}$ as
\begin{equation}\label{eq:parameter:mollification}
{  \ell_q = \overline{A}_{q+1}^{1 + \eps}}.
\end{equation}
Then we define $\{ w_q \}_{q \geq 1}$ as the sequence of velocity fields given by
\begin{equation}\label{eq:ChoiceOfSequenceWq}
 w_q =
 \begin{cases}
  b_q - b_{q-4} &\text{ if $q \in 4\N$;} \\
  0 &\text{otherwise.} \\
 \end{cases}
\end{equation}
The sequence $\{ u_q \}_{q \geq 0}$ is then given by
\begin{equation}\label{eq:DefinitionOfU_qSeq}
 u_q = b_0 \star \varphi_{\ell_0} + \sum_{j = 1}^q w_j \star \varphi_{\ell_j}.
\end{equation}
{
\begin{remark}\label{rmk:IntersectionProblem}
Expanding \eqref{eq:DefinitionOfU_qSeq}  for any $q \in 4\N$ we obtain  $u_q = b_q \star \varphi_{\ell_q} + \text{error}$, where the error consists of a sum of terms of the form $b_{j-4} \ast (\varphi_{\ell_{j-4}} - \varphi_{\ell_j})$.  
{The full error term leads to} a velocity field of magnitude at most $v_j$, supported in { thin} strips of width $ 2\ell_{j-4}$ that intersect the $(j-3)$-th generation pipes perpendicularly. We refer to this as the intersection problem and it is addressed in the stability analysis of Section~\ref{sec:stability}.
\end{remark} }

 We now prove several lemmas regarding the sequence $\{ u_q \}_{q \geq 0}$.

\begin{lemma}  \label{lemma:NS-Calpha}
Let $\{ u_q \}_{q \geq 0} \subset C^\infty(\T^2; \R^2)$ be the sequence of velocity fields defined above in \eqref{eq:DefinitionOfU_qSeq}, then there exist a constant $C>0$ and $\tilde{\eps} >0$ independent on $q \in \N$ such that the following inequalities hold for all $q \geq 1$
 \begin{align}
 \| u_{q+1} - u_q \|_{C^{\alpha}} \leq C a_{q}^{\tilde{\eps}};  \label{eq:regularityCalpha}
 \\
 \| u_q \cdot \nabla u_q - u_{q-1} \cdot \nabla u_{q-1} \|_{C^\alpha} \leq C a_q^{\tilde{\varepsilon}};  \, \label{eq:NS-nonlinear}
 \\
 \kappa_q \| \Delta u_{q+2} \|_{C^\alpha} \leq C q a_q^{\tilde{\varepsilon}}. \, \label{eq:NS-Laplacian}
 \end{align} 
\end{lemma}

\begin{proof}
 By \eqref{eq:LInftyFieldsLInfty}, $\| w_q \|_{L^{\infty}} \lesssim v_{q-4}$ so that $\| w_q \star \varphi_{\ell_q} \|_{C^{\alpha}} \lesssim v_{q-4} \ell_q^{- \alpha}$.
Moreover, Equations~\eqref{d:eps-delta-alpha} imply that there exists a $\tilde{\eps} > 0$ such that
\begin{align*} 
\dfrac{1}{2+\delta} - \eps - \alpha (1 + \delta)^4 \left( 3 \eps + 3 \delta + \frac{(1+\delta)^2}{2 + \delta} \right) > \tilde{\eps}(1 + \delta)^3, \\
\dfrac{1 + (1 + \delta)^4}{2 + \delta} - \eps(1 + (1 + \delta)^4) - (1 + \alpha) (1 + \delta)^8 \left( 3 \eps + 3 \delta + \frac{(1+\delta)^2}{2 + \delta} \right) > \tilde{\eps}(1 + \delta)^8, \\
\dfrac{1 + 2(1 + \delta)^5}{2 + \delta} - \eps - (2 + \alpha) (1 + \delta)^4 \left( 3 \eps + 3 \delta + \frac{(1+\delta)^2}{2 + \delta} \right) > \tilde{\eps}(1 + \delta)^4.
\end{align*}
From the definition of $\ell_q$, we find $\ell_q \gtrsim a_q^{3 \eps + 3 \delta + \frac{(1 + \delta)^2}{2 + \delta}}$
and therefore
 \begin{align*}
  \| w_q \star \varphi_{\ell_q} \|_{C^{\alpha}} &\lesssim v_{q-4} \ell_q^{- \alpha} \lesssim a_{q-4}^{- \eps + \frac{1}{2+ \delta} - \alpha (1 + \delta)^4 \left( 3 \eps + 3 \delta + \frac{(1 + \delta)^2}{2 + \delta} \right)}  \lesssim a_{q -1 }^{\tilde{\eps}}. \,,
 \end{align*}
 Hence,
 \begin{equation*}
  \| u_{q+1} - u_q \|_{C^{\alpha}} \leq \| w_{q+1} \star \varphi_{\ell_{q+1}} \|_{C^{\alpha}} \lesssim a_{q}^{\tilde{\eps}}.
 \end{equation*}
This proves \eqref{eq:regularityCalpha}. 
 To prove \eqref{eq:NS-nonlinear} we observe that \eqref{eq:DefinitionOfU_qSeq} implies
 \begin{align*}
 u_q \cdot \nabla u_q  & =  \sum_{i, j =0}^q  {w}_i \star \varphi_{\ell_i} \cdot \nabla (w_j \star \varphi_{\ell_j} ) 
 \\
 &  =   \sum_{i =0}^q  {w}_i \star \varphi_{\ell_i} \cdot \nabla (w_i \star \varphi_{\ell_i} ) +   \sum_{i =0}^{q-{ 4}}  {w}_i \star \varphi_{\ell_i} \cdot \nabla (w_{i+{ 4}} \star \varphi_{\ell_{i+{ 4}}} )
 \\
 & \quad +  \sum_{i =0}^{q-{ 4}}  {w}_{i+{ 4}} \star \varphi_{\ell_{i+{ 4}}} \cdot \nabla (w_{i} \star \varphi_{\ell_{i}} ) 
\end{align*}  
 where we use the shorthand notation $w_0 \coloneqq b_0$  and for the last equality we used that 
 \begin{align} \label{eq:disjoint-support-w}
 \supp ({w}_i \star \varphi_{\ell_i} ) \cap \supp (   w_j \star \varphi_{\ell_j} ) = \emptyset \quad  \text{for any} \quad  |i-j| {  \neq 4}
 \end{align}
 which follows from the construction.
 Therefore, {considering $q \in 4 \N$ without loss of generality, we have}
 \begin{align*}
  (u_q \cdot \nabla) u_q - (u_{q-1} \cdot \nabla) u_{q-1} &= ((w_{q} \star \varphi_{\ell_q}) \cdot \nabla) (w_{q} \star \varphi_{\ell_q}) 
  \\
  & \quad + ((w_{q-{ 4}} \star \varphi_{\ell_{q-{ 4}}}) \cdot \nabla) (w_{q} \star \varphi_{\ell_q}) 
  \\
  &\quad + ((w_{q} \star \varphi_{\ell_{q}}) \cdot \nabla) (w_{q-{ 4}} \star \varphi_{\ell_{q-{ 4}}}) 
 \end{align*}
 so that {for $q\geq 8$}
, using $\| (f \cdot \nabla) g \|_{C^{\alpha}} \lesssim \| f \|_{C^{0}} \| g \|_{C^{\alpha + 1}} + \| f \|_{C^{\alpha}} \| g \|_{C^{1}}$ and the fact that $\ell_q \leq \ell_{q-4}$ and $v_{q-4} \leq v_{q-8}$
 \begin{align*}
  \| (u_q \cdot \nabla) u_q - &(u_{q-1} \cdot \nabla) u_{q-1} \|_{C^{\alpha}} \lesssim v_{q-8} v_{q-4} \ell_q^{- 1 - \alpha} \\
  &\lesssim {a_{q-8}^{\frac{1 + (1+\delta)^4}{2+\delta} + \eps \left(1 + (1+\delta)^4 \right) - (1+\alpha)(1+\delta)^8 \left( 3 \eps + 3 \delta + \frac{(1+\delta)^2}{2 + \delta} \right)}} \\
&\lesssim {a_{q}^{\tilde{\eps}}}.
 \end{align*}
This proves \eqref{eq:NS-nonlinear}. 
Finally, we prove \eqref{eq:NS-Laplacian}. 
For all $q \geq 1$, since $q \mapsto \ell_q^{-k} v_{q-4}$ is increasing,
\begin{align*}
 \kappa_q \| \Delta u_q \|_{C^{\alpha}} &\lesssim \kappa_q \sum_{j = 0}^{q} \ell_j^{- 2 - \alpha} v_{j-4} \lesssim \kappa_q q \ell_q^{-2 - \alpha} v_{q-4} \\
 &\lesssim q a_{q-4}^{- (2 + \alpha)(1 + \delta)^4 \left( 3 \eps + 3 \delta + \frac{(1+\delta)^2}{2 + \delta} \right) - \eps + \frac{1 + 2 (1 + \delta)^5}{2 + \delta}} \lesssim q a_{q}^{\tilde{\eps}}
\end{align*}
where we used {\eqref{parameter:kappa-q}, \eqref{parameter:v-q}} and \eqref{d:eps-delta-alpha}.
\end{proof}

Thanks to this lemma, it is clear that $\{ u_q \}_{q \geq 0}$ is a Cauchy sequence in $C^{\alpha}(\T^2; \R^2)$ and hence we may define
\[
 u = \lim_{q \to \infty} u_q \in C^{\alpha}(\T^2; \R^2).
\]
In particular, a direct implication of this definition is that for any $q \in \N$, we have
\begin{equation}\label{eq:BoundBetweenApproxAndFinal}
 \| u - u_q \|_{L^\infty} \lesssim v_q \,.
\end{equation}

From Lemma \ref{lemma:PropertiesOfVelocityFieldsLInftyRectangles} and definition of $u_q$ as in \eqref{eq:DefinitionOfU_qSeq} the following lemma follows.
\begin{lemma}\label{lemma:AboutTheStructureOfHolderFields}
 The sequence of velocity fields $\{ u_q \}_{q \geq 0}$ constructed above satisfies the following properties for all $q \in \{ 2 + 4 k : k \N \}$:
 \begin{enumerate}
  \item \label{item:AboutTheStructureOfHolderFieldsItemOne} 
  For each $R \in \mathcal{R}_{q-1}$ and each $\widetilde{R} \in \mathcal{R}_{q}(R)$ it holds that
  \[
   u_{q+2}(x,y) = 
   \begin{pmatrix}
    v_q \\
    0 \\
   \end{pmatrix}
   \quad 
   \forall (x,y) \in \left[0, \frac{L_q}{2} \right] \times \left[- \frac{A_q}{5}, \frac{A_q}{5} \right] \subseteq \widetilde{R}
  \]
  in the coordinates of $\widetilde{R}$.
  \item \label{item:AboutTheStructureOfHolderFieldsItemTwo} 
  For each $R \in \mathcal{R}_{q}$, the collection of rectangles $\mathcal{F}_q(R)$ given by Item~\ref{item:LemmaTwoPointThree} of Lemma~\ref{lemma:PropertiesOfVelocityFieldsLInftyRectangles} has the property that
  \[
   \supp \left( u_{q+2}^{(2)} \right) \cap R \subseteq \bigcup_{\overline{R} \in \mathcal{F}_q(R)} \overline{R}
  \]
  and
  \[
   \| u_{q+2}^{(2)} \|_{L^{\infty}(\overline{R})} \leq v_{q+1} \quad \forall \overline{R} \in \mathcal{F}_q(R).
  \]
 \end{enumerate}
\end{lemma}

\subsection{The dissipative set} \label{subsec:dissipative}

{
For any $q \in \N$ we are now ready to define the so called dissipative set $D_q \subset \T^2$. This set is selected so that backward stochastic trajectories starting in $D_q$  of $u_{q+2}$ have a completely different behaviour depending on the realization, see \eqref{eq:TrajectoriesRemainClose} and \eqref{eq:TrajectoriesGoFar} for a more precise statement. There are three crucial properties satisfied by this set.
\begin{itemize}
\item For any $q \in \N$, $D_q$ is a union of rectangles with width comparable to the diffusivity parameter $\sqrt{ \kappa_q}$.
\item The velocity field $u_{q+2}$  is quantitatively negligible on $D_q$ thanks to an It\^o-Tanaka trick \eqref{eq:itotanaka}, so that we can prove \eqref{eq:TrajectoriesRemainClose}.
\item On a $\sqrt{\kappa_q}$ neighborhood of $D_q$, one can show that stochastic trajectories may be strongly influenced by the velocity field $u_{q+2}$ and deduce \eqref{eq:TrajectoriesGoFar}. This requires the ergodic property of the Brownian motion (see \eqref{eq:SequenceOfIneqs}) and the stability results from Section~\ref{sec:stability}.
\item The measure of $D_q$ can be bounded below uniformly in $q$ 
$$\inf_{q \geq 1} \Leb^2(D_q) > 0 \,. $$
\end{itemize} 
As specified in the first point, $D_q$ is a union of rectangles. However, to get the third property and prove the stability results in Section \ref{sec:stability} we need to introduce some suitable restrictions of such rectangles that require some notation.}

We define $\mathcal{G}_q (R) \subseteq \mathcal{R}_{q} (R)$ to be a suitable sub-collection of good rectangles obtained by removing the { $2 n_q^{1- \varepsilon}$ first and the $2 n_q^{1- \varepsilon}$ last} rectangles in the branching structure (see Figure~\ref{fig:TheSetsMWithDissipative}). Hence $\# \mathcal{G}_{q} (R)= \# \mathcal{R}_{q} (R) - {  4 n_q^{1- \varepsilon}} = 2n_q -{ 4 n_q^{1- \varepsilon} }$. 
Then, we define $\mathcal{G}_{q}$  by the following relation
\begin{align}\label{d:G_q}
\mathcal{G}_{q} = \bigcup_{R \in \mathcal{G}_{q-1}} \mathcal{G}_{q}(R) \quad \forall q \geq 2 \,.
\end{align}
Hence, for $q \geq 2$ using that $a_0$ can be taken sufficiently small we have
{
\begin{align} \label{eq:G_q}
 \# \mathcal{G}_{q} & = \prod_{j = 1}^q (2 n_j - 4 n_j^{1- \varepsilon})  \geq \prod_{j = 1}^q 2 n_j  \prod_{j = 1}^q (1 -  2 n_j^{- \varepsilon})   
 \\
 & \geq  N_q  \prod_{j = 2}^q \left(1 - \frac{1}{j^2} \right) =   N_q  \frac{ q+1}{2q} \geq \frac{N_q}{2} \,. \notag
\end{align}}

Before stating the last lemma of this subsection we will define { the so called {\em dissipative} sets} $\{ D_q \}_{q \geq 1}$.  Let $q \geq 5$ be arbitrary and define for each $R \in \mathcal{G}_q$ the set
\[
 D_{q,R} = \left[ {\frac{L_q}{9}} , \frac{L_q}{3} \right] \times \left[ - \frac{A_q}{2} - \frac{\sqrt{\kappa_q}}{50} , - \frac{A_q}{2} - \frac{\sqrt{\kappa_q}}{100} \right] \subseteq R
\]
in the coordinates of $R$.
Then, we define $D_q$ (see Figure~\ref{fig:TheSetsMWithDissipative}) as
\begin{equation} \label{d:dissipative}
 D_q = \bigcup_{R \in \mathcal{G}_q} D_{q,R}\,.
\end{equation}
 \begin{figure}[ht]
\begin{tikzpicture}[scale=0.6]
\small
\BMPipeWithRectanglesAndDissipativeSet{0}{0}{0.8}{8}{4}{15}{3}{2}
\draw[black, thick , dotted] (0,4.1) rectangle (15,-4.1);
\end{tikzpicture}
\centering
\caption{The dotted rectangle is a $R \in \mathcal{R}_q$. The water green rectangles represent the rectangles belonging to {$\mathcal{G}_{q+1}(R)$}. For any $\widetilde{R} \in {\mathcal{G}_{q+1}(R)}$ the red rectangles in the figure illustrates the sets $D_{q+1, \widetilde{R}}$.}\label{fig:TheSetsMWithDissipative}
\end{figure}
\begin{figure}[h]
\begin{tikzpicture}
\RectanglesTwoGenWithPipesPeriodicity{0}{0}{0.6}{6}{3}{10}{3}{3}
\path[fill=red, fill opacity=0.35] (0.6,-1.1) rectangle ++(2.8, -0.7);
\path[fill=red, fill opacity=0.55] (0.7,-1.2) rectangle ++(2.6, -0.5);
\draw[black, thick , dotted] (0,3.915) rectangle (10,-3.915);
\draw[black, thick, <->, dashed] (-0.2,-0.6) -- (-0.2,0.6);
\draw[black] (-0.2,0) node[anchor=east]{$A_q$};
\draw[black, thick, <->, dashed] (0,-4.2) -- (10,-4.2);
\draw[black] (5,-4.4) node[anchor=north]{$L_q$};
\draw[black, thick, <->, dashed] (10.2,-4.05) -- (10.2,4.05);
\draw[black] (10.2,-1.5) node[anchor=west]{$A_q + B_q$};
\draw[black, thick, <->, dashed] (0.25,-0.7) -- (0.25,-3.2);
\draw[black] (0.1,-2.1) node[anchor=east]{$L_{q+1}$};
\draw[black, thick, <->, dashed] (6.83,-0.5) -- (8.09,-0.5);
\draw[black] (7.6,0.26) node[anchor=west]{\rotatebox{60}{\tiny $A_{q+1} + B_{q+1}$}};
\draw[black, thick, <->, dashed] (0.48,-3.5) -- (1.76,-3.5);
\draw[black] (0.8,-3.55) -- (0.70,-4.5);
\draw[black] (0.70,-4.5) node[anchor=north]{\tiny $A_{q+1} + B_{q+1}$};
\draw[black, thick, <->, dashed] (1.76,-3.5) -- (3.04,-3.5);
\draw[black] (2.0,-3.55) -- (0.80,-4.5);
\draw[black] (10.2,3.5) node[anchor=west]{$R \in \mathcal{G}_q$};
\draw[black] (1.5,-1.5) -- (-0.5,-3.5);
\draw[black] (-0.5,-3.5) node[anchor=north]{$D_{q, R}$};
\draw[black] (2.7,-1.75) -- (3.4,-4.3);
\draw[black] (3.4,-4.2) node[anchor=north]{$I_{\frac{\sqrt{\kappa_q}}{150}} (D_{q, R})$};
\end{tikzpicture}
\centering
\caption{ Illustration of the periodicity inside the dissipative set. { The figure depicts a rectangle $R \in \mathcal{G}_q$, the support of $b_{q+2}$ inside $R$ and $D_{q,R}$.}}\label{fig:PeriodicityFigure}
\end{figure}

\begin{lemma}\label{lemma:gronwall+periodicity+dissipative}
 Let $D_q \subset \T^2$ constructed above and  $\mathcal{G}_q$ be the collection defined in \eqref{d:G_q}. Then,
  the following properties hold true
 \begin{enumerate}
  \item $\inf_{q \geq 1} \Leb^2(D_q) > 0$; \label{item:ItemTwoInLemmaGronwallDissPeriod}
  \item for all $R \in \mathcal{G}_q$, we have $\| b_{q+2}^{(1)} \|_{L^{\infty}(D_{q,R})} \leq v_{q+2}$ and in the coordinates of $R$ we have (see Figure~\ref{fig:PeriodicityFigure}) \label{item:ItemThreeInLemmaGronwallDissPeriod}
  \[
   b_{q+2}(x,y) = b_{q+2}(x + A_{q+{ 1}} + B_{q+{ 1}}, y) \quad \forall (x,y) \in I_{\sfrac{\sqrt{\kappa_q}}{150}} ( D_{q,R} ) \,.
  \]
 \end{enumerate}
\end{lemma}
\begin{proof}
Point \ref{item:ItemTwoInLemmaGronwallDissPeriod} follows from the fact that
\begin{align*}
 \Leb^2 (D_q) = \# \mathcal{G}_q \cdot \frac{\sqrt{\kappa_q}}{100} \cdot \frac{{ 2}L_q}{{ 9}} \stackrel{\eqref{eq:G_q}}{\geq} \frac{N_q}{2} \cdot \frac{\sqrt{\kappa_q} L_q}{{900}} \geq c > 0
\end{align*}
for some constant $c$ depending only on $a_0$. Point \ref{item:ItemThreeInLemmaGronwallDissPeriod} is a direct consequence of the construction.
\end{proof}

\section{Stability results} \label{sec:stability}
{ The objective of this section is to prove Proposition~\ref{lemma:stab-2}, { which is the second part of the proof that trajectories may be whisked away, see \eqref{eq:HeuristicsParticleGoAway} and Section \ref{subsec:InsightsMainTheorem}.}}

\subsection{Statements and proof of Proposition~\ref{lemma:stab-2}} 
{ The set $H_q$ is located at the end of $q$-th generation pipes close to $(q-1)$-th generation pipes (defined in Equation~\eqref{eq:DefsOfSets}). { The dissipative set $D_q$, defined  in \eqref{d:dissipative}, lies outside the $q$-th generation pipes but within the rectangles $R \in \mathcal{R}_q$, see Figure~\ref{fig:PeriodicityFigure}.}}
\begin{proposition}\label{lemma:stab-2}
 Let $q \in 2 + 4\N$ { be sufficiently large} and $X_{1,s}^{\kappa_q}$ be the backward stochastic flow with velocity field $u_{q+2}$ and noise parameter $\sqrt{2 \kappa_q}$  starting at $s=1$ from a fixed $(x,y) \in {D}_q$. Let us define the stopping time 
$$T_1 =  \sup \{ s \in ({\color{black} - \infty} ,1] : X_{1,s}^{\kappa_q}(x,y,\omega) \in H_{q} \} \vee 0$$ 
and suppose that 
\begin{equation} \label{eq:T_1}
 \Prob (T_1 \geq 1/2) \geq \bar{c} >0
\end{equation}
with a constant $\bar{c}$ independent of $q$ and $(x,y)$.
Then, there exists a constant $c>0 $ independent of $q$ and $(x,y)$ such that 
 \begin{align} \label{eq:prob-estim}
  \Prob (\omega: \dist ( X_{1,0}^{\kappa_q} (x,y, \omega) , (x,y) ) >c  ) \geq c \,. 
 \end{align}
 \end{proposition}
 { The proof of Proposition~\ref{lemma:stab-2} relies on the two following lemmas.
  The set $E_{q-2}$ is located at the end of $(q-2)$-th rectangles within pipes (defined in Equation~\eqref{eq:DefsOfSets}).}
  { 
 \begin{lemma}\label{lemma:FirstStepOfStab}
  Under the same assumptions as in Proposition~\ref{lemma:stab-2}, consider the stopping time
  \[
   T_0 = \sup \{ s \in ( {\color{black} - \infty} ,{\color{black} T_1}] : X_{1, s}^{\kappa_q}(x,y, \omega) \in E_{q-2} \} { \vee 0}.
  \]
  Then there exists a constant $c>0 $ independent of $q$ and $(x,y)$ such that
 \begin{align} \label{eq:stability-with-positive-prob}
  \Prob ( T_0 \geq 1/5 ) \geq c \,. 
 \end{align}
 \end{lemma} 
 For any stopping time $T_0 : \Omega \to [{0}, 1]$, we define the probability set 
\[
 \widetilde{\Omega}_{T_0} = \left\{ \omega \in \Omega : \sup_{t \in [0, T_0]} | W_t - W_{T_0}| \leq 1 \right\} \,,
\] 
which satisfies $\Prob [\widetilde{\Omega}_{T_0}  ] \geq c >0$ with a universal constant $c$.}
  \begin{lemma}\label{lemma:FullStability}
 Let $q \in 2 + 4\N$ {be sufficiently large} and let $X_{T,s}^{\kappa_q}$ denote the backward stochastic flow with velocity field $u_{q+2}$ 
  and noise parameter $\sqrt{2 \kappa_q}$ starting at $s=T$ from a fixed $(x,y) \in E_{q-2}$. Let $T_0 \colon \Omega \to [0,1]$ be an arbitrary stopping time. Then there exists a constant $c>0$ independent of $q$ and $(x,y)$ such that for all $\omega \in \tilde{\Omega}_{T_0} \cap \{ T_0 \geq 1/5 \}$ and for all $(x,y) \in E_{q-2}$ we have
 \begin{equation}\label{eq:UsefulCorollary}
  \dist (X_{T_0 ,0}^{\kappa_q}(x,y,\omega), (x,y) ) \geq c \,.
 \end{equation}
 \end{lemma}
 
  {
 
 \begin{proof}[Proof of Proposition 8.1.]
 Fix some $(x,y) \in D_q$ and define the stopping time $T_0 \colon \Omega \to [0,1]$ by $T_0 =  \sup \{ s \in ({\color{black} - \infty} ,T_1 (\omega)] : X_{T_1,s}^{\kappa_q}(x,y,\omega) \in E_{q-2} \} \vee 0$.
 Denote by $c$ a constant independent of $q$ and $(x,y)$ which may vary from line to line.
 By Lemma~\ref{lemma:FirstStepOfStab},  {\color{black} using in particular \eqref{eq:T_1}, we get}  $\Prob(T_0 \geq 1/5) \geq c > 0$.
 Thus $\Prob(\tilde{\Omega}_{T_0} \cap \{ T_0 \geq 1/5 \}) = \Prob(\tilde{\Omega}_{T_0}) \Prob ( T_0 \geq 1/5 ) \geq c > 0$ since $\tilde{\Omega}_{T_0}$ and $\{ T_0 \geq 1/5 \}$ are independent.
 Fix some $\omega \in \tilde{\Omega}_{T_0} \cap \{ T_0 \geq 1/5 \}$.
 By Lemma~\ref{lemma:FullStability}, since $X_{1, T_0}^{\color{black} \kappa_q} (x,y,\omega) \in E_{q-2}$, we have
 \[
  \dist (X_{T_0 ,0}^{\kappa_q}(X_{1, T_0}^{\color{black} \kappa_q}(x,y,\omega),\omega), X_{1, T_0}^{\color{black} \kappa_q}(x,y,\omega) ) \geq c > 0\,.
 \]
 Since $X_{T_0 ,0}^{\kappa_q}(X_{1, T_0}^{\color{black} \kappa_q}(x,y,\omega),\omega) = X_{1, 0}^{\color{black} \kappa_q}(x,y,\omega)$ and 
 \[
 { \color{black} \dist } ( X_{1, T_0}^{\color{black} \kappa_q}(x,y,\omega) , (x,y)) \leq L_{q-2} + B_{q-2} \to 0 \text{ as }q \to \infty, 
 \]
 we deduce that for $q$ sufficiently large
 \[
  \dist (X_{1 ,0}^{\kappa_q}(x,y,\omega), (x,y) ) \geq c > 0\,.
 \]
 Combined with the fact that $\Prob(\tilde{\Omega}_{T_0} \cap \{ T_0 \geq 1/5 \}) \geq c > 0$, we deduce that \eqref{eq:prob-estim} holds.
 \end{proof}
 }
 
 \subsection{Heuristics} \label{subsec:Heuristics-stability}
 { We provide heuristics for the proof of Lemmas~\ref{lemma:FirstStepOfStab} and~\ref{lemma:FullStability}, starting with Lemma~\ref{lemma:FullStability}.
 The proof exploits the fact that in certain regimes, advection drives the evolution and the noise becomes negligible.
 We prove that stochastic backwards trajectories which reach the $j$-th generation pipe also reach the $(j-1)$-th generation pipe and iterate this property.
 A {  backward} deterministic trajectory goes from a $j$-th pipe to a $(j-1)$-th pipe by first moving at velocity $v_j$ over a distance at most $2 L_j$ and secondly moving at velocity $v_{j-1}$ over a distance { at most} $L_{j-1}$. The time required for this is at most
\[
 t_{j, j-1} \leq  \dfrac{2 L_j}{v_j} + \dfrac{L_{j-1}}{v_{j-1}} \lesssim a_{j-1}^{\eps} \to 0 \text{ as } j \to \infty.
\]
Since the movement of the noise is typically comparable to $\sqrt{\kappa_q}$, if the velocity field is locally constant within balls of radius $r \gg \sqrt{\kappa_q}$ centred on a given deterministic backward trajectory, then locally stability is expected.
By the choice of the family of good rectangles $\mathcal{G}_j$ defined in Equation~\eqref{d:G_q}, we ensure that balls of radius \( A_j^\varepsilon \bar{A}_j \ll \bar{A}_j \), centred on such trajectories, remain contained within the \( j \)-th pipe. Thus, since our velocity field is made of pipes which are mostly locally constant, if 
 \begin{equation}\label{eq:HeuristicsStabilityOneRequirement}
 \sqrt{\kappa_q} \ll A_{j}^{\eps} {\overline{A}_j}
\end{equation}
holds true, then one may expect stability. Equation~\eqref{eq:HeuristicsStabilityOneRequirement} requires $j \leq q-1$. 
 However, there are two technical points where the velocity field is not locally constant on balls of radius $r \gg \sqrt{\kappa_q}$:
\begin{enumerate}
\item the intersection problem, due to the mollification procedure as described in Remark~\ref{rmk:IntersectionProblem};\label{item:ObstructionIntersectionProblem}
\item the change of direction of pipes, i.e. they turn. \label{item:ObstructionRotatingPipes}
\end{enumerate}
We will now provide heuristics explaining why \ref{item:ObstructionIntersectionProblem} cannot destroy { local} stability whenever $j \leq q-1$ and \ref{item:ObstructionRotatingPipes} cannot destroy { local} stability if $j \leq q-2$.  

We start with \ref{item:ObstructionIntersectionProblem}.
Consider the following model problem (see Figure~\ref{fig:ToyProblem}). }
 Let $u \in C^\infty (\T^2)$ be a divergence free velocity field and $\ell >0$. {Define} 
 $$M = \T \times [0, \ell] \,,  $$
 and  suppose that 
 $$ \| u^{(1)} \|_{L^\infty (M)} \leq v_{0} \quad u^{(2)}|_{[0, A] \times \T} = - v_{1} \,.$$
\begin{figure}
\begin{tikzpicture}
\small
\path[fill=orange, fill opacity=0.5] (0,-2) rectangle ++(7,-0.5);
\draw[black, thick, <->, dashed] (7.2,-2.0) -- (7.2,-2.5);
\draw[black, thick, black] (3,-4) -- ++(0,4);
\draw[black, thick, black] (4,-4) -- ++(0,4);
\draw[black] (7.2,-2.25) node[anchor=west]{$\ell$};
\pgfmathsetseed{1233}
\BMotion{3.5}{-4}{82}{0.02}{0.05}{darkgray, thick}{ }{0}{0.05};
\draw[black] (0.5,-2.2) node[anchor=west]{$M$};
\draw[black, thick, <->, dashed] (3,-4.2) -- (4,-4.2);
\draw[black] (3.5,-4.2) node[anchor=north]{$A$};
\end{tikzpicture}
\centering
\caption{The model problem. In the orange { strip} there is an horizontal shear flow and in the white pipe there is a vertical shear flow. The figure shows a {\em typical} realization of the  \emph{backward} stochastic trajectory (going from the bottom to the top of the figure) under a suitable choice of the parameters.}\label{fig:ToyProblem}
\end{figure}
 We want to find sufficient conditions on $v_1, v_0, \ell, \sqrt{\kappa}$ so that \emph{backward} stochastic trajectories with noise parameter $\sqrt{2 \kappa }$ 
 do not exit the region $[0,A] \times \T$ and pass through the region $M$. 
 Suppose that $x = A/2$. It is straightforward to { verify} that a  deterministic backward trajectory passes through the region $M$ if 
 $$ \ell \frac{v_0}{v_1} < \frac{A}{2} \,. $$
 This gives one constraint on the mollification parameter $\ell_{j-1} \ll \frac{v_j}{v_{j-1}} A_j^{1+ \varepsilon}$.
 { 
  On the other hand, we aim to define an 
$\alpha$-H\"older continuous velocity field, and this imposes the constraint 
$\ell_{q+4}^{-\alpha} v_q \ll 1$ on the choice of parameters. Thus, the intersection problem is resolved by selecting
a mollification parameter that satisfies both constraints.}
One can prove the same property for the backward stochastic trajectories assuming that $\sqrt{\kappa} \ll \ell$ with the use of a stopping time as we do in Step 2 of Lemma \ref{lemma:FullStability}.
  { To satisfy} the condition
 \begin{equation}\label{eq:ConditionNumber2ForStab}
 \sqrt{\kappa} \ll \ell
 \end{equation}
 we { choose to} mollify the velocity field {on every fourth step} (see Equation~\eqref{eq:ChoiceOfSequenceWq}) { since $\ell_{q-1} \ll \sqrt{\kappa_q} \ll \ell_{q-2}$}. In this way, the intersection problem appears { only} every {fourth step}. { Thus, when studying the velocity field $u_{q+2}$ ($q \in 2 + 4\N$) with diffusivity $\kappa_q$, the width of the thinnest strip where the intersection problem occurs is $2 \ell_{q-2}$.} 
{ We now address item~\ref{item:ObstructionRotatingPipes}, i.e. the change of direction from pipes with velocity $v_{j-1}$ to $v_j$ near $\overline{\{ |b_{q+2}| = v_j \}} \cap \overline{\{ |b_{q+2}| = v_{j-1} \}}$. By the mollification procedure, there is a thin strip of width smaller than $3 \ell_j$ which may contain velocities of magnitude $v_{j-1}$, reminiscent of the intersection problem. Satisfying~\eqref{eq:ConditionNumber2ForStab} with $\ell = 3 \ell_j$ requires $j \leq q-2$ which is the reason why Lemma~\ref{lemma:FullStability} is stated with $E_{q-2}$.
Thus, it only remains to resolve item~\ref{item:ObstructionRotatingPipes} in the cases $j = q$ and $j = q-1$. This is the content of the separate Lemma~\ref{lemma:FirstStepOfStab}.
Since \( \ell_q, \ell_{q-1} \ll \sqrt{\kappa_q} \), and because the velocity field is consistently directed in the desired direction for stability, we can select random realisations such that, within a time interval of at most $1/10$, the backward stochastic  trajectories enter the set
$
{\{ b^{(1)}_{q+2} = v_{j-1} \}}
$
with a distance of at least \( \sqrt{\kappa_q} \) from its boundary.
Due to independence of increments of Brownian motion, i.e the strong Markov property (see Theorem~\ref{thm:strong-markov}), we can mimic the proof of Lemma~\ref{lemma:FullStability} in the remaining parts of the proof.}

\subsection{Proofs of Lemmas~\ref{lemma:FirstStepOfStab} and \ref{lemma:FullStability}}
{ {
We will use repeatedly throughout the proofs that
\begin{align} \label{eq:constant}
 \frac{4 L_j { + 4 \sqrt{2 \kappa_q}}}{v_j a_j^\varepsilon} \leq 1 \,, \quad { \forall 1 \leq j \leq q-1} \,,
\end{align}
which follows from \eqref{eq:par-initial}, \eqref{parameter:L_q} and \eqref{parameter:v-q}.
}
{
The proofs are carried out by studying local properties of the flow.
Therefore, we define the sets $H_j$, $E_j$, $M_j$.
We refer to Figures~\ref{fig:zoomout} and \ref{fig:mainsets} for a visual representation of these sets.
Let $j \geq 1$ and $R \in \mathcal{G}_j$.
Up to a rotation and a translation, the rectangle $R$ can be written as
\[
 R = [0, L_j] \times \left[ - \frac{A_j + B_j}{2}, \frac{A_j + B_j}{2} \right].
\]
In these coordinates, we define the \emph{exit set} as
\begin{align}\label{d:set-E-q}
 E_{j, R} = \{ 0 \} \times \left[ - \frac{A_j}{2} + 2 A_j^{1 + \eps}, \frac{A_j}{2} - 2 A_j^{1 + \eps} \right].
\end{align}
We define the \emph{mollification set} as
\begin{equation*}
 M_{j, R} = [0, L_j] \times \left[ - \frac{A_j}{2} - \ell_{j} , - \frac{A_j}{2} + \ell_{j} \right] \cup [0, L_j] \times \left[ \frac{A_j}{2} - \ell_{j} , \frac{A_j}{2} + \ell_{j} \right].
\end{equation*}
Finally, we define the \emph{hitting set}  as
\begin{equation}\label{eq:HittingSetsDef}
 H_{j, R} = \{ - h_{j, R} \} \times \left[- \frac{A_{j} + B_{j}}{2}, \frac{A_{j} + B_{j}}{2} \right]
\end{equation}
with $- h_{j, R} = \inf \left\{ x < 0 : b_{q+2}^{(1)}(x^{\prime},y) = v_{j} \mathbbm{1}_{[- \frac{A_{j}}{2}, \frac{A_{j}}{2}]}(y) \quad \forall x^{\prime} \in [x,0] \right\} + \ell_{j+2}$.
The segment $H_{j, R}$ measures for how far outside of $R$ the velocity field remains a straight pipe (see Figures~\ref{fig:zoomout} and \ref{fig:mainsets}).
For all $j$, we define
\begin{equation}\label{eq:DefsOfSets}
 E_j = \bigcup_{R \in \mathcal{G}_j} E_{j, R}, \quad M_j = \bigcup_{R \in \mathcal{G}_j} M_{j, R}, \quad H_j = \bigcup_{R \in \mathcal{G}_j} H_{j, R}.
\end{equation}
Finally, throughout this subsection, $c$ denotes a constant independent on $q$ and $(x,y)$ which may vary from line to line.
We recall that for vector-valued maps $F$, we denote the $i-th$ component by $F^{(i)}$, in particular the $i$-th component of the flow will be denoted by $X_{t,s}^{\kappa_q, (i)}$.
}
{ \color{black}
\begin{figure}[ht]
\centering
\begin{subfigure}{\textwidth}
\begin{tikzpicture}[scale=0.75]
\small
\path[fill=orange, fill opacity=0.5] (0,-0.6) rectangle ++(10,-0.1);
\path[fill=orange, fill opacity=0.5] (0,0.6) rectangle ++(10,0.1);
\BMPipeWithRectanglesAndHitSet{0}{0}{0.6}{4}{3}{10}{2.5}{3}
\draw[black, thick , dotted] (0,3.55) rectangle (10,-3.55);
\path[fill=gray, fill opacity=0.25] (2.8,-0.15) rectangle (4.2,-1.0);
\draw[black, thick , dotted] (2.8,-0.15) rectangle (4.2,-1.0);
\draw[black] (0,3.2) node[anchor=west]{$R \in \mathcal{G}_{j-1}$};
\draw[blue, very thick] (0,0.5) -- (0,-0.5);
\draw[blue] (0,0) node[anchor=east]{$E_{j-1, R}$};
\end{tikzpicture}
\centering
\caption{A rectangle $R \in \mathcal{G}_{j-1}$ together with the sets introduced above. The set marked in orange is $M_{j-1, R}$ and the sets coloured in purple are $H_{j, \tilde{R}}$ for any  $\tilde{R} \in \mathcal{G}_{j}$ such that $\tilde{R} \subseteq R$.} \label{fig:zoomout}
\end{subfigure}
\begin{subfigure}{\textwidth}
\begin{tikzpicture}[scale=0.8]
\small
\draw[black, thick , dotted] (0,1) rectangle (10,-4);
\draw[black, thick] (0,1) -- (10,1);
\draw[black, thick] (8,0) -- (10,0);
\draw[black, thick] (0,0) -- (8,0) -- (8,-4);
\draw[black, thick] (0,-1) -- (5,-1) -- (5,-4);
\draw[black, thick, <->, dashed] (2,0) -- (2,-1);
\draw[black, thick, <->, dashed] (5,-2.7) -- (8,-2.7);
\draw[black, thick, ->] (0,-0.25) -- (1,-0.25);
\draw[black, thick, ->] (0,-0.5) -- (1,-0.5);
\draw[black, thick, ->] (0,-0.75) -- (1,-0.75);
\draw[black, thick, ->] (7.3,-1.5) -- (7.3,-2.0);
\draw[black, thick, ->] (6.5,-1.5) -- (6.5,-2.0);
\draw[black, thick, ->] (5.7,-1.5) -- (5.7,-2.0);
\path[fill=orange, fill opacity=0.5] (0,-2) rectangle ++(10,-0.3);
\path[fill=teal, fill opacity=0.25] (0,-3.3) rectangle ++(10,-0.7);
\draw[black, thick, <->, dashed] (10.2,-2.0) -- (10.2,-2.3);
\draw[black, very thick, red] (5,-1) -- (8,0);
\draw[black, very thick, olive] (5,-0.9) -- (8,0.1);
\draw[black, very thick, olive] (5,-1.1) -- (8,-0.1);
\draw[black] (2,-0.5) node[anchor=west]{$\overline{A}_j$};
\draw[black] (1,-0.5) node[anchor=west]{$v_{j-1}$};
\draw[black] (2,-3.6) node[anchor=west]{$\tilde{R}$};
\draw[black] (2,-2.1) node[anchor=west]{$M_{j-1, R}$};
\draw[black] (6,-2.7) node[anchor=north]{$A_j$};
\draw[black] (7.2,-1.7) node[anchor=west]{$v_j$};
\draw[red] (7.0,-0.6) node[anchor=west]{};
\draw[black, very thick, blue] (5.6,-3.3) -- (7.4,-3.3);
\draw[blue] (6.5,-3.3) node[anchor=north]{$E_{j, \tilde{R}}$};
\draw[black] (10.2,-2.25) node[anchor=west]{$2 \ell_{j-1}$};
\pgfmathsetseed{1233}
\BMotion{6.12}{-0.52}{13}{0.02}{0.05}{cyan, thick}{ }{-0.5}{0.031}
\fill[green] (6.12, -0.52) circle(0.05);
\draw[black, very thick, green] (6.2, -0.6) -- (6.12, -0.52);
\fill[blue] (6.2, -0.6) circle(0.05);
\draw[black, very thick, blue] (6.35,-0.65) -- (6.2, -0.6);
\BMotion{6.25}{-1.42}{20}{0.02}{0.05}{purple, thick}{ }{0.002}{0.0355}
\BMotion{7.04}{-2.5}{25}{0.02}{0.05}{darkgray, thick}{ }{-0.035}{0.046}
\BMotion{7}{-3.3}{20}{0.02}{0.05}{gray, thick}{ }{0.01}{0.042}
\fill[black] (7, -3.3) circle(0.05);
\path[fill=white, fill opacity=1] (-0.5,0) rectangle ++(0.47,-0.5);
\path[fill=gray, fill opacity=0.25] (6,-0.35) rectangle (6.5,-0.8);
\draw[black, thick , dotted] (6,-0.35) rectangle (6.5,-0.8);
\draw[purple] (9,-1.2) node[anchor=south]{$H_{j, \tilde{R}}$};
\draw[black, very thick, purple] (0,-1.2) -- (10,-1.2);
\end{tikzpicture}
\centering
\caption{Zoom in on the grey square in Figure \ref{fig:zoomout}. It contains all the sets of Figure~\ref{fig:zoomout}. Here we also represents the six steps of the proof of Lemma \ref{lemma:FullStability} starting from a point in $E_{j, \tilde{R}} \subset \tilde{R} \in \mathcal{G}_j$ and going to $E_{j-1, R}$ (not depicted in this figure).} \label{fig:mainsets}
\end{subfigure}
\hfill
\begin{subfigure}{0.3\textwidth}
\begin{tikzpicture}[scale=0.7]
\small
\draw[black, very thick, red] (5,-1) -- (8,0);
\draw[black, very thick, olive] (5,-2) -- (8,-1);
\draw[black, very thick, olive] (5,0) -- (8,1);
\pgfmathsetseed{1333}
\BMotion{5.65}{0.2}{3}{0.02}{0.05}{cyan, thick}{ }{-0.21}{0.035}
\BMotion{6.27}{-0.6}{20}{0.02}{0.05}{green, thick}{ }{-0.026}{0.046}
\BMotion{7.1}{-1.3}{22}{0.02}{0.05}{blue, thick}{ }{-0.033}{0.041}
\BMotion{7}{-2}{17}{0.02}{0.05}{purple, thick}{ }{0.0025}{0.04}
\path[fill=white, fill opacity=1] (4.7,0.5) rectangle ++(0.3,-0.4);
\draw[red] (6.6,-0.6) node[anchor=west]{\tiny $\partial {\{ b_{q+2}^{(1)} = v_{j-1} \} }  $};
\draw[black, thick, <->, dashed] (5.3,-1.9) -- (5,-1);
\draw[black, thick, <->, dashed] (8,0) -- (7.7,0.9);
\draw[black] (5.1,-1.4) node[anchor=west]{\tiny $2 \ell_{j-1}$};
\draw[black] (8,0.3) node[anchor=east]{\tiny $2 \ell_{j-1}$};
\end{tikzpicture}
\centering
\caption{Zoom in on the grey square in Figure \ref{fig:mainsets}. }
\end{subfigure}
\begin{subfigure}{0.6\textwidth}
\begin{tikzpicture}
\small
\draw[black, very thick, ->] (0,0) -- (8,0);
\draw[black, very thick] (1,-0.1) -- (1,0.1);
\draw[black, very thick] (3,-0.1) -- (3,0.1);
\draw[black, very thick] (5,-0.1) -- (5,0.1);
\draw[black, very thick] (7,-0.1) -- (7,0.1);
\draw[black] (7,0.1) node[anchor=south]{$\tau_{j+1}$};
\draw[black] (5,0.1) node[anchor=south]{$\tau_j$};
\draw[black] (3,0.1) node[anchor=south]{$\tau_{j-1}$};
\draw[black] (1,0.1) node[anchor=south]{$\tau_{j-2}$};
\draw[black] (3,-0.2) -- (1.2,-0.8);
\draw[black] (5,-0.2) -- (6.6,-0.8);
\draw[black, very thick, ->] (1,-1) -- (7,-1);
\draw[black, very thick] (6.6,-1.1) -- (6.6,-0.9);
\draw[black] (6.6,-1.2) node[anchor=north]{$\tau_j$};
\draw[black, very thick] (6,-1.1) -- (6,-0.9);
\draw[black] (6,-1.1) node[anchor=north]{${\tau}_{j-1}^1$};
\draw[black, very thick] (5.2,-1.1) -- (5.2,-0.9);
\draw[black] (5.2,-1.1) node[anchor=north]{${\tau}_{j-1}^2$};
\draw[black, very thick] (4.3,-1.1) -- (4.3,-0.9);
\draw[black] (4.3,-1.1) node[anchor=north]{${\tau}_{j-1}^3$};
\draw[black, very thick] (3.5,-1.1) -- (3.5,-0.9);
\draw[black] (3.5,-1.1) node[anchor=north]{${\tau}_{j-1}^4$};
\draw[black, very thick] (2.4,-1.1) -- (2.4,-0.9);
\draw[black] (2.4,-1.1) node[anchor=north]{${\tau}_{j-1}^5$};
\draw[black, very thick] (1.2,-1.1) -- (1.2,-0.9);
\draw[black] (1.2,-1.2) node[anchor=north]{$\tau_{j-1}$};
\end{tikzpicture}
\centering
\caption{The stopping times introduced  in Lemma \ref{lemma:FullStability}.}
\end{subfigure}
\caption{The six steps in the proof of Lemma~\ref{lemma:FullStability} in a rectangle $R \in \mathcal{G}_{j-1}$.} \label{fig:SixStepsFigure}
\end{figure}
\begin{proof}[Proof of Lemma~\ref{lemma:FullStability}.]
{
We consider a fixed $(x,y) \in E_{q-2}$ as in the statement. 
Call $\tau_{q-2} = T_0$ and define the stopping times $\tau_j \colon \Omega \to [- \infty, \tau_{j+1}]$ iteratively by
\[
 \tau_j (\omega) = \sup \{ s \in (- \infty, \tau_{j+1}] : X_{T_0,s}^{\kappa_q}(x,y,\omega) \in E_{j} \} \quad j = 1, \ldots, q-3.
\]
We fix some $\omega \in \widetilde{\Omega}_{T_0} \cap \{ T_0 \geq 1/5 \}$ and we claim the following.
\\
\textbf{Claim: } For all $2 \leq j \leq q-2$, we have $|\tau_j(\omega) - \tau_{j-1}(\omega)| \leq 6 a_{j-1}^{\eps}$.
\\
We postpone the proof and show how to conclude the proof based on this claim.
By definition of $E_1$, we have $X_{T_0, \tau_1}^{\kappa_q}(x,y,\omega) \in \supp (b_1)$.
By construction of $b_0$ as in \eqref{d:b-0} and the fact that the initial velocity is $v_0 = 1/4$}  and $\tau_1(\omega) \leq T_0(\omega) \leq 1$, it holds that 
\begin{equation}\label{eq:BeingInTheSupportOfBOne}
 X_{T_0, t}^{\kappa_q} (x,y, \omega) \in \supp (b_1) \,, \quad \forall t \in [\tau_1 - 1, \tau_1] \,.
\end{equation}
From the claim, we deduce that
\begin{align} \label{eq:inequality-t-0}
 |T_0(\omega) - \tau_1 (\omega)| \leq \sum_{\ell = 2}^{q-2} |\tau_{\ell-1}(\omega) - \tau_{\ell}(\omega)| \leq 6 \sum_{\ell = 2}^{+ \infty} a_{\ell - 1}^{\eps} \stackrel{{ \eqref{summability-eps}}}{\leq} 12 a_{1}^{\eps} < \frac{1}{5} \,.
\end{align}
Thus, since $T_0(\omega) \geq 1 / 5$, we deduce that $1 \geq \tau_1(\omega) \geq 0$.
Hence, \eqref{eq:BeingInTheSupportOfBOne} implies $X_{T_0, 0}^{\kappa_q}(x,y,\omega) \in \supp (b_1)$.
By construction of the rectangles in Section~\ref{sec:construction}, we have
\[
 \dist \left( \bigcup_{\tilde{R} \in \mathcal{R}_{q-2}} \tilde{R} \, , \, \supp(b_1) \right) \geq c > 0 { \quad \forall q \geq 5}
\]
 from which we deduce \eqref{eq:UsefulCorollary}.
It only remains to prove the claim.

This is done in 6 steps.
To lighten the the notation in this proof, we will write $X_{T_0, t}^{\kappa_q}$ instead of $X_{T_0 ,t}^{\kappa_q}(x,y,\omega)$ 
since $(x,y)$ and $\omega$ are fixed. For the same reason, stopping times are denoted without the dependence on $(x,y)$ and $\omega$.
{ Note that for $q$ sufficiently large
\begin{equation}\label{eq:DiffusivityMollifierGap}
 \sqrt{2 \kappa_q} \leq \frac{\ell_{j-1}}{10} \leq \frac{A_{j}^{1 + \eps}}{100} \quad \forall 1 \leq j \leq q-1.
\end{equation}
The first inequality follows from \eqref{eq:parameter:mollification}, \eqref{parameter:overlineA-q}, \eqref{parameter:kappa-q} and \eqref{d:eps-delta-1} while the second one follows from \eqref{eq:parameter:mollification} and \eqref{d:overlineA_q+1}.
Equation~\eqref{eq:DiffusivityMollifierGap} implies 
\begin{equation}\label{eq:DiffusivityMollifierGapConseq}
 10 \sqrt{2 \kappa_q} \leq \ell_j \quad \forall 1 \leq j \leq q-2.
\end{equation}
Moreover, note that 
\begin{equation}\label{eq:IntersectionConstraints}
 \left( \frac{\overline{A}_{j}}{A_j} \right)^{1 + \eps} \leq a_{j}^{\eps^2} \frac{v_j}{v_{j-1}} \quad \forall j \geq 1.
\end{equation}
This follows from \eqref{d:v-q+1} and \eqref{parameter:v-q}.
{Finally, observe
\begin{equation}\label{eq:AjmiunsoneSmallerThanLj}
 A_{j-1} \leq L_{j} \quad \forall j \geq 1.
\end{equation}
This follows from \eqref{parameter:A-q} and \eqref{parameter:L_q}.}
}
\\
{\textbf{ Proof of { Claim}.}}
By definition of the stopping time $\tau_j$ we have  $X_{T_0, \tau_j}^{\kappa_q} \in E_{j}$. 
Let $R \in \mathcal{R}_{j-1}$ and $\tilde{R} \in \mathcal{R}_j$ be such that
\[
 X_{T_0, \tau_j}^{\kappa_q} \in E_{j, \tilde{R}} \subseteq \tilde{R} \subseteq R.
\]
Recall that $R$, up to a rotation and a translation, can be written as
\[
 [0, L_{j-1}] \times \left[ - \frac{A_{j-1} + B_{j-1}}{2} , \frac{A_{j-1} + B_{j-1}}{2} \right].
\]
From now on, we work in these coordinates. 
Without loss of generality, we assume that
\[
 \tilde{R} \subseteq [0, L_{j-1}] \times \left[ - \frac{A_{j-1} + B_{j-1}}{2} , 0 \right].
\]
{ i.e. $\tilde{R}$ is located in the lower half of $R$ (see Figure~\ref{fig:zoomout}).}
\\
\textbf{Step 1:}
We define the stopping time
$$  {\tau}_{j-1}^{1} = \sup \{ s \in {(- \infty}, \tau_{j} ] : X_{T_0 ,s}^{\kappa_q} \in M_{j-1,R } \} \,. $$
From the construction of the velocity field, $u_{q+2}^{(1)} (X_{T_0 , s}^{\kappa_q}) = 0$
for any $\tau_j \geq s \geq  {\tau}_{j-1}^{1}$ since $\omega \in \tilde{\Omega}_{T_0} \cap \{ T_0 \geq 1/5 \}$.
Thus, by definition of $\tilde{\Omega}_{T_0}$,
\begin{equation}\label{eq:Step1Horizontal}
|X_{T_0 ,  s }^{\kappa_q , (1)} - X_{T_0 , {\tau}_{j}}^{\kappa_q,  (1)}| \leq \sqrt{2 \kappa_q} \stackrel{{\eqref{eq:DiffusivityMollifierGap}}}{\leq} \frac{A_{j}^{1+\varepsilon} }{10} \,, \quad \forall s \in [{\tau}_{j-1}^{1} , \tau_j ]\,.
\end{equation} 
This property in particular implies that $u^{(2)}_{q+2} (X_{T_0 ,  s }^{\kappa_q , (1)}) \equiv - v_j$ for any $s \in [{\tau}_{j-1}^{1} , \tau_j ]$ {(since the set $E_{j, \tilde{R}}$ is located $2 A_{j}^{1 + \eps}$ away from the extremity of the pipe, see Equation~\eqref{d:set-E-q})}, from which we deduce that
\begin{equation}\label{eq:Step1Stopping}
| \tau_{j}  -  {\tau}_{j-1}^{1} | \leq { \frac{ A_{j-1} + \sqrt{2 \kappa_q}}{v_j} \stackrel{\eqref{eq:AjmiunsoneSmallerThanLj}}{\leq} }  \frac{ L_j + \sqrt{2 \kappa_q}}{v_j} \stackrel{{ \eqref{eq:constant}}}{\leq} a_{j}^\varepsilon
\end{equation}
where we used the definition of the backward stochastic flow, 
{ and the fact that $M_{j-1, R}$ and $E_{j, \tilde{R}}$ are vertically separated by a distance $A_{j-1}$.} 
\\
\textbf{Step 2:}
We define the stopping time
$$ {\tau}_{j-1}^{2} = \sup \left\{ s \in {(- \infty},  {\tau}_{j-1}^{1} ] : X_{T_0 ,s}^{\kappa_q} \notin I_{\ell_{j-1} } (M_{j-1, \tilde{R}}) \right\} \,. $$
{ which due to \eqref{eq:DiffusivityMollifierGap} corresponds to a time where the backwards trajectory has already passed through  the intersection problem.}
We also define 
$$\tau= \sup \left\{ s \in {(- \infty},  {\tau}_{j-1}^{1} ] : | X_{T_0 ,s}^{\kappa_q,  (1)}  - X_{T_0 , \tau_j}^{\kappa_q,  (1)} | \geq  \frac{A_{j}^{1+\varepsilon} }{{5}} \right\} \,. $$
Note that for any $s \in [\tau,  {\tau}_{j-1}^{1}]$, we have $| X_{T_0 ,s}^{\kappa_q ,  (1)} - X_{T_0 , \tau_j}^{\kappa_q , (1)} | \leq \frac{A_j^{1+\varepsilon}}{5}$ and hence
\begin{align} \label{eq:constant-flow}
u^{(2)}_{q+2}(X_{T_0 , s}^{\kappa_q} ) = - v_j \quad \forall s \in [\tau,  {\tau}_{j-1}^{1}].
\end{align} 
We { show that} $\tau \leq  {\tau}_{j-1}^{2}$.
By contradiction, suppose that $\tau >  {\tau}_{j-1}^{2}$. Then, by definition of the backward stochastic flow and definition of $\tau$
$$\frac{A_j^{1+ \varepsilon}}{{ 10}} \leq |X_{T_0 , \tau}^{\kappa_q , (1)} - X_{T_0 , {{\tau}_{j-1}^{1}}}^{\kappa_q ,  (1)}| \leq  | {\tau}_{j-1}^{1} - \tau | v_{j-1} + \sqrt{2 \kappa_q} \,,$$
and
$$2 \ell_{j-1} + \sqrt{2 \kappa_q} \geq |X_{T_0, \tau}^{\kappa_q , (2)} - X_{T_0, {{\tau}_{j-1}^{1}}}^{\kappa_q , (2)}| \geq | {\tau}_{j-1}^{1} - \tau | v_{j} - \sqrt{2 \kappa_q}$$
which leads to
\begin{equation}\label{eq:IntersectionStuffEquation}
 \frac{A_{j}^{1+\varepsilon} - { 10}\sqrt{2 \kappa_q} }{{ 10}v_{j-1}} \leq { |\tau_{j-1}^1 - \tau| \leq }\frac{2 \ell_{j-1} + 2 \sqrt{2 \kappa_q}}{v_j}.
\end{equation}
{ Due to Equation~\eqref{eq:DiffusivityMollifierGap},}
Equation \eqref{eq:IntersectionStuffEquation} yields
\[
 \frac{A_{j}^{1+\eps}}{{ 20}v_{j-1}} \leq \frac{4 \ell_{j-1}}{v_j} \stackrel{\eqref{eq:parameter:mollification}}{=} \frac{4 \overline{A}_{j}^{1+\eps}}{v_j} \quad \Rightarrow \quad  \frac{v_j}{v_{j-1}} \leq { 80} \left( \frac{\overline{A}_j}{A_j} \right)^{1+\eps}.
\]
This contradicts
 { \eqref{eq:IntersectionConstraints} and thus $\tau \leq \tau_{j-1}^2$.}
Thanks to this and property \eqref{eq:constant-flow} we deduce that 
\begin{equation}\label{eq:GoodStabStepTwoConclusion}
| {\tau}_{j-1}^{1} -  {\tau}_{j-1}^{2} | { \leq \dfrac{3 \ell_{j-1} + \sqrt{2 \kappa_q}}{v_j}} \leq \frac{  L_j+ \sqrt{2 \kappa_q}}{v_j} \stackrel{{{\eqref{eq:constant}}}}{\leq} a_{j}^\varepsilon \,.
\end{equation}
\\
\textbf{Step 3:}
We define
$$ {\tau}_{j-1}^{3} =  \sup \left\{ s \in {(- \infty},  {\tau}_{j-1}^{2} ] : X_{T_0 ,s}^{\kappa_q}  { \not\in \{ b_{q+2}^{(2)} = v_j \}[2 \ell_j]} \right\} \,,$$
and observing that $u_{q+2}^{(1)} (X_{T_0 , s}^{\kappa_q}) = 0$ for any $s \in [ {\tau}_{j-1}^{3}, {\tau}_{j-1}^{2}]$, we deduce 
$$|  {\tau}_{j-1}^{3}   -  {\tau}_{j-1}^{2} | { \leq \frac{A_{j-1} + \sqrt{2 \kappa_q}}{v_j}  \stackrel{\eqref{eq:AjmiunsoneSmallerThanLj}}{\leq} } \frac{ L_j+ \sqrt{2 \kappa_q}}{v_j} \stackrel{{{\eqref{eq:constant}}}}{\leq} a_{j}^\varepsilon  \,. $$
From the previous steps and { since $u_{q+2}^{(1)} (X_{T_0 , s}^{\kappa_q}) = 0 $
for any $s \in [ {\tau}_{j-1}^{3},   {\tau}_{j-1}^{2}]$} we have 
\begin{equation}\label{eq:DistanceAtThirdTime}
| X_{T_0,  {\tau}_{j-1}^{3}}^{\kappa_q,  (1)}  - X_{T_0, \tau_{j}}^{\kappa_q ,  (1)} |  \leq \frac{3}{10} A_{j}^{1 + \varepsilon} \,.
\end{equation}
\\
\textbf{Step 4:}
We can now define 
$$ {\tau}_{j-1}^{4} =  \sup \left\{ s \in {(- \infty},  {\tau}_{j-1}^{3} ] : X_{T_0 ,s}^{\kappa_q} \in { \partial \{ b_{q+2}^{(1)} = v_{j-1} \} } \right\} \,,$$
and {since} $u_{q+2}^{(1)} (X_{T_0 ,s}^{\kappa_q}) \geq 0$, $u_{q+2}^{(2)} (X_{T_0 ,s}^{\kappa_q}) \leq - v_{j}/2$ for any $s \in [ {\tau}_{j-1}^{4},   {\tau}_{j-1}^{3}]$, we get
\begin{align} \label{new:eq-distance-y}
|X_{T_0 , {\tau}_{j-1}^4}^{\kappa_q,  (2)}  - X_{T_0 , {\tau}_{j-1}^3}^{\kappa_q,  (2)}  | \leq { 3} \ell_j + \sqrt{2 \kappa_q} \stackrel{\eqref{eq:DiffusivityMollifierGapConseq}}{\leq} { 4} \ell_j.
\end{align}
{ Therefore, $|\tau_{j-1}^4 - \tau_{j-1}^3| \leq \frac{4 \ell_j}{v_j / 2} = \frac{8 \ell_j}{v_j}$.
Since $0 \leq u_{q+2}^{(1)} (X_{T_0 ,s}^{\kappa_q}) \leq v_{j-1}$ for any $s \in [{\tau}_{j-1}^{4}, {\tau}_{j-1}^{3}]$, we find 
\begin{align*}
 |X_{T_0 , \tau_{j-1}^4}^{\kappa_q, (1)} - X_{T_0 , \tau_{j-1}^3}^{\kappa_q, (1)}| &\leq \dfrac{8 \ell_j v_{j-1}}{v_j} + \sqrt{2 \kappa_q} \stackrel{\eqref{eq:DiffusivityMollifierGapConseq}}{\leq} \dfrac{9 \ell_j v_{j-1}}{v_j} \stackrel{\eqref{d:v-q+1}, \eqref{eq:parameter:mollification}}{=} \dfrac{9 \overline{A}_{j+1}^{1 + \eps} A_j}{\overline{A}_j} \\
 &\leq \dfrac{1}{10} A_j^{1 + \eps}. 
\end{align*}
Thus, combined with Equation~\eqref{eq:DistanceAtThirdTime}, we find
\[
 |X_{T_0 , \tau_{j-1}^4}^{\kappa_q, (1)} - X_{T_0 , \tau_{j}}^{\kappa_q, (1)}| \leq \dfrac{4}{10} A_j^{1 + \eps},
\]
i.e. the trajectory has barely moved horizontally at this point.
}
\\
\textbf{Step 5:}
We now define 
$$ {\tau}_{j-1}^{5} (\omega) =  \sup \left\{ s \in {(- \infty},  {\tau}_{j-1}^{4} ] : X_{T_0 ,s}^{\kappa_q} \in { \{ b_{q+2}^{(1)} = v_{j-1} \}[2 \ell_j]} \right\}\,,$$
Since, $u_{q+2}^{(1)} ( X_{T_0 ,s}^{\kappa_q} ) \geq v_{j-1}/4$ for any $s \in [ {\tau}_{j-1}^{5},   {\tau}_{j-1}^{4}]$,  we find
$$ |  {\tau}_{j-1}^{4} - {\tau}_{j-1}^{5}  |  { \leq 4 \frac{L_{j-1} + \sqrt{2 \kappa_q}}{v_{j-1}} \stackrel{\eqref{eq:constant}}{\leq} a_{j-1}^\varepsilon \,,}$$
{ Due} to the definition of $\mathcal{G}_j$ in Subsection \ref{subsec:dissipative} we have  
\begin{equation}\label{eq:BoundsThanksToGoodRectangles}
|  X_{T_0, {\tau}_{j-1}^{5} }^{\kappa_q, (2)} | { \leq A_{j-1}/2 - n_j^{1 - \eps} \overline{A}_j} \stackrel{{ \eqref{d:overlineA_q+1}}}{\leq} A_{j-1}/2 - 3 A_{j-1}^{1+ \varepsilon} \,.
\end{equation}
{
 Since $X_{T_0, {\tau}_{j-1}^{5} }^{\kappa_q} \in \{ b_{q+2}^{(1)} = v_{j-1} \}[2 \ell_j]$, and
  $$\{ b_{q+2}^{(1)} = v_{j-1} \}[2 \ell_j] \subseteq \{ u_{q+2}^{(1)} = v_{j-1} \}[\ell_j]$$
and $u_{q+2}^{(2)} = 0$ whenever $u_{q+2}^{(1)} = v_{j-1}$ inside $R$, we deduce that
\begin{equation}\label{InsideAStripBetweenStep5And6}
 |X_{T_0, s }^{\kappa_q, (2)} - X_{T_0, {\tau}_{j-1}^{5} }^{\kappa_q, (2)}| \leq \sqrt{2 \kappa_q}  \stackrel{\eqref{eq:DiffusivityMollifierGap}}{\leq}  A_{j-1}^{1+ \varepsilon} \quad \forall s \in [\tau_{j-1}, \tau_{j-1}^{5}] \,,
\end{equation}
}
\\
\textbf{Step 6:}
In this last step, we prove $| {\tau}_{j-1} -   {\tau}_{j-1}^{5}| \leq a_{j-1}^\varepsilon$.
{
Due to \eqref{eq:BoundsThanksToGoodRectangles} and \eqref{InsideAStripBetweenStep5And6},
\[
 | X_{T_0, s }^{\kappa_q, (2)} | \leq A_{j-1}/2 - 2 A_{j-1}^{1+ \varepsilon} \quad \forall s \in [\tau_{j-1}, \tau_{j-1}^{5}],
\]
which means that the trajectory $s \mapsto X_{T_0, s }^{\kappa_q}$ intersects $E_{j-1, R}$. Since the length of the rectangle $R$ is $L_{j-1}$, we deduce}
$$ | {\tau}_{j-1} -   {\tau}_{j-1}^{5}| \leq \frac{ L_{j-1} + \sqrt{2 \kappa_q}}{v_{j-1}} \stackrel{{{\eqref{eq:constant}}}}{\leq} a_{j-1}^\varepsilon \,.$$
{ Thus, from the 6 steps we conclude}
$$| {\tau}_{j-1} -  \tau_j | \leq {  |{\tau}_{j-1} -  \tau_{j-1}^5| + \sum_{\ell = 1}^4 |\tau_{j-1}^{\ell + 1} - \tau_{j-1}^{\ell}| +  |\tau_{j-1}^5 - {\tau}_{j}| \leq } 6 a_{j-1}^\varepsilon$$
and hence the claim holds.
\end{proof} }
}
 
 { \color{black}
 \begin{proof}[Proof of { Lemma~\ref{lemma:FirstStepOfStab}}]
 { Recall that $c$ denotes a constant independent on $q$ and $(x,y)$ which may vary from line to line.}
Fix an arbitrary $(x,y) \in D_q$ and define the stopping time $T_0$ as
 $$T_0 (\omega) =   \sup \{ s \in ({\color{black} - \infty},T_1 (\omega)] : X_{1,s}^{\kappa_q}(x,y,\omega) \in E_{q-2} \} \vee 0 \,.$$
 We will prove that 
 \begin{equation}\label{eq:StoppingTimeTZero}
  \Prob [T_0 \geq 1/5] \geq c > 0\,.
 \end{equation}
 The proof is completed in 5 steps. 
Let $\tilde{R} \in \mathcal{G}_{q}$, $R \in \mathcal{G}_{q-1}$ and $\hat{R} \in \mathcal{G}_{q-2}$ be the unique rectangles such that
\[
 (x,y) \in D_{q, \tilde{R}} \subseteq \tilde{R} \subseteq R \subseteq \hat{R}.
\]
The rectangle $R$ can be written as
\[
 R = [0, L_{q-1}] \times \left[ - \frac{A_{q-1} + B_{q-1}}{2}, \frac{A_{q-1} + B_{q-1}}{2} \right]
\]
up to a rotation and a translation.
From now and until stated otherwise, we work in these coordinates.
Without loss of generality, we assume that
\[
 \tilde{R} \subseteq [0, L_{q-1}] \times \left[ - \frac{A_{q-1} + B_{q-1}}{2}, 0 \right] \subseteq R.
\]
{ i.e. $\tilde{R}$ is located in the lower half of $R$.}
For any {$ 0 \leq \tau^{\prime} \leq \tau \leq 1$} stopping times, we introduce the notation
\[
 \widetilde{\Omega}_{\tau^{\prime}, \tau} = \left\{ \omega \in \Omega : \sup_{t \in [\tau^{\prime}, \tau]} | W_t - W_{\tau}| \leq 1/2 \right\} \,
\]
{ and note that $\Prob(\widetilde{\Omega}_{\tau^{\prime}, \tau}) \geq {c} > 0$.}
\\
\textbf{Step 1:}
{ Recall that $X_{1, T_1}^{\kappa_q}(x,y, \omega) \in H_q$ for all $\omega \in \{ T_1 \geq 1 / 2 \}$.}
We introduce the stopping time
\[
 \tau^{1}(\omega) = \sup \left\{ s \in {\color{black} ( - \infty }, T_1 ] : X_{1 ,s}^{\kappa_q}(x,y,\omega) \in (\{ u_{q+2}^{(1)} = v_{q-1}\} \cap R) \left[ \sqrt{2 \kappa_q} \right] \right\} \vee \frac{9}{20}
\]
and define the set
{
\[
 \Omega^1 = \left\{ \omega \in \Omega : \sup_{t \in [{T_1 - \frac{1}{20}}, T_1]} W_t^{(2)} - W_{T_1}^{(2)} \geq 2 , \, \sup_{t \in [{T_1 - \frac{1}{20}}, T_1]} |W_t - W_{T_1}| \leq 10 \right\}.
\]}
{ Note that $\Omega^1$ is independent of $T_1$ and $\Prob(\Omega^1) \geq c > 0$.}
Since $\overline{A}_q \ll \sqrt{\kappa_q}$, $u_{q+2}^{(1)} \geq 0$ in $R$ and 
\[
 u_{q+2}^{(2)} \leq 0 \quad \text{in} \quad \left[0, L_{q-1}\right] \times \left[- \frac{A_{q-1} + B_{q-1}}{2}, 0 \right] \subseteq R, 
\]
 we deduce that for all $\omega \in \Omega^1 \cap \{ T_1 \geq 1/2 \}$, we have $\tau^1(\omega) > \frac{9}{20}$.
 { Indeed, if $\omega \in \Omega^1 \cap \{ T_1 \geq 1/2 \}$ then $X_{1, s}^{\kappa_q}(x,y,\omega)$ moves by at least $2\sqrt{2 \kappa_q}$ towards the center of the $q$-th pipe in the time interval $[T_1 - \frac{1}{20}, T_1]$. Hence $\tau^1(\omega) > T_1(\omega) - \frac{1}{20} \geq \frac{9}{20}$ for all $\omega \in \Omega^1 \cap \{ T_1 \geq 1/2 \}$. Thus, by independence of $\Omega^1$ and $T_1$ {\color{black} and \eqref{eq:T_1} }
 \[
 \Prob (\tau^1 > {9}/{20}) \geq   \Prob (\tau^1 > {9}/{20}, T_1 \geq 1 / 2) \geq \Prob (\Omega^1 \cap \{ T_1 \geq 1/2 \}) \geq c > 0.
 \]}
\\
\textbf{Step 2:}
We define
\[
 \tau^{2}(\omega) = \sup \left\{ s \in {\color{black} ( - \infty } , \tau^{1}] : X_{1 ,s}^{\kappa_q}(x,y,\omega) \in E_{q-1, R} \right\} \vee \frac{2}{5}.
\]
{ Since $X_{1, \tau^1}^{\kappa_q}(x,y, \omega) \in (\{ u_{q+2}^{(1)} = v_{q-1}\} \cap R) [ \sqrt{2 \kappa_q} ]$ and $u_{q+2}^{(2)} = 0$ in the set $\{ u_{q+2}^{(1)} = v_{q-1} \} \cap R$, we deduce that $|X_{1, s}^{\kappa_q, (2)}(x,y, \omega) - X_{1, \tau^1}^{\kappa_q, (2)}(x,y, \omega)| \leq \sqrt{\kappa_q}$}
for all $\omega \in \Omega_{\tau^{2}, \tau^{1}} \cap { \{ \tau^1 > {9}/{20}\}}$ { and all $s \in [\tau^2, \tau^1]$}. {  Therefore, since the length of $R$ is $L_{q-1}$, we deduce that for all $\omega \in \Omega_{\tau^{2}, \tau^{1}} \cap \{ \tau^1 > {9}/{20}\}$}
\[
 |\tau_{q-1}^{2} - \tau_{q-1}^{1}| \leq \frac{L_{q-1} + \sqrt{2 \kappa_q}}{v_{q-1}} \stackrel{{\eqref{eq:constant}}}{\leq} a_{q-1}^{\eps}.
\]
\\
\textbf{Step 3:}
We now change coordinates. 
Recall that $\hat{R}$ can be written, up to a rotation and a translation, as
\[
 \left[ 0, L_{q-2} \right] \times \left[ - \frac{A_{q-2} + B_{q-2}}{2} , \frac{A_{q-2} + B_{q-2}}{2} \right].
\]
From now, until the end of the proof, we work in these coordinates. Without loss of generality, we assume that
\[
 R \subseteq \left[ 0, L_{q-2} \right] \times \left[ - \frac{A_{q-2} + B_{q-2}}{2} , 0 \right] \subseteq \hat{R}.
\]
{ i.e. $R$ is located in the lower half of $\hat{R}$.}
We define the stopping time
\[
 \tau^{3}(\omega) = \sup \left\{ s \in {\color{black} ( - \infty }, \tau^{2}] : X_{1 ,s}^{\kappa_q}(x,y,\omega) \in { I_{2 \ell_{q+2}} ( \{ b_{q+2}^{(1)} = v_{q-2} \} )} \right\} \vee { \frac{7}{20}}.
\]
 { In this step, we need to take into account the intersection problem. However, as outlined in Subsection \ref{subsec:Heuristics-stability}, the intersection problem does not destroy stability in a $(q-1)$-th pipe which is what we are considering here. 
 To prove this it suffices to observe that Steps 1,2, and 3 of Lemma~\ref{lemma:FullStability} (where the intersection problem is treated) hold true with $j = q-1$. Indeed, { as pointed out in these steps in the proof of Lemma~\ref{lemma:FullStability}}, we only relied on \eqref{eq:constant}, \eqref{eq:DiffusivityMollifierGap} and \eqref{eq:IntersectionConstraints} which hold true in the case $j = q-1$.
 Thus for all $\omega \in \Omega_{\tau^{3}, \tau^{2}} \cap \Omega_{\tau^{2}, \tau^{1}} \cap { \{ \tau^1 > {9}/{20}\}}$
 \[
  |X_{1,s}^{\kappa_q, (1)}(x,y,\omega) - X_{1,\tau^2}^{\kappa_q, (1)}(x,y,\omega)| \leq \frac{3}{10} A_{q-1}^{1 + \eps} \quad \forall s \in [\tau^3, \tau^2].
 \]
 Hence $u_{q+2}^{(2)}(X_{1,s}^{\kappa_q, (1)}(x,y,\omega)) = - v_{q-1}$ for all $s \in [\tau^3, \tau^2]$ from which we deduce that for all $\omega \in \Omega_{\tau^{3}, \tau^{2}} \cap \Omega_{\tau^{2}, \tau^{1}} \cap { \{ \tau^1 > {9}/{20}\}}$
 \[
  |\tau^3 - \tau^2| \leq \dfrac{L_{q-1} + \sqrt{2 \kappa_q}}{v_{q-1}} \stackrel{{\eqref{eq:constant}}}{\leq} a_{q-1}^{\eps},
 \]
 where we used that $E_{q-1, R}$ and { ${ I_{2 \ell_{q+2}} ( \{ b_{q+2}^{(1)} = v_{q-2} \} )}$}  are vertically separated by a distance at most $L_{q-1}$.} 
\\
\textbf{Step 4:}
We define the stopping time
\[
 \tau^{4}(\omega) = \sup \left\{ s \in {\color{black} ( - \infty }, \tau^{3}] : X_{1 ,s}^{\kappa_q}(x,y,\omega) \in \{ u_{q+2}^{(1)} = v_{q-2} \}\left[ \sqrt{2 \kappa_q} \right] \right\} \vee { \frac{3}{10}}
\]
and the set
{
\[
 \Omega^2 = \left\{ \omega \in \Omega : \sup_{t \in [{ \tau^3-\frac{1}{20}}, \tau^3]} W_t^{(2)} - W_{\tau^3}^{(2)} \geq 2 , \, \sup_{t \in [{ \tau^3-\frac{1}{20}}, \tau^3]} |W_t - W_{\tau^3}| \leq 10 \right\}.
\]}
{ Note that $\Prob(\Omega^2) \geq c > 0$}.
{ Recall that $X_{1, \tau^3}(x,y, \omega) \in { I_{2 \ell_{q+2}} ( \{ b_{q+2}^{(1)} = v_{q-2} \} )}$.}
Since $\ell_{q+2} \ll \sqrt{\kappa_q}$ {  and
\[
 u_{q+2}^{(1)} \geq 0, \, u_{q+2}^{(2)} \leq 0 \text{ in $\left[ 0, L_{q-2} \right] \times \left[ - \frac{A_{q-2} + B_{q-2}}{2} , 0 \right]$}
\]
}
we deduce that for all $\omega \in \Omega^2 \cap \Omega_{\tau^{3}, \tau^{2}} \cap \Omega_{\tau^{2}, \tau^{1}} \cap { \{ \tau^1 > {9}/{20}\}}$, we have $\tau^4(\omega) > { 3/10}$.
 { Indeed, if $\omega \in \Omega^2 \cap \Omega_{\tau^{3}, \tau^{2}} \cap \Omega_{\tau^{2}, \tau^{1}} \cap { \{ \tau^1 > {9}/{20}\}}$ then $X_{1, s}^{\kappa_q}(x,y,\omega)$ moves by at least $2\sqrt{2 \kappa_q}$ towards the center of the $(q-2)$-th pipe in the time interval $[\tau^3 - \frac{1}{20}, \tau^3]$. Hence $\tau^4(\omega) > \tau^3(\omega) - \frac{1}{20} \geq \frac{3}{10}$ for all $\omega \in \Omega^2 \cap \Omega_{\tau^{3}, \tau^{2}} \cap \Omega_{\tau^{2}, \tau^{1}} \cap { \{ \tau^1 > {9}/{20}\}}$. Thus, by independence of increments
 \[
  \Prob(\tau^4 > 3 / 10) \geq \Prob( \Omega^2 \cap \Omega_{\tau^{3}, \tau^{2}} \cap \Omega_{\tau^{2}, \tau^{1}} \cap { \{ \tau^1 > {9}/{20}\}}) \geq c > 0.
 \]
 } 
\\
\textbf{Step 5:}
We define
\[
 \tau^{5}(\omega) = \sup \left\{ s \in {\color{black} ( - \infty } , \tau^{4}] : X_{1 ,s}^{\kappa_q}(x,y,\omega) \in E_{q-2, \hat{R}} \right\} \vee \frac{1}{5}.
\]
{ Recall that $X_{1, \tau^4}(x,y, \omega) \in \{ u_{q+2}^{(1)} = v_{q-2} \}\left[ \sqrt{2 \kappa_q} \right]$. Since $u_{q+2}^{(2)} = 0$ in $\{ u_{q+2}^{(1)} = v_{q-2} \} \cap \hat{R}$, we find that for all 
$$\omega \in  \Omega_{\tau^{5}, \tau^{4}} \cap \Omega^2 \cap \Omega_{\tau^{3}, \tau^{2}} \cap \Omega_{\tau^{2}, \tau^{1}} \cap \Omega^1 \cap \{ T_1 \geq 1/2 \},$$ 
we have 
\[
|X_{1, s}(x,y, \omega) - X_{1, \tau^4}(x,y, \omega)| \leq \sqrt{\kappa_q} \quad  \forall s \in [\tau^5, \tau^4].
\]
Thus $u_{q+2}^{(1)}(X_{1, s}(x,y, \omega)) = v_{q+2}$ for all $s \in [\tau^5, \tau^4]$.
Therefore, since the length of $\hat{R}$ is $L_{q-2}$}
we deduce that
\[
 |\tau^5 - \tau^4| { \leq \frac{L_{q-2} + \sqrt{2 \kappa_q}}{v_{q-2}}} \stackrel{{\eqref{eq:constant}}}{\leq} a_{q-2}^{\eps} \quad \forall \omega \in \Omega_{\tau^{5}, \tau^{4}} \cap { \{ \tau^4 > 3 / 10 \} }.
\]
This ends Step 5.
{ Thus, by independence of increments
\[
 \Prob (\tau^5 > 1 / 5) \geq \Prob(\Omega_{\tau^{5}, \tau^{4}} \cap \{ \tau^4 > 3 / 10 \}) \geq c > 0.
\]
Finally, since 
$T_0 = \tau^5$ on $\{ \tau^5 > 1 / 5 \}$, we have \eqref{eq:StoppingTimeTZero}.
}
 \end{proof}
 }

\color{black}
\section{Proof of Theorem \ref{thm-main}}\label{sec:ProofMainThm}
In this section we prove that all the assumptions of Proposition \ref{prop:criterion} are satisfied for the sequence of velocity fields $\{ u_{q+2} \}_{q \in 2+4 \N}$ and the sequence of diffusivity parameters $\{ \kappa_q \}_{q \in 2+4 \N}$.
Recall the definition of the dissipative sets $D_q$ in Item~\ref{item:ItemTwoInLemmaGronwallDissPeriod} in Lemma~\ref{lemma:AboutTheStructureOfHolderFields}. We stress again that
\[
 \inf_{q \geq 1} \Leb^2(D_q) >0.
\]
We consider $X_{t,s}^{{ \kappa_q}} $ the backward flow of $u_{q+2}$ with noise parameter $\sqrt{2 \kappa_q}$ defined in Section \ref{sec:choice}. Firstly, we notice that   
\begin{align} \label{eq:estimate:u-q+2}
\frac{\| u - u_{q+2} \|_{L^2}^2}{\kappa_q } \lesssim \frac{v_{q+2}^2}{\kappa_q} \lesssim a_{q+1}^{- 2 \eps (1 + \delta) + \frac{2 \delta}{2 + \delta}} \lesssim a_{q+1}^{\eps} \to 0 \,, 
\end{align}
thanks to Equation \eqref{eq:BoundBetweenApproxAndFinal}, Equations \eqref{parameter:kappa-q} and \eqref{parameter:v-q} in Lemma~\ref{parameter} and the fact that $0< \eps \ll \delta < 1$. \\
Our goal is now to prove the following: \\
\textbf{Main Claim:}
There exist two constants ${c}_1, {c}_2, Q > 0$ depending only on $a_0$ such that for any $q \geq Q$ and any $x \in D_q$, there exist two sets $\Omega_{q,1,x}, \Omega_{q,2,x} \subseteq \Omega$ for which:
\begin{itemize}
 \item $\min \{ \Prob (\Omega_{q,1,x}), \Prob (\Omega_{q,2,x}) \} \geq {c}_1$ for all $q \geq Q$;
 \item for all $q \geq Q$, we have
 \begin{equation}\label{eq:TrajectoriesRemainClose}
 X_{{ 1},0}^{{ \kappa_q}}(x, \omega) \in B_{{c}_2}(x) \quad \forall \omega \in \Omega_{q,1,x};
 \end{equation}
 and
 \begin{equation}\label{eq:TrajectoriesGoFar}
  X_{{ 1},0}^{{ \kappa_q}}(x, \omega) \not \in B_{2 {c}_2}(x) \quad \forall \omega \in \Omega_{q,2,x}.
 \end{equation}
\end{itemize}
Assuming that this last claim holds, the proof follows from Proposition~\ref{prop:criterion}. Indeed, the main claim implies that
\begin{equation} \label{variance:dissipative}
 \inf_{x \in D_q} \Expect \left[ |X_{{ 1},0}^{{ \kappa_q}}(x, \omega) - \Expect [X_{{ 1},0}^{{ \kappa_q}}(x, \cdot)]|^2 \right] \geq \frac{{c}_1 {c}_2^2}{4}.
\end{equation}
Since this last inequality holds for all $q \geq Q$, all the assumptions of Proposition~\ref{prop:criterion} are satisfied and the proof of Theorem~\ref{thm-main} follows. 
\\
\\
\textbf{Proof of Main Claim:} { Throughout the proof, $c$ denotes a constant independent of $q$ and $(x,y)$ which may vary from line to line. We recall that for vector-valued maps $F$, we denote the $i-th$ component by $F^{(i)}$.} The proof is divided into two parts. In the first part, we prove that there exists a set $\Omega_{q,1,x} \subset \Omega$ such that \eqref{eq:TrajectoriesRemainClose} holds. In the second part, we show that for any integer $q \geq 1$, there exists a set $\Omega_{q,2} \subset \Omega$ such that \eqref{eq:TrajectoriesGoFar} holds. \\
\textbf{Part 1: (Proof of \eqref{eq:TrajectoriesRemainClose})} 
We now fix $x \in D_q$, then $x \in D_{q,R}$ for some $R \in \mathcal{G}_q$ (see \eqref{d:G_q}). We define the stopping time
$$\tau (\omega) = \sup \left \{ s \in ({\color{black} - \infty},{ 1}]: X_{{ 1},s}^{{ \kappa_q}} (x, \omega) \notin I_{\frac{\sqrt{\kappa_q}}{100}} (D_{q,R}) \right \} {\color{black} \vee 0} \,.$$
We claim that there exists a set $\Omega_{q,1,x} \subset \Omega$ with $\mathbb{P} (\Omega_{q,1,x}) \geq c >0$, 
such that $\tau (\omega ) =0 $ for any $\omega \in \Omega_{q,1,x}$.
 Let 
\[
 {\Omega}_{1} = \left\{ \omega \in \Omega : \sup_{t \in [0,{ 1}]} |W_t - W_{ 1}| \leq \dfrac{1}{200} \right\}.
\]
It is clear that $\Prob (\Omega_1) \geq c >0$. We work in the coordinates of $R$ and we denote $u_{q+2} = (u_{q+2}^{(1)} , u_{q+2}^{(2)})$.
Recall the collection of rectangles $\mathcal{F}_q(R)$ coming from Item~\ref{item:LemmaTwoPointThree} in Lemma~\ref{lemma:PropertiesOfVelocityFieldsLInftyRectangles} and note that it follows from Item~\ref{item:AboutTheStructureOfHolderFieldsItemTwo} of Lemma~\ref{lemma:AboutTheStructureOfHolderFields} that
\begin{align} \label{eq:proof-5.1}
\int_{{\Omega}_{1}} & \int_{{\tau(\omega)}}^{{ 1}} \left | u_{q+2}^{(2)} ( X_{{ 1},s}^{{ \kappa_q}} (x , \omega)) \right | ds   \ d\mathbb{P} (\omega) 
\\
& \leq  \int_0^{{ 1}} v_{q+1}  \mathbb{P} \left[ \Omega_1 \cap \left\{ \omega \in \Omega : X_{{ 1},s}^{{ \kappa_q}} (x , \omega) \in \bigcup_{\overline{R} \in \mathcal{F}_q(R)} \overline{R} \right\} \right] ds. \notag
\end{align}
{ To bound the last term we need to control the advection term in the first component and use that the rectangles $\overline{R} \in \mathcal{F}_q(R)$ are distant each other at least $\frac{B_{q+1}}{4}$ in the first component thanks to \ref{item:LemmaTwoPointThree} of Lemma \ref{lemma:PropertiesOfVelocityFieldsLInftyRectangles}.  }
It is clear from property \ref{item:ItemThreeInLemmaGronwallDissPeriod} in Lemma \ref{lemma:gronwall+periodicity+dissipative} that $u_{q+2}$ is $A_{q+1} + B_{q+1}$ periodic 
{ in the first variable inside}  the dissipative set $I_{\sfrac{\sqrt{\kappa_q}}{ 150}} (D_{q,R})$ and by Item~\ref{item:odd} of Lemma~\ref{lemma:PropertiesBranchingMerging} $u_{q+2}^{(1)}$ is zero-average with respect to the first variable inside any $\tilde{R} \in \mathcal{R}_{q+1}(R)$, see Figure~\ref{fig:PeriodicityFigure}. We slightly abuse the notation denoting as $u_{q+2}$ the $A_{q+1} + B_{q+1}$ periodic extension in the first variable.
 Hence, an It\^o-Tanaka trick applied to $u_{q+2}^{(1)}$ yields }
\begin{equation}\label{eq:itotanaka}
\sup_{t \in [0,{ 1}]} \int_{\Omega} \left | \int_t^{{ 1}} {u}_{q+2}^{(1)}(X_{1,s}^{{ \kappa_q}}) ds \right | \, {\color{black}d \Prob} \lesssim a_{q+1}^{- \varepsilon} \ell_{q+2}^{- \varepsilon} \frac{B_{q+1}}{B_q} v_{q+2} \lesssim a_{q+1}^{\frac{\delta}{2 + \delta} - 4 \eps (1 + \delta)} A_{q+1} \,,
\end{equation}
Indeed, using the It\^o formula
with $f$ such that
$$ \Delta f = {u}_{q+2}^{(1)}, \qquad \int f = 0 \,,$$
we deduce that the quantity $ \left | \int_t^{{ 1}} {u}_{q+2}^{(1)} (X_{1,s}^{{ \kappa_q}}) ds \right | $ is equal to
\begin{align*}
\frac{1}{\kappa_q} \left | f(X_{1,t}^{{ \kappa_q}}) - f(x) + \int_{t}^{{ 1}} \nabla f (X_{1,s}^{{ \kappa_q}}) \cdot {u}_{q+2} (X_{1,s}^{{ \kappa_q}}) ds + \sqrt{2 \kappa_q} \int_t^{ 1} \nabla f (X_{1,s}^{{ \kappa_q}}) \cdot d W_s  \right |.
\end{align*}
Thanks to Property~\ref{item:ItemThreeInLemmaGronwallDissPeriod} in Lemma~\ref{lemma:gronwall+periodicity+dissipative}  and \eqref{eq:DefinitionOfU_qSeq} we have
$$ \| \nabla^k f \|_{L^\infty} \lesssim B_{q+1}^{2-k} \| u_{q+2}^{(1)} \|_{C^\varepsilon ({I_{{\sqrt{\kappa_q}}/{100}} (D_{q,R})}) }  \lesssim  \ell_{q+2}^{- \varepsilon} B_{q+1}^{2-k} v_{q+2} $$ for $k=0,1$, { where we used that the periodicity is bounded by $A_{q+1} + B_{q+1} \leq 2 B_{q+1}$ due to \eqref{eq:hierarchy-parameters}.}  Therefore, using \eqref{d:parameter:diffusive} and
$v_{q+2} \leq B_q$, for $Q$ sufficiently large, we have due to It\^o isometry 
\begin{align*} {\color{black}\int_{\Omega}} \left | \int_t^{{ 1}} {u}_{q+2}^{(1)} (X_{s,1}^{{ \kappa_q}}) ds \right | \, {\color{black}d \Prob} &\leq \frac{\| f \|_{L^\infty}}{\kappa_q} + \frac{ \| {u}_{q+2} \|_{L^\infty} \| \nabla f \|_{L^\infty} }{\kappa_q} + \frac{\| \nabla f \|_{L^\infty }}{\sqrt{\kappa_q}} \\
&\leq a_{q+1}^{- \varepsilon} \frac{B_{q+1}}{B_q} v_{q+2} {\color{black} \quad \forall t \in [0,1]}\,,
\end{align*}
where the norms are taken in the set ${I_{\frac{\sqrt{\kappa_q}}{100}} (D_{q,R})}$.
This proves \eqref{eq:itotanaka}.
Therefore, { if $Q$ is sufficiently large,} by \eqref{eq:itotanaka} and Markov's inequality we obtain 
$$ \mathbb{P} \left  ( \omega \in \Omega : \left | \int_t^{{ 1}} {u}_{q+2}^{(1)} (X_{{ 1},s}^{\color{black} \kappa_q} (x, \omega)) ds \right | \geq   A_{q+1} \right ) \leq a_{q+1}^{\frac{\delta}{2 + \delta} - 4  \eps(1 + \delta) }  \,, $$
and we denote this set $\Omega_{{u}_{q+2}, t}$.
 Using the previous property we conclude from the formula of the integral curves of $X_{{ 1},s}^{{ \kappa_q}}$ that 
 \begin{align*}
&  \left\{ \omega \in \Omega: X_{{ 1},s}^{{ \kappa_q}} (x, \omega ) \in \bigcup_{\overline{R} \in \mathcal{F}_q(R)} \overline{R} \right\} 
 \\
 & \qquad \subseteq \left\{ \omega \in (\Omega_{{u}_{q+2}, s})^c  :  \sqrt{2 \kappa_q} W_s^{(1)} \in \bigcup_{\overline{R} \in \mathcal{F}_q(R)} \pi_1 {  ( I_{A_{q+1}} ( \overline{R}  )) } \right\} \cup \Omega_{{u}_{q+2}, s} \,.
 \end{align*}

From this, it follows that
\begin{align*}
\int_0^{{ 1}} v_{q+1} \mathbb{P} & \left( \Omega_1 \cap \left\{ X_{{ 1},s}^{{ \kappa_q}} (x , \omega) \in \bigcup_{\overline{R} \in \mathcal{F}_q(R)} \overline{R} \right\} \right) ds   
\\
 &  \leq \int_0^{{ 1}} v_{q+1}  \mathbb{P} \left( \Omega_1 \cap \left\{ \sqrt{2 \kappa_q} W_s^{(1)} \in \bigcup_{\overline{R} \in \mathcal{F}_q(R)} \pi_1 {  ( I_{A_{q+1}} ( \overline{R}  )) } \right\} \right) ds 
  \\
  & \quad + v_{q+1} a_{q+1}^{\frac{\delta}{2 + \delta} - 4 \eps (1+ \delta)}\,,
\end{align*}
and we can bound, $v_{q+1} a_{q+1}^{\frac{\delta}{2 + \delta} - 4 \eps (1+ \delta)} \leq a_0 \sqrt{\kappa_q}$ thanks to \eqref{parameter} for any $q$ sufficiently large thanks to $\delta \gg \varepsilon$, i.e. \eqref{d:eps-delta-1}.
 
Using Item~\ref{item:AboutTheStructureOfHolderFieldsItemTwo} of Lemma~\ref{lemma:AboutTheStructureOfHolderFields} about the structure of the collection of rectangles $\mathcal{F}_q(R)$, we can estimate the first term on the right-hand side above as follows:
\begin{align} \label{eq:proof:estimate-insideset}
&\int_0^{{ 1}} v_{q+1}  \mathbb{P} \left( \Omega_1 \cap \left\{ \sqrt{2 \kappa_q} W_s^{(1)} \in \bigcup_{\overline{R} \in \mathcal{F}_q(R)} \pi_1 {  ( I_{A_{q+1}} ( \overline{R}  )) } \right\} \right) ds \notag \\
& \leq \int_0^{{ 1}} \sum_{j=-n_{q+1}}^{n_{q+1}} \int_{ j B_{q+1} +[ - 2 A_{q+1} , 2 A_{q+1}] \cap  |y | \leq \sqrt{2 \kappa_q}  } \frac{1}{\sqrt{4 \pi s \kappa_q}} \exp \left ( - \frac{|y|^2}{4 \kappa_q s} \right ) v_{q+1} dy ds \notag
\\
& \leq  \int_0^{{ 1}} \sum_{j=0}^{\frac{2 \sqrt{2 \kappa_q} }{B_{q+1}}} \int_{j B_{q+1} +[ - 2 A_{q+1} , 2 A_{q+1}]  } \frac{1}{\sqrt{4 \pi s \kappa_q}} \exp \left ( - \frac{|y|^2}{4 \kappa_q s} \right ) v_{q+1} dy ds \notag
\\
& \leq  \dfrac{ 2 \sqrt{2 \kappa_q} }{B_{q+1}} \frac{A_{q+1} v_{q+1}}{\sqrt{\kappa_q}} \leq a_0^{2 \eps} \sqrt{\kappa_q} \,,
\end{align}
where in the last we used the properties of the { parameters  \eqref{eq:par-initial}} and \eqref{parameter}. Then, by Markov for any $x \in D_q$ we bound
\begin{align*}
 \mathbb{P} & \left ( \omega \in \Omega_1 :  \int_{0}^{{ 1}} | u_{q+2}^{(2)} (X_{{ 1},s}^{{ \kappa_q}} (x, \omega)) | ds > a_0^{\eps} \sqrt{\kappa_q} \right ) 
 \\
 & \quad  \leq \frac{\int_{\Omega_1} \int_{0}^{{ 1}} | u_{q+2}^{(2)} (X_{{ 1},s}^{{ \kappa_q}} (x, \omega)) | ds }{a_0^{\eps} \sqrt{\kappa_q}} \leq a_0^{\eps} \,,
\end{align*}
where in the last inequality we used the  previous computation \eqref{eq:proof:estimate-insideset}. Therefore,
 there exists $\Omega_{q,1,x} \subset \Omega_1$ with $\mathbb{P} (\Omega_{q,1,x}) \geq \mathbb{P}(\Omega_1) - a_0^{\eps}$  such that
$$  \int_{0}^{{ 1}} | u_{q+2}^{(2)} (X_{{ 1},s}^{{ \kappa_q}} (x, \omega)) | ds \leq a_0^{\eps} \sqrt{\kappa_q} $$
for any  $\omega \in \Omega_{q,1,x}$.
Finally, since $|u_{q+2}^{(1)}(X_{1,s}^{\kappa_q})| \leq v_{q+2} \leq a_0^{1/2} \sqrt{\kappa_q}$ for all $s \in [\tau, 1]$, we find that for all $\omega \in \Omega_{q,1,x}$
\begin{align*}
|X_{{ 1}, \tau}^{{ \kappa_q}} (x, \omega) - x| & \leq \int_0^{{ 1}} |u_{q+2}(X_{{ 1}, s}^{{ \kappa_q}} (x, \omega))| ds  + \sqrt{2 \kappa_q} |W_1 - W_0| 
\\
& \leq  2 a_0^{\eps} \sqrt{\kappa_q} +  \frac{\sqrt{\kappa_q}}{200} < \frac{\sqrt{\kappa_q}}{100}
 \,, 
\end{align*}
  thanks to the fact that $a_0^{1/2} \leq \frac{1}{400}$, which implies $\tau (\omega) \equiv 0$ for any $\omega \in \Omega_{q,1,x}$. Therefore, we conclude the proof of  \eqref{eq:TrajectoriesRemainClose}.
 \\
\textbf{Part 2: (Proof of \eqref{eq:TrajectoriesGoFar})}
{
Fix some $(x,y) \in D_q$ arbitrary. Let $R \in \mathcal{G}_q$ be such that $(x,y) \in D_{q,R}$.
Recall that thanks to Item~\ref{item:ItemTwoInLemmaGronwallDissPeriod} in Lemma~\ref{lemma:AboutTheStructureOfHolderFields} the set $D_{q,R}$, in the coordinates of $R$, is given by (see Figure \ref{fig:IneqFigure})
\[
 D_{q,R} = \left[ { \frac{L_q}{9} }, \frac{L_q}{3} \right] \times \left[ - \frac{A_q}{2} - \frac{\sqrt{\kappa_q}}{50} , - \frac{A_q}{2} - \frac{\sqrt{\kappa_q}}{100} \right] \subseteq R\,.
\]
\begin{figure}[ht]
\begin{tikzpicture}[scale=0.6]
\small
\BMPipeWithRectanglesAndDissipativeSetTwo{0}{0}{0.8}{8}{4}{15}{3}
\draw[black, thick , dotted] (0,4.1) rectangle (15,-4.1);
 \draw[yellow, very thick] (- 0.07*15, 4.1) -- (0.38*15, 4.1) -- (0.38*15, - 4.1) -- (- 0.07*15, - 4.1);
 \draw[purple, very thick] (- 0.07*15, 4.1) -- (- 0.07*15, - 4.1);
 \draw[black, thick , <->, dashed] (- 0.07*15, 3) -- (0, 3);
 \draw[purple] (- 0.07*15, 2) node[anchor=east]{$H_{q, R}$};
 \draw[yellow] (0.03*15, 4.1) node[anchor=north]{$O_{q, R}$};
 \draw[black] (- 0.03*15, 3) node[anchor=north]{$h_{q, R}$};
\end{tikzpicture}
\centering
\caption{The rectangle in the coordinates $R = [0, L_q] \times [- \frac{A_q+B_q}{2}, \frac{A_q+B_q}{2}]$. The set $D_{q,R}$ in red, $[- h_{q,R} , 5L_q/12] \times [-A_q/4, A_q/4]$ in blue, $H_{q,R}$ in purple and $O_{q,R}$ in yellow.} \label{fig:IneqFigure}
\end{figure}
From now on, we will always work in the coordinates of $R$. 
Recall the definition of $H_{q, R}$ given in {~\eqref{eq:HittingSetsDef}}
that in the coordinates of $R$ can be written as
 $$ H_{q, R} =  \{ - h_{q,R} \} \times \left[ - \frac{A_q}{2} - \frac{B_q}{2}, \frac{A_q}{2} + \frac{B_q}{2} \right]$$
 where $h_{q,R} \in [0, a_q^{\eps \delta} B_q]$.
 The goal is to prove that 
 the stopping time
$$T_1 =  \sup \{ s \in ({\color{black} - \infty},{ 1}] : X_{{ 1},s}^{\kappa_q}(x,y,\omega) \in {\color{black} H_{q,R} } \} \vee 0 \,,$$ 
satisfies 
\begin{equation}\label{eq:ProbabilityStoppingTime}
\Prob \left [T_1 \geq 1/2 \right ] \geq c >0 \,.
\end{equation}
  Then, we apply  Proposition~\ref{lemma:stab-2} with the stopping time $T_1$ to 
  {\color{black} obtain}
$$\Prob (\omega: \dist ( X_{{ 1},0}^{\kappa_q} (x,y, \omega) , (x,y) ) >c  ) \geq c \,, $$
and we conclude the proof.
Note that the stopping time $T_1$, in the coordinates of $R$, can be written as
\[
 T_1(\omega) = \sup \left\{ t \in ({\color{black} - \infty} , { 1}] : X_{{ 1}, t}^{ \kappa_q} (x, \omega) \in   H_{q, R} \right\} \vee 0 \,.
\]
Additionally, we introduce the set $O_{q,R}$---{\color{black} consisting of three segments, which allows us to neglect the contribution of the velocity field outside the region confined in $O_{q,R} \cup H_{q,R}$, see Figure \ref{fig:IneqFigure}}---together with the two stopping times $\tau_{\mi}$ and $\tau_{\exit}$ defined by
\begin{align*}
O_{q, R} &= \left [-h_{q,R}, \frac{5L_q}{12} \right] \times \left\{ - \frac{A_q}{2} - \frac{B_q}{2} , \frac{A_q}{2} + \frac{B_q}{2} \right\} \\
&\quad \cup \left\{ \frac{5L_q}{12} \right\} \times \left[ - \frac{A_q}{2} - \frac{B_q}{2}, \frac{A_q}{2} + \frac{B_q}{2} \right]
\end{align*}
\begin{align*}
 \tau_{\mi} &= \sup \left\{ t \in {(\color{black}- \infty}, { 1}] : X_{{ 1}, t}^{ \kappa_q} (x, y, \omega) \in \left[ 0, \frac{L_q}{3} \right] \times \{ 0 \} \right\} {\color{black} \vee 0}; \\
  \tau_{\exit} &= \sup \left\{ t \in {(\color{black}- \infty}, {\color{black} 1}] : X_{{ 1}, t}^{ \kappa_q} (x, y, \omega) \in O_{q, R} \right\} {\color{black} \vee 0}.
\end{align*}
}
Firstly, we
{\color{black} obtain}
\begin{align*}
\Prob [ {T_1 } \geq \sfrac{{ 1}}{2}]
\geq 
 & \Prob \left[ {\tau_{\mi}} \geq \sfrac{3 }{4}, {\tau_{\exit}} \leq \sfrac{{ 1}}{2},  {T_1} \geq \sfrac{{ 1}}{2} \right] \,.
\end{align*}
Thanks to \eqref{d:parameter:diffusive} and \eqref{parameter} it holds that $\Prob \left[ {\tau_{\exit}} \leq \sfrac{{1}}{2} \right] \geq 1-2\exp(-a_0^{-\frac{\delta}{2}})$.
Then, thanks to the fact that the length of the $q$-th pipe is at most $L_q$ and $u_{q+2}^{(1)} \equiv v_q$ on $[0, 5 L_q/12] \times [- A_q/4, A_q/4]$ (see Figure \ref{fig:IneqFigure}), we have that  $\Prob [ { T_1}   \geq \sfrac{{ 1}}{2} \,, {\tau_{\mi}} \geq \sfrac{3 }{4} \,, {\tau_{\exit}}
 \leq \sfrac{{ 1}}{2} ] $ is greater or equal than
\begin{align*}
   \Prob \left[ \int_{{ 1}/2}^{3 /4} v_q \mathbbm{1}_{ [ X_{{ 1}, t}^{ { \kappa_q}}(x, \omega) \in [0, \sfrac{5 L_q}{12}] \times [- \sfrac{A_q}{{ 4}}, \sfrac{A_q}{{4}} ]]} \, dt \geq L_q \,, {\tau_{\mi}} \geq \sfrac{3}{4}, {\tau_{\exit}} \leq \sfrac{{ 1}}{2} \right] \,.
\end{align*}
Thanks to the fact that $u_{q+2}^{(2)} (X_{{ 1},t}^{\kappa_q})$ and $X_{{ 1},t }^{\kappa_q, (2)}$ have the same sign for $t \in [{ 1}/2,3/4 ]$ (see Figure~\ref{fig:IneqFigure}) we also have
\begin{align*}
&  \Prob  \left[ \int_{1/2}^{3 /4} v_q \mathbbm{1}_{ [ X_{{ 1}, t}^{ { \kappa_q}}(x, \omega) \in [0, \sfrac{5 L_q}{12}] \times [- \sfrac{A_q}{{4}}, \sfrac{A_q}{{4}} ]]} \, dt \geq L_q\,,    {\tau_{\mi}} \geq \sfrac{3}{4},   {\tau_{\exit}} \leq \sfrac{{ 1}}{2}  \right] 
\\
& \geq \Prob \left[ \int_{{{1}}/{2}}^{{3}/{4}} v_q \mathbbm{1}_{ [ \sqrt{2 \kappa_q} W_{t-\tau_{\mi}}^{{ (2)}} \in  [- \sfrac{A_q}{{4}}, \sfrac{A_q}{{4}} ] ] } \, dt \geq L_q \,,   {\tau_{\mi}} \geq \sfrac{3}{4},   {\tau_{\exit}} \leq \sfrac{{ 1}}{2}   \right]  \,.
\end{align*}
Hence, using that $\Prob \left[ \tau_{\mi} \geq \sfrac{3}{4} \right] \geq c > 0$ with a constant independent on $q$ and $a_0$, it is sufficient to prove
\begin{equation}\label{eq:BoundBelowOnProbTouchingPipe}
\Prob \left[  \int_{\sfrac{{ 1}}{2}}^{\sfrac{3 }{4}} v_q \mathbbm{1}_{ [ \sqrt{2 \kappa_q} W_t^{{ (2)}} \in  [- \sfrac{A_q}{{4}}, \sfrac{A_q}{{4}} ] ] } \, dt \geq L_q    \right] \geq 2 c^{-1} \exp(-a_0^{-\frac{\delta}{2}})
\end{equation}
By scaling properties of the Brownian motion and Theorem \ref{thm:ergodic}, we have
\begin{align}
\begin{split}\label{eq:SequenceOfIneqs}
 &\Prob \left[  \int_{0}^{1/4} v_q \mathbbm{1}_{ [ \sqrt{2 \kappa_q} W_t^{{ (2)}} \in  [- \sfrac{A_q}{{4}}, \sfrac{A_q}{{4}} ] ] } \, dt \geq L_q    \right] \\
 &\qquad \geq \Prob \left[ \int_{0}^{\sfrac{{ 1}}{4}} \mathbbm{1}_{ \left[ \sqrt{\frac{2 \kappa_q}{A_q^2}} W_t^{{ (2)}} \in  [- \sfrac{1}{{4}}, \sfrac{1}{{4}} ] \right] } \, dt \geq \dfrac{L_q}{v_q} \right] \\
 &\qquad = \Prob \left[ \int_{0}^{\frac{2 \kappa_q}{4 A_q^2}  } \mathbbm{1}_{ [ W_{s}^{{ (2)}} \in  [- \sfrac{1}{{4}}, \sfrac{1}{{4}} ] ] } \, ds \geq \dfrac{2 \kappa_q L_q}{A_q^2 v_q} \right] 
 \\
 &\qquad \geq \exp(-a_0^{\frac{\delta}{2} - \frac{2 \delta}{2 + \delta}}) \\
 \end{split}
\end{align}
where for the last inequality we used that 
$\frac{2 \kappa_q}{4 A_q^2} \approx a_q^{-2 \eps }$ and $\dfrac{2 \kappa_q L_q}{A_q^2 v_q} \approx a_{q}^{- \eps}$ up to constants depending only on $a_0$ and 
{ $\frac{L_q}{v_q} = a_q^\varepsilon a_0^{1- \frac{\delta}{2 + \delta} - \varepsilon} \leq a_q^\varepsilon a_0^{1/2}$ thanks to \eqref{eq:par-initial} and \eqref{parameter}.  }
Hence, \eqref{eq:BoundBelowOnProbTouchingPipe} holds and provided $a_0$ is small enough, \eqref{eq:ProbabilityStoppingTime} holds true and we conclude the proof.
\color{black}

\section{Proof of Theorem \ref{thm:NS} about the 3D forced Navier--Stokes equations}\label{sec:NS}
 We use the trick of the $(2+ \frac{1}{2})$-dimensional Navier--Stokes equations that has been already  exploited in other papers, see for instance \cite{JY20,JYo20,BDL22}. More precisely, we suppose that all the functions in the system are independent on the last variable $x_3$ where $x = (x_1, x_2, x_3) \in \T^3$. In this case the first two equations decouple with the last one, i.e.  given the initial data $u_{\initial, \kappa} : \T^2 \to \R^2$ and $\theta_{\initial, \kappa} : \T^2 \to \R$ and the force $f_\nu : \T^2 \to \R^2 $, if $(u_\nu, p_\nu , \theta_\nu)$ is a solution of 
\begin{align} \label{NS-2+1/2}
\begin{cases} 
\partial_t u_\nu + u_\nu \cdot \nabla u_\nu + \nabla p_\nu = \nu \Delta  u_\nu + f_{\nu}  \,,
\\
\diver (u_\nu) =0 \,,
\\
u (0, \cdot ) = u_{\initial, \nu } (\cdot) \,,
\\
\partial_t \theta_\nu +  u_\nu \cdot \nabla \theta_\nu = \nu \Delta \theta_\nu \,,
\\
\theta_{\nu} (0, \cdot) = \theta_{\initial, \nu} (\cdot) \,,
\end{cases}
\end{align}
where $u_\nu : [0,1 ] \times \T^2 \to \R^2 $ and $p_\nu : \T^2 \to \R$ and $\theta_\nu : [0,1] \times \T^2 \to \R$,
then 
$v_\nu (t, x_1, x_2, x_3) = (u_\nu (t, x_1, x_2), \theta_\nu (t, x_1, x_2))$ and $ P_\nu (t, x_1, x_2, x_3) = p_\nu (t, x_1, x_2)$ is a solution of \eqref{NS} with initial datum $v_{\initial, \nu} (x_1, x_2, x_3)= (u_{\initial, \nu} (x_1, x_2) , \theta_{\initial, \nu} (x_1, x_2))$ and force $F_\nu (x_1, x_2, x_3) = f_\nu (x_1, x_2)$. 

\begin{proof}[Proof of Theorem \ref{thm:NS}]

We aim at finding a solution of \eqref{NS-2+1/2} which is independent on $x_3$.
We choose $\nu_q = \kappa_q$, $u_{\initial, \nu_q} (x_1, x_2)= u_{q+2} $, where $u_{q+2}$  is the autonomous velocity field defined  in \eqref{eq:DefinitionOfU_qSeq}, 
$F_{\nu_q} = \mathcal{L} ( u_{q+2} \cdot \nabla u_{q+2} ) - \nu_q \Delta u_{q+2}$, where $\mathcal{L}$ is the Leray projector  into divergence-free velocity fields and $\theta_{\initial, \nu } = \theta_{\initial}$ defined in Proposition \ref{prop:criterion}.  Let $\theta_{\kappa_q} : [0,1] \times \T^2 \to \R$ be the solution to 
\begin{align*}
\begin{cases}
\partial_t \theta_{\kappa_q} + u_{q+2} \cdot \nabla \theta_{\kappa_q} = \kappa_q \Delta \theta_{\kappa_q}
\\
\theta_{\kappa_q} (0, \cdot) = \theta_{\initial } (\cdot) \,.
\end{cases}
\end{align*}
and $p_{q+2} : \T^2 \to \R$ be  the zero average solution to 
$$ \Delta p_{q+2} = - \diver \diver (u_{q+2} \otimes u_{q+2}) \,.$$

It is straightforward to check that $(u_\nu, p_\nu, \theta_\nu)= (u_{q+2}, p_{q+2}, \theta_{\kappa_q})$
is the unique solution to \eqref{NS-2+1/2} with the given initial data and force defined above (recall that  $u_{q+2}$  is the autonomous velocity field defined  in \eqref{eq:DefinitionOfU_qSeq}).
  Finally, for $v_{\nu_q} = (u_{\nu_q}, \theta_{\nu_q})$ we have
 \begin{align*}
  \limsup_{\nu_q \to 0} \nu_q \int_0^1 \int_{\T^3} | \nabla v_{\nu_q} (t, x)|^2 dx dt & \geq  \limsup_{\kappa_q \to 0} \kappa_q \int_0^1 \int_{\T^3} | \nabla \theta_{\kappa_q} (t, x)|^2 dx dt >0 \,, 
\end{align*} 
where for the last we used Proposition \ref{prop:criterion}, Proposition  and Lemma \ref{lemma:energy} together with $\frac{\| u - u_{q+2} \|_{L^2}}{\sqrt{\kappa_q}} \to 0 $ as $\kappa_q \to 0$ thanks to \eqref{eq:estimate:u-q+2}.
Thanks to Lemma \ref{lemma:NS-Calpha} it is straightforward to check that 
$ \| F_{\nu_q} -  F_0 \|_{C^\alpha} + \| v_{\initial, \nu_q} - v_{\initial} \|_{C^\alpha} \to 0$ as $\nu_q \to 0$ for some force $F_0 \in C^\alpha$ and initial datum $v_{\initial } \in C^\alpha$.

We now prove the last remaining property. Thanks to the fact that $u_{q+2} \to u $ in $C^\alpha$ and  
$$\sup_{\nu_q } \|  \theta_{\nu_q}  \|_{L^\infty} \leq  \| \theta_{\initial} \|_{L^\infty} \leq C$$
then there exists a converging subsequence of $\{ v_{\nu_q} = (u_{q+2}, \theta_{\nu_q}) \}_{q}$ in the weak*-$L^\infty$ topology and $v_0 = (u, \theta_0)$ where $\theta_{\nu_q} \overset{*-L^\infty}{\rightharpoonup} \theta_0$ up to subsequences.
Thanks to the fact that all the functions are independent on $x_3$, it is straightforward to check that $v_0$ is a solution to the $3D$ forced Euler equations with initial datum $v_{\initial} \in C^\alpha$ and force $F_0 \in C^\alpha$. The Lipschitz property of the energy of $v_0$ follows from Proposition \ref{prop:absolute} since $u$ is independent of time. This concludes the proof.
\end{proof}

\bibliographystyle{plain}
\bibliography{biblio}
 
\end{document}